\documentclass[letterpaper,10pt]{article}

\usepackage{amsmath,amsthm}
\usepackage{mathrsfs}
\usepackage{tikz-cd}
\usepackage{amssymb}
\usepackage{graphicx}
\usepackage{subfigure}
\usepackage{subfloat}
\usepackage[abs]{overpic}
\usepackage{framed}
\usepackage{enumerate}
\usepackage[left=1.2in,top=1in,right=1.2in,bottom=1in,letterpaper]{geometry}
\usepackage[backref]{hyperref}
\hypersetup{
  colorlinks   = true,    
  urlcolor     = blue,    
  linkcolor    = blue,    
  citecolor    = red,     
}

\numberwithin{equation}{section}
\allowdisplaybreaks
\DeclareMathOperator*{\argmin}{arg\,min}

\definecolor{DukeBlue}{HTML}{001A57}

\newtheorem{theorem}{Theorem}
\newtheorem{lemma}[theorem]{Lemma}
\newtheorem{corollary}[theorem]{Corollary}
\newtheorem{remark}{Remark}[section]
\newtheorem{definition}{Definition}[section]
\newtheorem{example}[theorem]{Example}
\newtheorem{assumption}[theorem]{Assumption}

\begin{document}

\title{The Diffusion Geometry of Fibre Bundles: Horizontal Diffusion Maps}
\author{Tingran Gao\thanks{Department of Statistics and Committee on Computational and Applied Mathematics (CCAM), The University of Chicago (tingrangao@galton.uchicago.edu)}}
\date{}
\maketitle

\begin{abstract}
Kernel-based non-linear dimensionality reduction methods, such as Local Linear Embedding (LLE) and Laplacian Eigenmaps, rely heavily upon pairwise distances or similarity scores, with which one can construct and study a weighted graph associated with the dataset. When each individual data object carries additional structural details, however, the correspondence relations between these structures provide extra information that can be leveraged for studying the dataset using the graph. Based on this observation, we generalize \emph{Diffusion Maps} (DM) in manifold learning and introduce the framework of \emph{Horizontal Diffusion Maps} (HDM). We model a dataset with pairwise structural correspondences as a \emph{fibre bundle} equipped with a \emph{connection}. We demonstrate the advantage of incorporating such additional information and study the asymptotic behavior of HDM on general fibre bundles. In a broader context, HDM reveals the sub-Riemannian structure of high-dimensional datasets, and provides a nonparametric learning framework for datasets with structural correspondences.
\end{abstract}


\tableofcontents

\section{Introduction}
\label{sec:introduction}
Acquiring complex, massive, and often high-dimensional data sets has become a common practice in many fields of science. While inspiring and stimulating, these data sets can be challenging to analyze or understand efficiently. To gain insight despite the volume and dimension of the data, methods from a wide range of science fields have been brought into the picture, rooted in statistical inference, machine learning, signal processing, to mention just a few. Among the exploding research interests and directions in data science, the relation between the graph Laplacian~\cite{Chung1997} and the manifold Laplacian~\cite{Rosenberg1997Laplacian} has emerged as a useful guiding principle. Specifically, the field of \emph{non-linear dimensionality reduction} has witnessed the emergence of a variety of kernel-based spectral techniques, such as Locally Linear Embedding (LLE)~\cite{LLE2000}, ISOMAP~\cite{ISOMAP2000}, Hessian Eigenmaps~\cite{HessianLLE2003}, Local Tangent Space Alignment (LTSA)~\cite{LTSA2005}, Diffusion Maps~\cite{CoifmanLafon2006}, Orientable Diffusion Maps (ODM)~\cite{SingerWu2011ODM}, Vector Diffusion Maps (VDM)~\cite{SingerWu2012VDM}, and Schr\"odinger Eigenmaps~\cite{SSSE2014}. The general practice of these methods is to treat each object (images, texts, shapes, etc.) in the data set as a vertex of a graph, and two ``similar'' vertices are connected through an edge weighted by their similarity score. The graph is then embedded into a Euclidean space of relatively low dimensionality using the eigenvectors of the graph Laplacian (or its variant) associated with the similarity graph. Built with varying flexibility, these methods provide valuable tools for organizing complex networks and data sets by ``learning'' the global geometry from the local connectivity and weights.

In reality, graph-based data analysis is known to fall short of their expressiveness in capturing multiplex, heterogeneous, and time-varying pairwise relations commonly encountered in data science problems. Social network analysis has long been aware of the importance of preserving the ``additional information,'' such as structural, compositional, and affiliation \emph{attributes}, for avoiding potential loss of accuracy due to the over-simplified abstraction of complex social relations into simple nodes and edges in graph models \cite{Goffman1974,WF1994,BMBL2009,BCMM2015}. Recent technological advancement has also fostered an increasing trend of extending the graph-based analysis to networks of multiple types of connections, or \emph{networks of networks} \cite{DAS2014,KPB2015}, that encode multi-modal pairwise relations as \emph{multilayer} complex systems supported on a set of shared vertices \cite{BBCDG+2014,KABGMP2014,Bianconi2019}. These new developments essentially follow the same methodology of enriching the graph representation with structures beyond simple vertices/edges and scalar weights on them.

We propose in this paper \emph{Horizontal Diffusion Maps} (HDM), a novel graph-based framework for analyzing complex data sets with non-scalar or functional pairwise relations, with a focus on data sets in which similarity scores between samples can be obtained from ``correspondence relations'' between sophisticated \emph{individual structures} carried within each sample. We distinguish \emph{data objects}, which constitute the vertices of the graph, from the \emph{data points} sampled from each data object that represent the internal structure of the data object. Just like manifold learning assumes that data lie approximately on a smooth manifold, we view the data objects as approximately sampled from a smooth \emph{base} manifold, and the data points as samples on the \emph{fibres} of a fibre bundle over the base manifold; data points on the same data object are assumed to come from the same fibre. One such example is the biological shape data in geometric morphometrics (see Figure~\ref{fig:correspondence} and Section~\ref{sec:autogm}), where each individual shape is a data object and each point on the shape in a data object; similar examples can be found e.g. in image analysis, where images are data objects and pixels on each image are data points.  In many of these instances, the data acquired is too noisy, has huge degrees of freedom, or contains un-ordered (as opposed to sequential) features. Computing pairwise similarity between data objects typically requires optimizing some functional over the space of admissible \emph{pairwise structural correspondences}, and the ``optimal correspondence'' is used to assign a distance or similarity score between the two data objects under comparison. Figure~\ref{fig:correspondence} illustrates two objects from a data set of anatomical surfaces, discretized as triangular meshes; an ``optimal correspondence'' between the pair is a diffeomorphism between the two meshes that minimizes an energy functional whose minimum defines a distance between disk-type surfaces. Often the optimal correspondence encodes substantial information missing from the distance, which is merely a scalar condensed from the diffeomorphism. The HDM framework aims to mine this hidden information from pairwise structural correspondences. For a data set consisting of data objects, data points, and pairwise structural correspondences, horizontal diffusion maps provide a two-level data representation that first ``synchronizes'' the data objects with respect to ``denoised'' structure correspondences by embedding the data points into a Euclidean space, and then, building on top of the first-level embedding for the data points, embed the data objects into another Euclidean space as the second level. As the second-level embedding for the data objects leverages the rich structural information at the level of data points, they are expected to be semantically more meaningful than the spectral representation obtained from standard diffusion maps which can not take advantage of the individual structural information; the synchronized spectral representation of the data points at the first level also adds to the interpretative power of HDM, enabling detailed domain-specific analysis for the data objects that is often beyond the scope of standard diffusion maps.


\begin{figure}[htp]
  \centering  \includegraphics[width=0.3\textwidth]{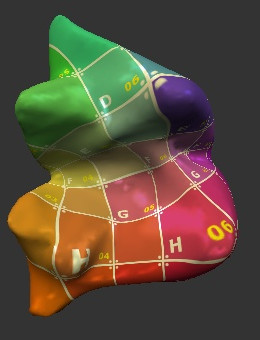}\qquad\qquad
  \includegraphics[width=0.3\textwidth]{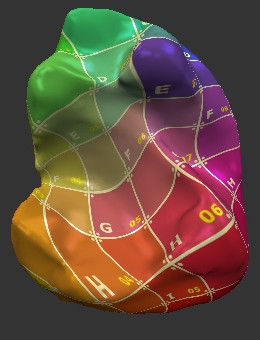}
  \caption{\emph{An optimal structural correspondence between two lemur teeth illustrated by pushing forward a texture on the left tooth onto the right tooth. This correspondence leads to the \emph{Continuous Procrustes Distance}~\cite{CP13} between shape pairs. HDM utilizes the abundant geometric information in such correspondences.}}
  \label{fig:correspondence}\vspace{-0.1in}
\end{figure}

In the remainder of this section we relate HDM to other recent work in diffusion geometry, summarize our main theoretical contribution, and then describe the organization of the paper.

\subsection{Related Work}
\label{sec:related-work}

The Diffusion Map (DM) framework~\cite{CoifmanLafon2006,LafonThesis2004,CoifmanLafonLMNWZ2005PNAS1,CoifmanLafonLMNWZ2005PNAS2,CoifmanMaggioni2006,SingerWu2011ODM,SingerWu2012VDM} proposes a probabilistic interpretation for graph-Laplacian-based dimensionality reduction algorithms. Under the assumption that the discrete graph is appropriately sampled from a smooth manifold, it assigns transition probabilities from a vertex to each of its neighbors (vertices connected to it) according to the edge weights, thus defining a graph random walk the continuous limit of which is a diffusion process~\cite{WatanabeIkeda1981SDE,Durrett1996SC} over the underlying manifold. The eigenvalues and eigenvectors of the graph Laplacian, which converge to those of the manifold Laplacian under appropriate assumptions~\cite{BelkinNiyogi2005,BelkinNiyogi2007}, then reveal intrinsic information about the smooth manifold. More precisely, ~\cite{BBG1994} proves that these eigenvectors embed the manifold into an infinite dimensional $\ell^2$ space, in such a way that the $\ell^2$ distance between embedded points equals to the \emph{diffusion distance}~\cite{CoifmanLafon2006} between the sample points on the manifold. Appropriate truncation of these sequences leads to an embedding of the smooth manifold into a finite dimensional Euclidean space, with small metric distortion.

Under the manifold assumption, ~\cite{SingerWu2011ODM,SingerWu2012VDM} recently observed that estimating random walks and diffusion processes on structures associated with the original manifold (as opposed to estimates of diffusion on the manifold itself) are able to handle a wider range of tasks, or obtain improved precision or robustness for tasks considered earlier. For instance, ~\cite{SingerWu2011ODM} constructed a random walk on the \emph{orientation bundle}~\cite[\S I.7]{BottTu1982} associated with the manifold, and translated the detection of orientability into an eigenvector problem, the solution of which reveals the existence of a global section on the orientation bundle; ~\cite{SingerWu2012VDM} introduced a random walk on the \emph{tangent bundle} associated with the manifold, and proposed an algorithm that embeds the manifold into an $l^2$ space using eigen-vector-fields instead of eigenvectors (and thus the name Vector Diffusion Maps (VDM)). Both ~\cite{SingerWu2011ODM} and ~\cite{SingerWu2012VDM} incorporate additional structures into the graph Laplacian framework: in ~\cite{SingerWu2012VDM} this is an extra orthogonal transformation (estimated from local tangent planes) attached to each weighted edge in the graph; in ~\cite{SingerWu2011ODM} the edge weights are overwritten with signs determined by this orthogonal transformation. These methods are successful, partly because they incorporate more local geometry (by estimating tangent planes) \emph{en route} to dimensionality reduction. In~\cite{Wu2013VDMEmbedding} the VDM approach is used, analogously to ~\cite{BBG1994}, to embed the manifold into a finite dimensional Euclidean space. Although the VDM embedding does not reduce the dimensionality as much as standard diffusion embedding methods, it benefits from improved robustness to noise, as illustrated by the analysis of some notoriously noisy data sets~\cite{KarouiWu2015,KW2016}.

This paper stems from the observation that it is possible to adopt the methodology of \cite{SingerWu2011ODM,SingerWu2012VDM} to tackle problems in much broader contexts, where the local geometric information can be of a different type than tangent spaces. For instance, many data sets carries abundant structural details on each individual object in the data set, such as pixels in an image, vertices/faces on a triangular mesh, or a collection of persistent diagrams~\cite{TMB2014} representing a shape. Typically, kernel eigenmap methods begin by ``abstracting away'' these details, encoding only pairwise similarites using a kernel function. The major advantage, like kernel methods in general, is the flexibility (no need to extract explicit features) and efficiency (most kernels are easy to compute); however, in some circumstances, the structural details may themselves be of interest. For example, in the geometry processing problem of analyzing large collections of 3D shapes, it is desirable to enable user exploration of shape variations across the collection, for which reducing each individual shape as a graph vertex completely ignores its spatial configuration. Furthermore, even when sticking to pairwise similarity scores significantly simplifies the data manipulation, the best way to score similarity (or to craft the kernel function) is not always clear. In practice, the similarity measure is often dictated by practical heuristics, which may be misguided for incompletely understood data.

Like ODM and VDM, HDM extends the diffusion map framework, but takes an essentially different path. In this paper, we are most interested in the scenario in which the sample points are themselves manifolds; the entire data set is thus modeled as a ``manifold of manifolds.'' To provide a mathematical model for such consideration, we first augment the manifold underlying diffusion maps, denoted as $M$, with extra dimensions. To each point $x$ on $M$, this augmentation attaches an individual manifold, denoted as $F_x$; since pairwise correspondences exist between nearby individual manifolds, we assume that around each $x\in M$ there exists an open neighborhood $U$ such that on $U$ the augmented structure ``looks like'' $U\times F$, the product of $U$ with a ``universal template'' manifold $F$. Intuitively, $M$ plays the role of a ``parametrization'' for the collection of individual manifolds $\left\{F_x\mid x\in M\right\}$. Of course, the existence of such a universal template makes sense only if the $F_x$'s are compatible with each other in some appropriate sense (e.g. each $F_x$ should at least be diffeomophic to $F$); however, such compatibility is not uncommon for many data sets of interest, as we shall see in Section \ref{sec:formulation}. This picture of parametrizing a family of manifolds with an underlying manifold is reminiscent of the modern differential geometric concept of a \emph{fibre bundle}, which played an important role in the development of geometry, topology, and mathematical physics in the past century. Therefore, we shall refer to this geometric object as the underlying \emph{fibre bundle} of the data set. Adopting the terminology from differential geometry, we call $M$ the \emph{base manifold}, the universal template manifold $F$ the \emph{fibre}, and each $F_x$ a \emph{fibre at $x$}. The fibre bundle is itself a manifold, denoted as $E$ and referred to as the \emph{total manifold}. We emphasize here that the fibre bundle setting we consider in this paper is even more general and flexible than the principal bundle formulation in \cite{SW2016}, which provided a unified theoretical framework for diffusion maps and its various extensions. Whereas the principal bundle framework \cite{SW2016} builds upon an explicitly specified Lie group and defines the fibre bundle as a quotient space of the group action, in the framework of HDM the fibre bundles are trivialized by local parallel-transports. This flexibility allows us to analyze data sets satisfying the \emph{fibre bundle assumption} (see Section~\ref{sec:formulation}) but for which the structure group can not be identified \emph{a priori}. We shall elaborate on this in greater detail in Section~\ref{sec:formulation}.

A different line of research closely related to our work is the construction of adaptive \emph{cone kernels} \cite{Diannakis2015,ZG2016} in the data-driven study of dynamical systems. Unlike the geometric setting in our work (or \cite{SingerWu2012VDM,SW2016}), the low-dimensional manifold structure lives in the phase space, and the kernels are constructed from finite differences of time-ordered data samples. In \cite{Diannakis2015}, the author constructed a family of nonhomogeneous and anisotropic family of kernels that assign higher affinity scores to more aligned velocity vectors; the resulting diffusion processes generate paths that asymptotically ``follow along'' the integral curves of the dynamical vector field. The intimate connection between the intrinsic geometry of the data and general nonhomogeneous, anisotropic kernels is characterized in great detail in \cite{BT2016}. The usage of these more general and flexible kernels is similar in spirit to our construction of the \emph{coupled diffusion operator} in Section~\ref{sec:horiz-rand-walk} in the specific case when the Riemannian metric on the fibre bundle splits into the direct sum of horizontal and vertical components; however, it is worth pointing out that the lack of a fibre bundle structure in \cite{CoifmanLafon2006,Diannakis2015,ZG2016} makes these applications of anisotropic diffusions drastically different from HDM: in our terminology, these constructions are targeted at understanding the total manifold, whereas our goal is to extract information jointly and consistently from the total manifold and the base manifold. Specifically, our definitions of \emph{horizontal base diffusion map} (HBDM) and \emph{horizontal base diffusion distance} (HBDD) in Section~\ref{sec:spectr-dist-embedd} are meaningful only at the presence of an underlying fibre bundle structure. Most strikingly, as we point out in Remark~\ref{rem:hdm-nontrivial}, the HDM framework differs in an essential way from directly applying an anisotropic diffusion kernel construction to the total manifold of the fibre bundle; the two constructions coincide only in the very special case when the fibres are totally geodesically embedded into the total manifold. These subtle phenomena are characterized for the first time in the diffusion geometry literature. We thus believe that the classical differential geomtric concepts of fibre bundles, Riemannian submersions, and horizontal/vertical Laplacians, though introduced into the blossoming field of geometric data analysis only for the first time, open new opportunities for gaining deeper understandings of real world data through the lens of diffusion geometry.

\subsection{Main Contribution}
\label{sec:main-contribution}

The main theoretical contribution of this paper is to provide a probabilistic interpretation of HDM as a \emph{horizontal random walk} on the fibre bundle, extending the random walk picture of diffusion maps to a broader class of geometric objects. In one step, the transition occurs either between points on adjacent but distinct fibres, or within the same fibre. If transitions between distinct fibres depend solely on geometric proximity specified through a metric on the total manifold $E$, this looks no different from a direct application of diffusion maps on $E$. In contrast, HDM also incorporates the pairwise correspondences between individual manifolds in the fibre bundle formulation, by requiring transitions between distinct fibres to follow certain directional constraints imposed by correspondences. The resulting random walk is no longer a standard random walk on the total manifold, but rather a ``horizontal lift'' of a random walk on the base manifold $M$. Under mild assumptions, its continuous limit is a diffusion process on the total manifold $E$, infinitesimally generated by a \emph{hypoelliptic differential operator}~\cite{Hoermander1967}. We can then map the total manifold into a Euclidean space using the eigenfunctions of this partial differential operator; discretely this corresponds to solving for the eigenvectors of \emph{graph horizontal Laplacians}. It turns out that, by varying a couple of parameters in its construction, the family of graph horizontal Laplacians includes the discrete analogue of several important and informative partial differential operators on the fibre bundle, relating the geometry of the base manifold with that of the total manifold. Compared with \cite{SingerWu2012VDM,SW2016}, the limiting differential operators can be employed to reveal the sub-Riemannian structures of a fibre bundle (or \emph{Riemannian submersion} \cite[Chapter 9]{Besse2007EinsteinManifolds}), a task that can not be accomplished in the principal bundle framework of \cite{SingerWu2012VDM,SW2016}. Our numerical experiments revealed intriguing geometric phenomena, such as \emph{adiabatic limits}, when embedding the fibre bundle using eigenvectors of these new graph Laplacians; these phenomena have never been reported in any related work within the framework of \cite{SingerWu2012VDM,SW2016}.

We note that the idea of studying diffusion processes and random walks on an ``augmentation'' of the original data set, or extracting information from pairwise structural correspondences between sample points, has appeared elsewhere as well, in several distinct fields (e.g. shape collection analysis~\cite{Kim12FuzzyCorr}, manifold alignment~\cite{WangMahadevan2009}, and neurogeometry~\cite{Boscain2014}). To our knowledge, HDM is the first theoretical framework that provides the mathematical and statistical foundation for these research directions; in particular, like diffusion maps, HDM enables decoupling the probabilistic treatment of sampling from the geometry of the data set.

The rest of this paper is organized as follows:  Section~\ref{sec:formulation} formulates the problem and discusses the \emph{fibre bundle assumption}; Section~\ref{sec:algorithm} describes the algorithmic construction; Section~\ref{sec:infin-gener} contains the main technical results of this paper, several explicit calculations on some concrete examples of fibre bundles with totally geodesic fibres, along with a numerical example on $\mathrm{SO(3)}$ to validate the theoretical findings; finite sampling results and applications to biological shape analysis problems will be pursued in Section~\ref{sec:finite-sampl-results} and Section~\ref{sec:autogm}, respectively; Section~\ref{sec:conlusion} concludes with a brief discussion and propose potentially interesting directions for future work. The differential geometry concepts essential for developing the theoretical framework, as well as technical proofs of the main results, are postponed to the appendices.


\section{Horizontal Diffusion Maps on Fibre Bundles}
\label{sec:formulation}
In this section, we build the theoretical framework of horizontal diffusion maps, and relate it, where appropriate, with practical considerations for data processing.

\subsection{The Fibre Bundle Assumption}
\label{sec:fibre-bundles}

We say that the data set consists of \emph{data objects}, and each data object contains \emph{data points} (note that the number of data points contained in each data object may vary). Pairwise structural correspondences exist between data objects with high similarity scores; each correspondence is defined from a \emph{source} data object (the collection of source data points) to a \emph{target} data object (the collection of target data points), and can either be a point-to-point map or a ``multi-valued map'' that associates a source data point with multiple target data points. In the latter case, the correspondence may also assign similarity scores between source and target data points. To put data objects, data points, and pairwise structure correspondences in a unified geometric model, we resort to the following general definition of \emph{fibre bundles}.

\begin{definition} [Fibre Bundle, \cite{BGV2003,Michor2008}]
  \label{defn:fibre-bundle}
Let $\pi:E\rightarrow M$ be a smooth map from a \emph{total manifold} $E$ to a \emph{base manifold} $M$. We call the quadruple $\mathscr{E}=\left(E,M,F,\pi\right)$ a \emph{fibre bundle} with \emph{fibre manifold} $F$ if there is an open cover $\left\{U_i\right\}$ of $M$ with diffeomorphisms
\begin{equation*}
\phi_i:\pi^{-1}\left(U_i\right)\longrightarrow U_i\times F
\end{equation*}
such that $\pi:\pi^{-1}\left(U_i\right)\rightarrow U_i$ is the composition of $\phi_i$ with projection onto the first factor $U_i$ in $U_i\times F$. In other words, the following diagram is commutative:
\begin{center}
\begin{tikzcd}[column sep=small]
\pi^{-1}\left( U_i \right) \arrow{rr}{\phi_i} \arrow[swap]{dr}{\pi}& &U_{i}\times F \arrow{dl}{\mathrm{Proj}_1}\\
& U_i & 
\end{tikzcd} 
\end{center}
\end{definition}

It follows immediately from this definition that $\pi^{-1}\left( x \right)$ is diffeomorphic to $F$ for any $x\in M$. We denote $F_x$ for $\pi^{-1}\left( x \right)$ and call it the \emph{fibre over $x\in M$}. The diffeomorphism $\phi_i:\pi^{-1}\left( U_i \right)\rightarrow U_i\times F$ is also known as a \emph{local trivialization} of the fibre bundle $\mathscr{E}$ over the open set $U_i$. Unless otherwise stated, we assume throughout this paper that $M$ and $F$ are orientable Riemannian manifolds so the volume form and integration are well-defined; the dimensions of $M,F$ will be denoted as $d=\mathrm{dim}\left( M \right)$, $n=\mathrm{dim}\left( F \right)$, respectively. Using the language of fibre bundles, our basic assumptions for the data set can be summarized as follows:
\begin{enumerate}
\item Data points lie approximately on a fibre bundle;
\item Data points on the same data object are sampled from the same fibre.
\end{enumerate}
As stated above, the data sets of interest, to which the fibre bundle assumption applies, are those with pairwise correspondences between data objects, or fibres in the fibre bundle. This additional piece of information can now be easily incorporated into the fibre bundle framework: we interpret pairwise correspondences as \emph{parallel-transports} along \emph{geodesics} on the base manifold $M$, generated by a \emph{connection} (see \ref{sec:connections}) on the fibre bundle $\mathscr{E}$. For our purposes, the base manifold $M$ plays the same role as the manifold that underlies the diffusion maps (i.e., from which data objects are drawn); additionally, we assume that each data object $x\in M$ carries a manifold structure that is diffeomorphic to a fixed fibre manifold $F$; the entire data set can thus be interpreted as a collection of instantiations of the fibre $F$ (which can be viewed as a ``template''), indexed by points on the base manifold $M$ as $\left\{F_x:x\in M\right\}$. From a fibre bundle point of view, it is natural to study the base manifold $M$ using the extra information in the total manifold $E=\cup_{x\in M} F_x$. In the remaining paper, unless otherwise specified, we assume all Riemannian manifolds are geodesically complete.



Roughly speaking, a data set satisfies the fibre bundle assumption if the data generation process can be viewed as first drawing fibres from the fibre bundle (equivalent to sampling on the base manifold) and then sampling on each fibre. The fibre bundle assumption admits ``inconsistency'' of pairwise correspondences as to the nature of the underlying geometry: though pairwise correspondences only exist (or are of high fidelity) between nearby data objects, by knitting together these correspondences along ``small hops'' one can still build correspondences between far-apart data objects (provided the base manifold is connected); correspondences constructed in this manner are generally inconsistent with each other in the sense that knitting together correspondences along different paths connecting the same data objects leads to different correspondences. In our framework, this inconsistency would reflect the \emph{curvature} and \emph{holonomy} of the \emph{connection} on the fibre bundle; see \ref{sec:connections}.

The concept of fibre bundles we chose to present above is but one of several equivalent definitions; some other popular ones can be found, for instance, in \cite[Chapter 3, Chapter 10]{Taubes2011DG}. Our choice is based not only on the conciseness and flexibility of Definition~\ref{defn:fibre-bundle}, but also---most importantly---because there is no need to explicitly specify a \emph{structure group}. In stark contrast is the equivalent definition of principal and associated fibre bundles, e.g. in \cite[Appendix A]{SW2016}, in which principal bundles are defined as orbit spaces of Lie group actions, and an associated bundle is obtained from a principal bundle through representations of the Lie group. The unification of all diffusion maps and variants in \cite{SW2016} is made possible by specifying the structure groups explicitly for each particular type of diffusion maps. Nevertheless, in most practical applications of interest to us, it is difficult to explicitly know the structure group of the fibre bundle underlying the data set. For instance, as briefly surveyed in \cite[\S4.1]{GBM2016}, for some data sets it may be unrealistic to model the correspondence relations between data objects as group elements; \emph{groupoids} seem to be the more natural abstraction in those settings. Similar consideration motivated topological data analysts to propose \emph{sheaves} as data models; see e.g. \cite{HG2018} and the references therein. Even in cases in which the pairwise correspondences can be modeled as group elements, the group can be too large to manipulate efficiently, such as Lie groups of diffeomorphisms or isometries commonly encountered in non-isometric collection shape analysis \cite{BBK2008,HuangZhangGHBG2012,HuangGuibas2013,LZ2017}. While it is not uncommon to perform \emph{reductions} of principal bundles to reduce the structure group to smaller subgroups whenever possible, in the discrete setting this often boils down to the difficult group theoretic and combinatorial problem of understanding the rigidity or approximability of representations of discrete lattices of Lie groups \cite{GKR1974,Kazhdan1982,dCGLT2017}. These difficulties motivated us to take an alternative path to viewing the data sets we encountered as fibre bundles, without explicitly referring to the structure group. Fortunately, the following classical result of R. Hermann provides us with one possible route:

\begin{theorem}[{{\cite{Hermann1960},\cite[Theorem 9.3]{Besse2007EinsteinManifolds}}}]
  \label{thm:hermann1960}
  Let $\pi:E\rightarrow M$ be a Riemannian submersion (c.f. \cite[Definition 9.8]{Besse2007EinsteinManifolds}). If $E$ is a complete, then $\pi:E\rightarrow M$ is a fibre bundle.
\end{theorem}

The proof of Theorem~\ref{thm:hermann1960} is constructive. In a nutshell, Hermann explicitly constructed local trivializations around each $x\in M$, by connecting points on the fibre $\pi^{-1} \left( x \right)$ to points on any neighboring fibre $\pi^{-1}\left( y \right)$ by horizontally lifting the geodesic on $M$ that connects $x$ to $y$. Here the horizontal lifting is made possible by the Riemannian structure on $E$, which canonically splits the tangent bundle of $E$ into the direct sum of a horizontal and vertical subbundles. As pointed out in \cite[\S9.E]{Besse2007EinsteinManifolds}, the horizontal subbundle is an \emph{Ehresmann connection} (see \ref{sec:connections}) on the fibre bundle. The structure group of the fibre bundle can then be determined from the \emph{holonomy} of the Ehresmann connection; see \cite[\S9.47]{Besse2007EinsteinManifolds} for more details. Obviously, the data required in Theorem~\ref{thm:hermann1960} to fully specify the fibre bundle structure can be provided in a slightly different order: if we are given a Riemannian manifold $M$ and another manifold $E$ but without a prescribed Riemannian structure, and $\pi:E\rightarrow M$ is a smooth submersion with an Ehresmann connection on $E$, then we can define a product Riemannian structure on $E$ which imposes the orthogonality between the horizontal and vertical subbundles of the tangent bundle $T\!E$. It is straightforward to verify that $\pi:E\rightarrow M$ is a Riemannian submersion with such a Riemannian structure on $E$. In other words, a fibre bundle can be defined equivalently by a smooth submersion between the total and base manifold (with appropriate completeness assumptions), a Riemannian structure on the base manifold, and an Ehresmann connection. We close the discussion in this section by emphasizing that, though it might appear that our fibre bundle framework ``discards'' the notion of structure groups compared with the fibre bundle formulation pioneered in \cite{SingerWu2012VDM,SW2016}, structure groups indeed are specified, just in an indirect manner.



\subsection{Horizontal Random Walks and Diffusion Processes on Fibre Bundles}
\label{sec:horiz-rand-walk}

Equipped with the geometric notion of fibre bundles, we are now ready to define a random walk tailored to a data set with pairwise correspondences. Starting from a point $e\in E$, in one step a random walker is allowed to jump to a neighboring $e'\in E$ only $\pi \left( e' \right)\neq \pi \left( e \right)$ and $e,e'$ can be joined by a horizontally lifted image of a piecewise geodesic connecting $\pi \left( e \right)$ to $\pi' \left( e \right)$ on $M$. More specifically, just as a standard random walk on $M$ jumps from $x\in M$ to a neighbor $y\in M$ following a transition probability $\mathbb{P}\left( y\mid x \right)$, a horizontal random walk jumps from $e\in F_x\subset E$ to $P_{yx}\left( e \right)\in F_y\subset E$ with transition probability $\mathbb{P}\left( y\mid x \right)$; note in particular that this transition probability depends only on the projections $x=\pi \left( e \right)$ and $y=\pi \left( P_{yx}\left( e \right) \right)$. In this sense, a horizontal random walk on the fibre bundle $\mathscr{E}$ can be viewed as ``driven'' by an underlying random walk on the base manifold $M$ (see Figure~\ref{fig:stochastic_parallel_transport} for an illustration). Passing to the continuous limit (in the weak sense as the random walk step size approaches zero, see \cite{BNR2017} and Section~\ref{sec:conv-rate-from}), both random walks on the fibre bundle and the base manifold converge to diffusion processes. For the convenience of exposition, hereafter we refer to the limit diffusion process on the fibre bundle as the \emph{horizontal lift} of the limit diffusion process on the base manifold. In the Riemannian setting, this construction is reminiscent of the notion of \emph{stochastic parallel transport}~\cite{Ito1962ICM,Hsu2002Book} in stochastic differential geometry.
\begin{figure}[htp]
  \centering
  \includegraphics[width=0.75\textwidth]{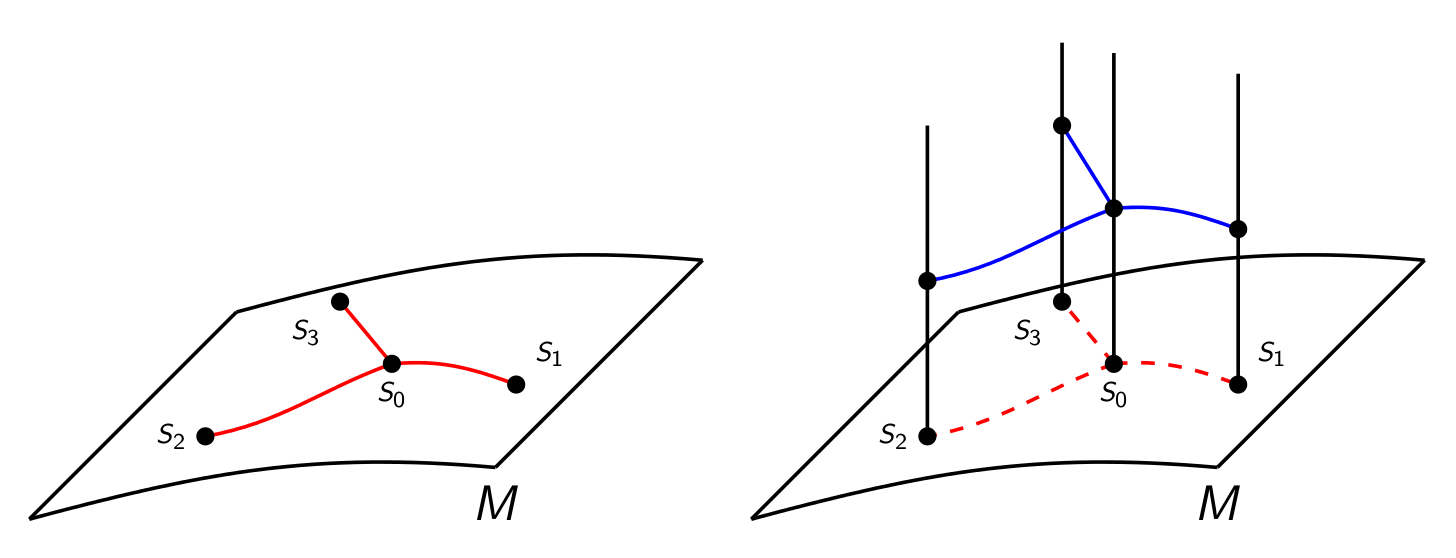}
  \caption{Left: A random walk on the base manifold $M$ jumps in one step from a point $s_0$ on $M$ to one of its neighboring points $s_1,s_2,s_3$. Right: The horizontal lift of the same random walk from $M$ to the fibre bundle $\mathscr{E}$, which jumps in one step from a point on fibre $F_{s_0}$ to a neighboring point on the fibre $F_{s_1}$, $F_{s_2}$, or $F_{s_3}$.}
  \label{fig:stochastic_parallel_transport}
\end{figure}


The following is a precise description of the horizontal diffusion processes on the fibre bundle in the language of symmetric Markov semigroups. For clarity, let us assume $M$ and $F$ are both orientable. Let \emph{kernel} $K:\mathbb{R}\rightarrow \mathbb{R}^{\geq 0}$ be a smooth function compactly supported on the unit interval $\left[ 0,1 \right]$. For \emph{bandwidth parameter} $\epsilon>0$ and any pairs of $x,y\in M$, define
\begin{equation*}
  K_{\epsilon}\left( x,y \right)=K \left( \frac{d_M^2 \left( x,y \right)}{\epsilon} \right),
\end{equation*}
where $d_M \left( \cdot,\cdot \right)$ stands for the geodesic distance on $M$. Note that $K_{\epsilon}\left( \cdot,\cdot \right)$ is non-zero only if $x,y$ are sufficiently close to each other under the Riemannian metric on $M$, due to the compactness of the kernel function $K$. For any $f\in C^{\infty}\left( E \right)$, define the diffusion operator $H_{\epsilon}:C^{\infty}\left( E \right)\rightarrow C^{\infty}\left( E \right)$ as
\begin{equation*}
  H_{\epsilon}f \left( x, v \right)=\int_MK_{\epsilon}\left( x,y \right)f \left( y, P_{yx}v \right)d\mathrm{vol}_M \left( y \right),\quad \forall x\in M, v\in F_x,
\end{equation*}
where $d\mathrm{vol}_M$ stands for the Riemannian volume element on $M$. Intuitively, at each point $\left( x,v \right)\in E$, $H_{\epsilon}$ averages the value of $f$ on a neighborhood around $\left( x,v \right)$ expanded by parallel-transporting $\left( x,v \right)$ along geodesics on $M$. Compared with the diffusion semigroup characterization of a diffusion process on the base manifold, $H_{\epsilon}$ incorporates the extra information provided by the connection. 

Variants of $H_{\epsilon}$ that involve the sampling density can be similarly constructed, which is useful since in practice it is difficult to uniformly sample from $M$. Consider a density function $p\in C^{\infty}\left( M \right)$ with respect to which the samples are generated. For simplicity, assume $p$ is bounded away from zero from below, i.e.,
\begin{equation}
\label{eq:assum_density_below}
  \int_M p\left( y \right)d\mathrm{vol}_M \left( y \right)=1\quad\textrm{and}\quad p \left( y \right)\geq p_0>0\quad\forall y\in M.
\end{equation}
Similar to the construction of diffusion maps~\cite{CoifmanLafon2006}, we can decouple the geometry of the manifold from the influence of sampling density by normalizing the integral kernel. To this end, we set
\begin{equation*}
  p_{\epsilon} \left( x \right)=\int_MK_{\epsilon}\left( x,y \right)p \left( y \right)d\mathrm{vol}_M \left( y \right)
\end{equation*}
and denote for any \emph{normalization parameter} $\alpha\in \left[ 0,1 \right]$
\begin{equation*}
  K_{\epsilon}^{\left( \alpha \right)}\left( x,y \right)=\frac{K_{\epsilon}\left( x,y \right)}{p_{\epsilon}^{\alpha}\left( x \right)p_{\epsilon}^{\alpha}\left( y \right)},
\end{equation*}
then define the \emph{horizontal diffusion operator}
\begin{equation}
\label{eq:horizontal_diffusion_operator}
  H_{\epsilon}^{\left( \alpha \right)}f \left( x,v \right)=\frac{\displaystyle\int_MK_{\epsilon}^{\left( \alpha \right)}\left( x,y \right)f \left( y, P_{yx}v \right)p \left( y \right)d\mathrm{vol}_M \left( y \right)}{\displaystyle\int_MK_{\epsilon}^{\left( \alpha \right)}\left( x,y \right)p \left( y \right)d\mathrm{vol}_M \left( y \right)},\quad \forall x\in M,v \in F_x
\end{equation}
for any $f\in C^{\infty}\left( E \right)$. As we shall see later, the infinitesimal generator of $H_{\epsilon}^{\left( \alpha \right)}$ is a second order partial differential operator in which all derivatives, as vector fields on $E$, are horizontal.

A different practical consideration is that pairwise correspondences can be relaxed from maps to couplings of probability measures when fibres are discretized. Examples for such relaxed pairwise correspondences include the \emph{soft-assign Procrustes matching}~\cite{RangarajanChuiBookstein1997} in medical imaging, the \emph{soft maps}~\cite{SNBBG2012SoftMaps} in geometry processing, the \emph{transport plans}~\cite{Villani2003,Villani2008} in optimal transportation, to name just a few. In the HDM framework, these relaxed correspondences also define diffusion processes on the fibre bundle, now consisting of two ingredients: a horizontal lift of a diffusion process on the base manifold, composed with another diffusion process within the fibre. In this setting, it is an interesting question to ``learn'' the connection from the composition of two diffusion processes; practically, this amounts to ``recovering'' maps from couplings in a collection of data objects. In some applications (see e.g. Section~\ref{sec:autogm}), one can also ``learn'' the structure of the template fibre from the connection. Making an analogy with the terminology \emph{manifold learning}, we call this type of learning problems \emph{fibre learning}. Similar to \eqref{eq:horizontal_diffusion_operator}, we can write the diffusion process considered in fibre learning in the language of Markov semigroups. Let $K:\mathbb{R}^2\rightarrow\mathbb{R}^{\geq 0}$ be a smooth bi-variate function compactly supported on the unit square $\left[ 0,1 \right]\times \left[ 0,1 \right]$, and let $\epsilon>0,\delta>0$ be \emph{bandwidth parameters}. Define
\begin{equation}
\label{eq:warped_kernel}
  K_{\epsilon,\delta}\left( x,v;y,w \right)=K \left( \frac{d_M^2 \left( x,y \right)}{\epsilon},\frac{d_{F_y}^2 \left( P_{yx}v,w \right)}{\delta} \right)
\end{equation}
for $\left( x,v \right)\in E,\left( y,w \right)\in E$, where $d_M \left( \cdot,\cdot \right), d_{F_y} \left( \cdot,\cdot \right)$ are the geodesic distances on $M$, $F_y$ respectively. Assume $p\in C^{\infty}\left( E \right)$ is a density function bounded away from zero from below, i.e.
\begin{equation}
\label{eq:density_assumption}
  \int_M\!\int_{F_y}p \left( y,w \right)d\mathrm{vol}_{F_y}\left( w \right)d\mathrm{vol}_M \left( y \right)=1
\end{equation}
and
\begin{equation}
\label{eq:density_positivity_assumption}
p \left( y,w \right)\geq p_0>0,\quad\forall \left( y,w \right)\in E.
\end{equation}
For $\alpha\in \left[ 0,1 \right]$, if we set
\begin{equation*}
  p_{\epsilon,\delta}\left( x,v \right)=\int_M\!\int_{F_y}K_{\epsilon,\delta}\left( x,v;y,w \right)p \left( y,w \right)d\mathrm{vol}_{F_y}\left( w \right)d\mathrm{vol}_M \left( y \right),
\end{equation*}
and
\begin{equation*}
  K_{\epsilon,\delta}^{\left( \alpha \right)}\left( x,v;y,w \right)=\frac{K_{\epsilon,\delta}\left( x,v;y,w \right)}{p_{\epsilon,\delta}^{\alpha}\left( x,v \right)p_{\epsilon,\delta}^{\alpha}\left( y,w \right)},
\end{equation*}
then the \emph{coupled diffusion operator} for all $\left( x,v \right)\in E$ can be written as
\begin{equation}
  \label{eq:fibre_bundle_diffusion_operator}
  H_{\epsilon,\delta}^{\left( \alpha \right)}f \left( x,v \right)=\frac{\displaystyle\int_M\!\int_{F_y}K_{\epsilon,\delta}^{\left( \alpha \right)}\left( x,v;y,w \right)f \left( y,w \right)p \left( y,w \right)d\mathrm{vol}_{F_y}\left( w \right)d\mathrm{vol}_M \left( y \right)}{\displaystyle\int_M\!\int_{F_y}K_{\epsilon,\delta}^{\left( \alpha \right)}\left( x,v;y,w \right)p \left( y,w \right)d\mathrm{vol}_{F_y}\left( w \right)d\mathrm{vol}_M \left( y \right)}. 
\end{equation}
The infinitesimal generator of $H_{\epsilon,\delta}^{\left( \alpha \right)}$ has to be considered differently from that of $H_{\epsilon}^{\left( \alpha \right)}$ due to the appearance of two (instead of one) bandwidth parameters $\epsilon,\delta$. It turns out that the relative rate with which $\epsilon$ and $\delta$ approach $0$ affects the type of the infinitesimal generator associated with the diffusion process, see Section~\ref{sec:infin-gener}.

\section{The HDM Algorithm}
\label{sec:algorithm}
In this section, we describe the manifold learning framework of HDM that extracts feature information in a data set with pairwise similarity and structural correspondences, based on the geometric intuition explained in Section~\ref{sec:formulation}. The construction of graph horizontal Laplacians and spectral embeddings apply to any fibred graph and symmetric similarity measure satisfying the structural assumptions in this section; the theoretical results to be presented in Section~\ref{sec:infin-gener} and Section~\ref{sec:finite-sampl-results} apply to the concrete scenario where the graph arises from sampling the fibre bundle as an embedded submanifold of an ambient Euclidean space and the similarity measure encodes the connection information (see Section~\ref{sec:conv-rate-from} for more details).

\subsection{Graph Horizontal Laplacians}
\label{sec:graph-hypo-lapl}
The data set considered in the HDM framework is a triplet $\left( \mathscr{X},\rho,G \right)$, where
\begin{enumerate}[(1)]
\item\label{item:10} The \emph{total} data set $\mathscr{X}$ can be partitioned into a collection of \emph{data objects} $X_1,\cdots,X_n$
  \begin{equation*}
    \mathscr{X}=\bigcup_{j=1}^n X_j,\quad X_j\cap X_k=\emptyset\textrm{ for all }1\leq j\neq k\leq n,
  \end{equation*}
where each data object $X_j$ is referred to as the $j$-th \emph{fibre} of $\mathscr{X}$, which contains $\kappa_j$ \emph{data points}
\begin{equation*}
  X_j=\left\{ x_{j,1},x_{j,2},\cdots,x_{j,\kappa_j} \right\}.
\end{equation*}
We call the collection of fibres the \emph{base} data set
\begin{equation*}
  \mathscr{B}=\left\{ X_1,X_2,\cdots,X_n \right\},
\end{equation*}
and let $\pi:\mathscr{X}\rightarrow\mathscr{B}$ be the \emph{canonical projection} from $\mathscr{X}$ to $\mathscr{B}$
\begin{equation*}
  \begin{aligned}
    \pi:\mathscr{X} & \longrightarrow \mathscr{B}\\
    x_{j,k} & \longmapsto X_j,\quad 1\leq j\leq n, 1\leq k\leq \kappa_j.
  \end{aligned}
\end{equation*}
Denote the total number of points in $\mathscr{X}$ as
\begin{equation*}
  \kappa=\kappa_1+\kappa_2+\cdots+\kappa_n.
\end{equation*}

\item\label{item:11} The \emph{mutual similarity measure} $\rho:\mathscr{X}\times \mathscr{X}\rightarrow \mathbb{R}^{\geq 0}$ is a symmetric non-negative function that vanishes on each fibre, i.e.
\begin{equation*}
    \rho \left( \xi, \eta\right)\geq 0,\quad \rho \left( \xi, \eta \right) = \rho \left( \eta, \xi \right)\quad\forall \xi,\eta\in\mathscr{X}
\end{equation*}
and
\begin{equation*}
  \rho \left( \xi, \eta \right)=0\quad\textrm{if $\xi, \eta\in X_j$ for some $1\leq j\leq n$.}
\end{equation*}
For simplicity of notation, we denote the restriction of $\rho$ on $X_i\times X_j$ as
\begin{equation*}
  \rho_{ij}\left( s,t \right):=\rho \left( x_{i,s},x_{j,t} \right)\quad \forall x_{i,s}\in X_i,\,x_{j,t}\in X_j.
\end{equation*}
In words, $\rho_{ij}$ is an $\kappa_i\times \kappa_j$ matrix on $\mathbb{R}$, to which we will refer as the \emph{mutual similarity matrix} between $X_i$ and $X_j$. Note that $\rho_{ij}=0$ if $i=j$.

\item\label{item:12} The \emph{affinity graph} $G= \left( V,E \right)$ has $\kappa$ vertices, with each $v_{i,s}$ corresponding to a point $x_{i,s}\in \mathscr{X}$. Without loss of generality, assume $G$ is connected. (In our applications, each $x_{i,s}$ is typically connected to several $x_{j,t}$'s on neighboring fibres.) If there is an edge between $v_{i,s}$ and $v_{j,t}$ in $G$, then $x_{i,s}$ is a \emph{neighbor} of $x_{j,t}$ (and $x_{j,t}$ is a neighbor of $x_{i,s}$); $X_i$ is called a \emph{neighbor} of $X_j$ (and similarly $X_j$ a neighbor of $X_i$) if there is an edge in $G$ linking one point in $X_i$ with one point in $X_j$. Implicitly, these define a graph $G_B = \left( V_B,E_B \right)$ in which vertices of $V_B$ are in one-to-one correspondences with fibres of $\mathscr{X}$, and $E_B$ encodes the neighborhood relations between pairs of fibres. $G_B$ will be called as the \emph{base affinity graph}.
\end{enumerate}

With the triplet $\left( \mathscr{X},\rho,G \right)$ specified, we detail below the construction of the \emph{graph horizontal Laplacian}. Let $W\in\mathbb{R}^{\kappa\times \kappa}$ be the \emph{weighted adjacency matrix} of the graph $G$, i.e., $W$ is a block matrix in which the $\left( i,j \right)$-th block is $\rho_{ij}$.
The $\left( s,t \right)$ entry in $W_{ij}$ stands for the edge weight $\rho_{ij}\left( s,t \right)$ between $v_{i,s}$ and $v_{j,t}$. Since $\rho_{ij}=\rho_{ji}^{\top}$, $W$ is a symmetric matrix. Let $D$ be the $\kappa\times \kappa$ diagonal matrix in which the $j$-th diagonal entry equals to the $j$-th row sum of $W$.
We define the \emph{graph horizontal Laplacian} for the triplet $\left( \mathscr{X},\rho,G \right)$ as the weighted graph Laplacian of $G$ with edge weights $W$, i.e.
\begin{equation}
  \label{eq:hypoelliptic_graph_laplacian}
  L^H:=D-W.
\end{equation}
Since $G$ is connected, the diagonal elements of $D$ are all non-zero. Thus $D$ is invertible and we can define the \emph{random-walk} and \emph{normalized} version of $L^H$:
\begin{equation}
  \label{eq:hypoelliptic_graph_laplacian_rw}
  L^H_{\textrm{rw}}:=D^{-1}L^H=I-D^{-1}W,
\end{equation}
\begin{equation}
  \label{eq:hypoelliptic_graph_laplacian_normalized}
  L^H_{*}:=D^{-1/2}L^HD^{-1/2} = I-D^{-1/2}WD^{-1/2}.
\end{equation}
Following~\cite{CoifmanLafon2006}, we can also repeat these constructions on a renormalized graph of $G$ by setting for some $\alpha\in \left[ 0, 1 \right]$
\begin{equation}
  \label{eq:W_alpha}
  W_{\alpha}:=D^{-\alpha}WD^{-\alpha}
\end{equation}
and constructing the graph horizontal Laplacians from $W_{\alpha}$ instead of $W$. More precisely, let $D_{\alpha}$ be the $\kappa\times \kappa$ diagonal matrix in which the $j$-th diagonal entry equals to the $j$-th row sum of $W_{\alpha}$, and
set
\begin{equation}
  \label{eq:hypoelliptic_graph_laplacian_alpha}
  L_{\alpha}^H:=D_{\alpha}-W_{\alpha},
\end{equation}
\begin{equation}
  \label{eq:hypoelliptic_graph_laplacian_rw_alpha}
  L^{H}_{\alpha,\textrm{rw}}:=D_{\alpha}^{-1}L^{H}_\alpha=I-D_{\alpha}^{-1}W_{\alpha},
\end{equation}
\begin{equation}
  \label{eq:hypoelliptic_graph_laplacian_normalized_alpha}
  L^{H}_{\alpha,*}:=D_{\alpha}^{-1/2}L^H_{\alpha}D_{\alpha}^{-1/2} = I-D_{\alpha}^{-1/2}W_{\alpha}D_{\alpha}^{-1/2}.
\end{equation}

\begin{remark}
  The block structure in the matrix $W$ is reminiscent of the \emph{graph connection Laplacian} \cite[Section~3]{SingerWu2012VDM}, but the constraints on the blocks are different: blocks of the graph connection Laplacian are built from matrix representations of a Lie group, but blocks of the graph horizontal Laplacian represent similarity between data objects and are matrices with non-negative entries. The normalization we apply to $W$ is the same as for standard diffusion maps \cite{CoifmanLafon2006}. Formally, the constructions of $L_{\alpha}^H$, $L^{H}_{\alpha,\textrm{rw}}$, and $L^{H}_{\alpha,*}$, as well as the embeddings derived from their eigen-decompositions, appears identical to their counterparts in standard diffusion maps, but we will show below that the unique fibred structure of the graph $G$ allows us to characterize more subtle geometry in $\left( \mathscr{X},\rho,G \right)$ than standard diffusion maps could (see Remark~\ref{rem:hdm-nontrivial} and Remark~\ref{rem:hdm-nontrivial-sampling}).
\end{remark}

\subsection{Spectral Distances and Embeddings}
\label{sec:spectr-dist-embedd}

Spectral distances are defined via the eigen-decompositions of graph Laplacians. Since $L^H_{\alpha,\mathrm{rw}}$ differs from $L^{H}_{\alpha,*}$ only by a similarity transformation
$$L_{\alpha,\textrm{*}}^H=D_{\alpha}^{1/2}L_{\alpha,\mathrm{rw}}^H D_{\alpha}^{-1/2},$$
the two Laplacians have essentially the same eigen-decomposition. We shall focus on $L^{H}_{\alpha,*}$ for the rest of this section due to its computational advantage as a real symmetric matrix.

Any right eigenvector $v\in\mathbb{R}^{\kappa}$ of $L_{\alpha,*}^H$ defines a function on the vertices of $G$. By the construction of $L_{\alpha,*}^H$, the length-$\kappa$ vector $v$, when written as the concatenation of $n$ segments of length $\kappa_1,\cdots,\kappa_n$ respectively, defines a function on each of the $n$ fibres $X_1,\cdots,X_n$. We assume eigenvectors are always column vectors, and write
\begin{equation*}
  v = \left( v^{\top}_{\left[ 1 \right]},\cdots, v_{\left[ n \right]}^{\top}\right)^{\top}
\end{equation*}
where each column vector $v_{\left[ j \right]}\in\mathbb{R}^{\kappa_j}$ defines a function on the fibre $X_j$. Now let
$$\lambda_0\leq\lambda_1\leq\lambda_2\leq\cdots\leq\lambda_{\kappa-1}$$
be the $\kappa$ eigenvalues of $L^H_{\alpha,*}$ in ascending order, and denote the eigenvector corresponding to eigenvalue $\lambda_j$ as $v_j$. By the connectivity assumption for $G$, we know from spectral graph theory~\cite{Chung1997} that $\lambda_0=0$, $\lambda_0<\lambda_1$, and $v_0$ is a constant multiple of the column vector with all entries equal to $1$; we have thus
\begin{equation*}
  0=\lambda_0<\lambda_1\leq \lambda_2\leq\cdots\leq \lambda_{\kappa-1}.
\end{equation*}
By the spectral decomposition of $L_{\alpha,*}^H$,
\begin{equation}
  \label{eq:spectral_decomposition}
  L_{\alpha,*}^H=\sum_{l=0}^{\kappa-1} \lambda_l v_lv_l^{\top},
\end{equation}
and for any fixed \emph{diffusion time} $t\in\mathbb{R}^{>0}$,
\begin{equation}
  \label{eq:spectral_decomposition_t}
  \left(L_{\alpha,*}^H\right)^t=\sum_{l=0}^{\kappa-1} \lambda^t_l v_lv_l^{\top},
\end{equation}
with the $\left( i,j \right)$-th block
\begin{equation}
  \label{eq:spectral_decomposition_block}
  \left(\left(L_{\alpha,*}^H\right)^t\right)_{ij}=\sum_{l=0}^{\kappa-1} \lambda^t_l v_{l \left[ i \right]} v_{l \left[ j \right]}^{\top}.
\end{equation}
In general, this block is not a square matrix. Its Frobenius norm can be computed as
\begin{equation}
  \label{eq:spectral_decomposition_block_HSnorm}
  \begin{aligned}
    \left\|\left(\left(L_{\alpha,*}^H\right)^t\right)_{ij}\right\|_{\mathrm{F}}^2&=\mathrm{Tr}\left[ \left(\left(L_{\alpha,*}^H\right)^t\right)_{ij}\left(\left(L_{\alpha,*}^H\right)^t\right)_{ij}^{\top} \right]=\mathrm{Tr}\left[ \sum_{l,m=0}^{\kappa-1}\lambda_l^t\lambda_m^t v_{l \left[ i \right]}v_{l \left[ j \right]}^{\top}v_{m \left[ j \right]}v_{m \left[ i \right]}^{\top} \right]\\
    &=\mathrm{Tr}\left[ \sum_{l,m=0}^{\kappa-1}\lambda_l^t\lambda_m^t v_{m \left[ i \right]}^{\top}v_{l \left[ i \right]}v_{l \left[ j \right]}^{\top}v_{m \left[ j \right]} \right]=\sum_{l,m=0}^{\kappa-1}\lambda_l^t\lambda_m^t v_{m \left[ i \right]}^{\top}v_{l \left[ i \right]}v_{l \left[ j \right]}^{\top}v_{m \left[ j \right]}.
  \end{aligned}
\end{equation}
Define the \emph{horizontal base diffusion map} (HBDM) as
\begin{equation}
  \label{eq:hbdm}
  \begin{aligned}
    V^t:\mathscr{B}&\longrightarrow \mathbb{R}^{\kappa^2}\\
    X_j &\longmapsto \left( \lambda_l^{t/2}\lambda_m^{t/2}v_{l \left[ j \right]}^{\top}v_{m \left[ j \right]} \right)_{0\leq l,m\leq \kappa-1}
  \end{aligned}
\end{equation}
with which
\begin{equation}
  \label{eq:HDM_innerproduct}
  \begin{aligned}
    \left\|\left(\left(L_{\alpha,*}^H\right)^t\right)_{ij}\right\|_{\mathrm{F}}^2 = \left\langle V^t \left( X_i \right), V^t \left( X_j \right) \right\rangle,
\end{aligned}
\end{equation}
where $\langle \cdot,\cdot\rangle$ is the standard Euclidean inner product on $\mathbb{R}^{\kappa^2}$. Furthermore, we define the \emph{horizontal base diffusion distance} (HBDD) on $\mathscr{B}$ as
\begin{equation}
  \label{eq:hbdd}
  \begin{aligned}
    d_{\mathrm{HBDM},t}&\left( X_i,X_j \right) = \left\|V^t \left( X_i \right)- V^t \left( X_j \right) \right\|\\
    &=\left\{\left\langle V^t \left( X_i \right), V^t \left( X_i \right) \right\rangle+\left\langle V^t \left( X_j \right), V^t \left( X_j \right) \right\rangle-2\left\langle V^t \left( X_i \right), V^t \left( X_j \right) \right\rangle  \right\}^{\frac{1}{2}}.
  \end{aligned}
\end{equation}

From a learning point of view, the map $V^t:\mathscr{B}\rightarrow\mathbb{R}^{\kappa^2}$ is equivalent to the unsupervised features learned from the data set with structural correspondences. Note also that HBDM embeds the base data set $\mathscr{B}$ into a Euclidean space of dimension $\kappa^2$, which is of much higher dimensionality than the size of the original data set. In practice, however, one often truncates the spectrum of graph Laplacians, thus embedding the data set into a Euclidean of reduced dimensionality. In our numerical experiments and applications (see Section~\ref{sec:autogm}), we found it is usually sufficient to retain the first $O \left( \sqrt{\kappa} \right)$ to $O \left( \kappa \right)$ eigenvalues. Even though this truncation still involves higher spatial complexity than diffusion maps, our results show that HBDM significantly outperforms DM for our purposes; we thus believe that the high-dimensional embedding is a modest price to pay for extracting the hidden information in the structural correspondences.


In addition to handling the base data set $\mathscr{B}$, HDM is also capable of embedding the total data set $\mathscr{X}$ into Euclidean spaces. Define for each \emph{diffusion time} $t\in\mathbb{R}^{+}$ the \emph{horizontal diffusion map} (HDM)
\begin{equation}
\label{eq:hdm}
  \begin{aligned}
    H^t:\mathscr{X}&\longrightarrow \mathbb{R}^{\kappa-1}\\
    x_{j,s}&\longmapsto \left(\lambda_1^t v_{1\left[ j \right]}\left( s \right),\lambda^t_2v_{2 \left[ j \right]}\left( s \right),\cdots,\lambda_{\kappa-1}^tv_{\left(\kappa-1\right)\left[ j \right]}\left( s \right) \right).
  \end{aligned}
\end{equation}
where $v_{l \left[ j \right]} \left( s \right)$ is the $s$-th entry of the $j$-th segment of the $l$-th eigenvector, with $j=1,\cdots,n, s=1,\cdots,\kappa_j$. We could also have written
\begin{equation*}
  v_{l \left[ j \right]} \left( s \right) = v_l \left( s_j+s \right),\quad\textrm{where } s_1=0 \textrm{ and } s_j=\sum_{p=1}^{j-1}\kappa_p \textrm{ for }j\geq 2.
\end{equation*}
Following a similar argument as in~\cite{CoifmanLafon2006}, we can define the \emph{horizontal diffusion distance} (HDD) on $\mathscr{X}$ as
\begin{equation}
  \label{eq:hdd}
  d_{\mathrm{HDM},t}\left( x_{i,s},x_{j,t} \right) = \left\|H^t \left( x_{i,s} \right) - H^t \left( x_{j,t} \right) \right\|.
\end{equation}
As it stands, $H^t$ embeds the total data set $\mathscr{X}$ into a Euclidean space preserving the horizontal diffusion distance on $\mathscr{X}$. Moreover, this embedding automatically suggests a global registration for all fibres that respects the mutual similarity measure $\rho$; similar ideas was already implicit in~\cite{Kim12FuzzyCorr}. For simplicity of notation, let us write
\begin{equation*}
  H^t_j:=H^t\restriction X_j
\end{equation*}
for the restriction of $H^t$ to fibre $X_j$, and call this the $j$-th \emph{component} of $H^t$. Up to scaling, the components of $H^t$ bring the fibres of $\mathscr{X}$ to a common ``template'', such that points $x_{i,s}$ and $x_{j,t}$ with a high similarity measure $\rho_{ij}\left( s,t \right)$ tend to be close to each other in the embedded Euclidean space. One can then reconstruct pairwise structural correspondences between fibres $X_i,X_j$ in the embedded Euclidean space, now between the embedded point clouds in $\mathbb{R}^{\kappa^2}$. With appropriate truncation of the spectrum of the graph horizontal Laplacian, these reconstructed structural correspondences are the ``denoised version'' of the original correspondences. Moreover, recalling that each $X_j$ is sampled from some manifold $F_j$, one can often estimate a \emph{template fibre} $F\subset\mathbb{R}^m$ from the embedded images
$$H^t_1 \left( X_1 \right),\cdots,H^t_n \left( X_n \right),$$
and extend (by interpolation) $H_j^t$ from a discrete correspondence to a continuous bijective map from $F_j$ to $F$, then build correspondence maps between an arbitrary pair $X_i,X_j$ by composing (the interpolated continuous maps) $H_i^t$ with $\left(H_j^t\right)^{-1}$. Pairwise correspondences reconstructed in this manner are globally consistent, since they all go through the common template manifold $F$. We discuss in greater detail an application of HDM and HDD to a data set of shapes in geometric morphometrics in Section~\ref{sec:autogm}.

\section{Infinitesimal Generators for Horizontal and Coupled Diffusion Operators}
\label{sec:infin-gener}


We are now ready to present the main technical results of this paper. First, we characterize the infinitesimal generator of the horizontal diffusion operator $H^{\left( \alpha \right)}_{\epsilon}$ in \eqref{eq:horizontal_diffusion_operator}.

\begin{theorem}
\label{thm:horizontal_diffusion_generator}
  Suppose $\mathscr{E}=\left(E,M,F,\pi\right)$ is a fibre bundle, $M$ is a smooth Riemannian manifold without boundary, and $E$ is equipped with the Riemannian metric \eqref{eq:metric_on_total_manifold}. For any $f\in C^{\infty}\left( E \right)$ and $\left( x,v \right)\in E$,
  \begin{equation}
    \label{eq:horizontal_generator_theorem}
    \begin{aligned}
      \lim_{\epsilon\rightarrow 0} \frac{H_{\epsilon}^{\left( \alpha \right)} f \left( x,v \right)-f \left( x,v \right)}{\epsilon}=\frac{m_2}{2m_0}\frac{\left[\Delta_H\left(f\bar{p}^{1-\alpha}\right) -f\Delta_H\bar{p}^{1-\alpha}\right]\left( x,v \right)}{p^{1-\alpha}\left( x \right)},
    \end{aligned}
  \end{equation}
where $m_0$, $m_2$ are positive constants depending only on the base manifold $M$ and the kernel $K$, $\Delta_H$ is the rough horizontal Laplacian on $E$ in \eqref{eq:bochner_horizontal_laplacian}, and $\bar{p}=p\circ\pi\in C^{\infty}\left( E \right)$.
\end{theorem}
The proof of Theorem~\ref{thm:horizontal_diffusion_generator} can be found in \ref{sec:proof_biase}. 
\begin{corollary}
  Under the same assumptions as in {\rm Theorem~\ref{thm:horizontal_diffusion_generator}}, when $\alpha=1$,
  \begin{equation}
    \label{eq:corollary_horizontal_diffusion_generator}
    \lim_{\epsilon\rightarrow 0} \frac{H_{\epsilon}^{\left( 1 \right)} f \left( x,v \right)-f \left( x,v \right)}{\epsilon} = \frac{m_2}{2m_0}\Delta_Hf \left( x,v \right).
  \end{equation}
\end{corollary}

Characterizing the infinitesimal generator of the coupled diffusion operator $H_{\epsilon,\delta}^{\left( \alpha \right)}$ is slightly more subtle: the generator of the diffusion process depends on the relative speed at which the two bandwidth parameters $\epsilon,\delta$ approach $0$. For clarity, we first state the result for the case when the ratio $\delta/\epsilon$ remains bounded as $\epsilon,\delta\rightarrow0$.

\begin{theorem}[Bounded Ratio $\delta/\epsilon$]
\label{thm:warped_diffusion_generator}
  Suppose $\mathscr{E}=\left(E,M,F,\pi\right)$ is a fibre bundle, $M$ is a smooth Riemannian manifold without boundary, and $E$ is equipped with the Riemannian metric \eqref{eq:metric_on_total_manifold}. For any $f\in C^{\infty}\left( E \right)$ and $\left( x,v \right)\in E$, if $\delta=O \left( \epsilon \right)$ as $\epsilon\rightarrow0$,
then
  \begin{equation}
    \label{eq:warped_horizontal_generator_theorem}
    \begin{aligned}
      H_{\epsilon,\delta}^{\left( \alpha \right)} & f \left( x,v \right) = f \left( x,v \right) + \epsilon \frac{m_{21}}{2m_0} \frac{\left[\Delta_H \left( fp^{1-\alpha} \right)-f\Delta_Hp^{1-\alpha} \right]\left( x,v \right)}{p^{1-\alpha}\left( x,v \right)}\\
      &+\delta\frac{m_{22}}{2m_0} \frac{\left[\Delta_E^V \left( fp^{1-\alpha} \right)-f\Delta_E^Vp^{1-\alpha} \right]\left( x,v \right)}{p^{1-\alpha}\left( x,v \right)}+O \left( \epsilon^2+\epsilon\delta+\delta^2 \right),
    \end{aligned}
  \end{equation}
where $m_0$, $m_{21}$, $m_{22}$ are positive constants depending only on the total manifold $E$ and the kernel $K$, $\Delta_H$ is the rough horizontal Laplacian on $E$ defined in \eqref{eq:bochner_horizontal_laplacian}, and $\Delta_E^V$ is the vertical Laplacian of the fibre bundle $\mathscr{E}$ defined in \eqref{eq:horizontal_vertical_laplacians}.
\end{theorem}
For a proof of Theorem~\ref{thm:warped_diffusion_generator}, see \ref{sec:proof_biase}. Note that in the distribution sense Theorem~\ref{thm:horizontal_diffusion_generator} can be interpreted as a special case of Theorem~\ref{thm:warped_diffusion_generator} when $\delta=o \left( \epsilon \right)$ as $\epsilon\rightarrow0$. From a different point of view, Theorem~\ref{thm:warped_diffusion_generator} can also be interpreted as \cite[Theorem 2]{CoifmanLafon2006} applied on a fibre bundle $\left(E,M,F,\pi\right)$ with a family of varying Riemannian metrics
\begin{equation*}
  g_{\delta/\epsilon}^E=g^M\oplus \frac{\delta}{\epsilon}g^F,
\end{equation*}
which is known as the \emph{canonical variation} in the literature of Riemannian submersion \cite[\S9.G]{Besse2007EinsteinManifolds}\cite[\S2.7.5]{GLP1999}. If $\delta/\epsilon\rightarrow 0$, then the rescaled metric
\begin{equation*}
  \frac{\epsilon}{\delta}g_{\delta/\epsilon}^E=\frac{\epsilon}{\delta}g^M\oplus g^F
\end{equation*}
is said to approach its \emph{adiabatic limit}, or taking adiabatic limits amounts to blowing up or contracting the fibres, which is very useful in studying foliations. In the horizontal diffusion maps framework, the adiabatic limits can be indirectly taken by adjusting the relative magnitudes of the horizontal and vertical bandwidth parameters; see Figure~\ref{fig:laplacians} for an illustration. An in-depth discussion of adiabatic limits is beyond the scope of this paper, and we refer interested readers to \cite{LiuZhang1999Adiabatic,Bismut2013} and references therein.

\begin{corollary}
\label{cor:warped_diffusion_generator}
  Under the same assumptions as in {\rm Theorem~\ref{thm:warped_diffusion_generator}}, if the limit of the ratio $\delta/\epsilon$ exists and is finite, i.e.,
  \begin{equation*}
    \beta:=\lim_{\epsilon\rightarrow 0} \delta / \epsilon<\infty,
  \end{equation*}
then
\begin{equation}
  \label{eq:constant_ratio_warped_diffusion_generator}
  \begin{aligned}
    \lim_{\epsilon\rightarrow 0}\frac{H_{\epsilon,\delta}^{\left( \alpha \right)}f \left( x,v \right)-f \left( x,v \right)}{\epsilon} = \frac{1}{2}\frac{\left[L_{\beta}\left( fp^{1-\alpha} \right)-fL_{\beta}p^{1-\alpha} \right]\left( x,v \right)}{p^{1-\alpha}\left( x,v \right)}
  \end{aligned}
\end{equation}
where $L_{\beta}$ is a second order partial differential operator on $E$ given by
\begin{equation}
  \label{eq:warped_laplacian}
  L_{\beta}=\frac{m_{21}}{m_0}\Delta_H+\beta \frac{m_{22}}{m_0}\Delta_E^V.
\end{equation}
In particular, if $m_{21}=\beta m_{22}$ and $\pi:E\rightarrow M$ is a harmonic map, then $L=c\Delta_E$ where $\Delta_E$ is the Laplace-Beltrami operator on $E$ and $c$ a multiplicative constant. In addition, if $\alpha=1$, then
\begin{equation*}
  \lim_{\epsilon\rightarrow 0}\frac{H_{\epsilon,\delta}^{\left( 1 \right)}f \left( x,v \right)-f \left( x,v \right)}{\epsilon} = \frac{c}{2}\Delta_Ef \left( x,v \right).
\end{equation*}
\end{corollary}
\begin{proof}
  If $\pi:E\rightarrow M$ is a harmonic map, the fibres of $\pi$ are minimal submanifolds of $E$ (\emph{vice versa}, see e.g.~\cite[Lemma 2.2.4]{GLP1999}) and $\Delta_H=\Delta_E^H$ (see Remark~\ref{rem:different_horizontal_laplacian}).
\end{proof}

\begin{remark}
  \label{rem:hdm-nontrivial}
  Corollary~\ref{cor:warped_diffusion_generator} clearly indicates that the coupled diffusion operator $H_{\epsilon,\delta}^{\left( \alpha \right)}$ differs from the anisotropic diffusion operators considered in \cite{LafonThesis2004,CoifmanLafon2006} and the dynamical system literature \cite{Diannakis2015,ZG2016} in an essential way: in general, when the fires are not totally geodesic submanifolds of the fibre bundle, the infinitesimal generator \eqref{eq:warped_laplacian} will never equal to the Laplace-Beltrami operator of the total manifold, regardless of the relative ratio between $\delta$ and $\epsilon$ --- even when the two constants in front of $\Delta_H$ and $\Delta_E^V$ coincide. This is essentially due to the difference between the rough horizontal Laplacian $\Delta_H$ and the \emph{bona fide} ``horizontal Laplacian'' $\Delta_E^H$ commonly encountered in sub-Riemannian geometry and Riemannian submersions; see Appendix~\ref{rem:different_horizontal_laplacian} for more details. The HDM framework is thus by no means a straightforward application of the anisotropic diffusion maps to the total manifold of the fibre bundle.
\end{remark}


In order to state the result for the case when the ratio $\delta/\epsilon$ is not asymptotically bounded as $\epsilon\rightarrow0$, let us define the \emph{fibre average} of any function $f\in C^{\infty}\left( E \right)$ as
\begin{equation}
  \label{eq:fibre_average}
  \langle f\rangle \left( x \right)=\int_{F_x}f \left( x,v \right)\,d\mathrm{vol}_{F_x} \left( v \right)
\end{equation}
whenever the integral converges. If $\langle f\rangle \left( x \right)$ exists for all $x\in M$ (e.g. when the fibre is compact or $f$ is integrable), obviously $\langle f\rangle\in C^{\infty}\left( M \right)$.

Consider now the probability density function $p$ in the definition of $H_{\epsilon,\delta}^{\left( \alpha \right)}$. The fibre average $\langle p\rangle$ is a probability density function on $M$, since
\begin{equation*}
  \int_M\langle p\rangle \left( x \right)\,d\mathrm{vol}_M \left( x \right)=\int_M\!\!\int_{F_x}p \left( x,v \right)\,d\mathrm{vol}_{F_x} \left( v \right)d\mathrm{vol}_M \left( x \right) = 1.
\end{equation*}
Note that $\langle p\rangle$ is bounded away from $0$ from below according to our assumption~\eqref{eq:density_positivity_assumption}. We can thus divide $p$ by $\langle p\rangle$ and define the \emph{conditional probability density function} on $E$ as
\begin{equation}
  \label{eq:conditional_pdf}
  p \left( v\mid x \right):= \frac{p \left( x,v \right)}{\langle p \rangle \left( x \right)}.
\end{equation}
The name comes from the observation that $p \left( v\mid x \right)$ defines a probability density function when restricted to a single fibre:
\begin{equation*}
  \int_{F_x}p\left( v\mid x \right)\,d\mathrm{vol}_{F_x}\left( v \right) = \frac{\displaystyle\int_{F_x}p \left( x,v \right)\,d\mathrm{vol}_{F_x} \left( v \right)}{\langle p \rangle \left( x \right)}=1.
\end{equation*}
The last piece of notation we need for Theorem~\ref{thm:warped_diffusion_generator_unbounded_ratio} is
\begin{equation}
  \label{eq:fibre_average_wrt_density}
  \langle f\rangle_p \left( x \right) := \int_{F_x}f \left( x,v \right)p \left( v\mid x \right)\,d\mathrm{vol}_{F_x}\left( v \right),
\end{equation}
for any function $f\in C^{\infty}\left( E \right)$. We shall refer to $\langle f\rangle_p$ as the \emph{fibre average of $f$ with respect to the probability density function $p$}.

\begin{theorem}[Unbounded Ratio $\delta/\epsilon$]
\label{thm:warped_diffusion_generator_unbounded_ratio}
  Suppose $\mathscr{E}=\left(E,M,F,\pi\right)$ is a fibre bundle, $M$ is a smooth Riemannian manifold without boundary, and $E$ is equipped with the Riemannian metric \eqref{eq:metric_on_total_manifold}. Define $\gamma:=\delta/\epsilon$ (equivalently $\delta=\gamma\epsilon$). For any $f\in C^{\infty}\left( E \right)$ and $\left( x,v \right)\in E$, as $\epsilon\rightarrow0$,
  \begin{equation}
    \label{eq:warped_generator_thm_unbounded_ratio}
    \begin{aligned}
      \lim_{\gamma\rightarrow\infty}&H_{\epsilon,\gamma\epsilon}^{\left( \alpha \right)}  f \left( x,v \right)\\
      & = \langle f\rangle_p \left( x \right) + \epsilon \frac{m_{2}'}{2m_0'} \frac{\left[\Delta_M \left( \langle f\rangle_p\langle p\rangle^{1-\alpha} \right)-\langle f\rangle_p\Delta_M\langle p\rangle^{1-\alpha} \right]\left( x \right)}{\langle p\rangle^{1-\alpha}\left( x \right)}+O \left( \epsilon^2 \right),
    \end{aligned}
  \end{equation}
where $m_0'$, $m_2'$ are positive constants depending only on the base manifold $M$ and the kernel $K$, $\Delta_M$ is the Laplace-Beltrami operator on $M$, $\langle p\rangle$ is the fibre average of the probability density function $p$, and $\langle f\rangle_p$ is the fibre average of $f$ with respect to the density $p$. In particular, if $\alpha=1$, then
\begin{equation*}
  \lim_{\gamma\rightarrow\infty}H_{\epsilon,\gamma\epsilon}^{\left( 1 \right)}=\langle f\rangle_p \left( x \right) + \epsilon \frac{m_{2}'}{2m_0'} \Delta_M\langle f\rangle_p+O \left( \epsilon^2 \right).
\end{equation*}
\end{theorem}
The proof of Theorem~\ref{thm:warped_diffusion_generator_unbounded_ratio} can be found in \ref{sec:proof_biase}. Intuitively, Theorem~\ref{thm:warped_diffusion_generator_unbounded_ratio} states that if the vertical bandwidth parameter $\delta\rightarrow\infty$ then the coupled diffusion operator contains little information about the fibres.
Comparing Theorem~\ref{thm:warped_diffusion_generator_unbounded_ratio} with Theorem~\ref{thm:warped_diffusion_generator}, one can see that in general
\begin{equation*}
  \lim_{\epsilon\rightarrow 0}\lim_{\gamma\rightarrow\infty} \frac{H_{\epsilon,\gamma\epsilon}^{\left( \alpha \right)}f \left( x,v \right)-f \left( x,v \right)}{\epsilon}\neq \lim_{\gamma\rightarrow\infty}\lim_{\epsilon\rightarrow 0} \frac{H_{\epsilon,\gamma\epsilon}^{\left( \alpha \right)}f \left( x,v \right)- f \left( x,v \right)}{\epsilon},
\end{equation*}
thus an asymptotic expansion of $H_{\epsilon,\delta}^{\left( \alpha \right)}f \left( x,v \right)$ for small $\epsilon,\delta$ is not well-defined without careful consideration of the behavior of $\delta/\epsilon$ if it is not asymptotically bounded.

\begin{remark}
  The subtlety in the characterization of the infinitesimal generator $H_{\epsilon,\delta}^{\left( \alpha \right)}$ speaks of the peculiarity of the nonhomogeneous, anisotropic diffusion processes considered in \cite{CoifmanLafon2006,SEKC2009}, at the presence of an underlying fibre bundle structure. These phenomena not only indicate that the horizontal and coupled diffusion operators are capable of unveiling richer geometric structures in complex real world data sets, but also imply that additional care has to be taken when tuning the bandwidth parameters in practice --- the flexibility in choosing the approriate relative scale between $\delta$ and $\epsilon$ adapts the HDM framework to a myriad of scenarios in which the relative importance of the structural information in the data objects vary drastially. The dependence of the infinitesimal generators on the ratio $\delta/\epsilon$ is also reminiscent of recent trends of studying ``big data'' in high-dimensional statistics \cite{BV2011,RH2017}, where new paradigms of estimation and inference arise as the ratio between the number of features and the number of samples becomes unbounded asymptotically.
\end{remark}

\section{Finite Sampling Results on Unit Tangent Bundles}
\label{sec:finite-sampl-results}

The algorithm and theoretical results discussed so far are very general --- we assumed that the diffusion kernel \eqref{eq:warped_kernel} is constructed from abstract, geodesic distances on the base and fibre manifolds. This section investigates the finite sampling aspects of horizontal diffusion maps, which connects the discrete, graph construction in Section~\ref{sec:algorithm} with the continuous, infinitesimal characterization in Section~\ref{sec:infin-gener}. We focus on analyzing the finite sample rate of convergence for \emph{unit tangent bundles}, the fibre bundle with compact fibres that is as prevalent as manifolds. This is a subbundle of the tangent bundle $T\!M$ (which is non-compact) defined as 
\begin{align*}
  UT\!M := \coprod_{x\in M}S_x,\quad S_x:=\left\{ v\in T_xM\mid g_x \left( v,v \right)=1  \right\} \subset T_xM.
\end{align*}
In particular, $UT\!M$ is a hypersurface of $TM$ equipped with a metric induced from $T\!M$. The volume form on $UT\!M$ with respect to the induced metric
\begin{align*}
  d\Theta \left( x,v \right)=d\mathrm{vol}_{S_x}\!\!\left( v \right)d\mathrm{vol}_M\!\!\left( x \right).
\end{align*}
is often known as the \emph{Liouville measure} or the \emph{kinematic density}~\cite[Chapter VII]{Chavel2006}. It is the only invariant measure on $UT\!M$ under geodesic flows. The coupled diffusion operator on $UT\!M$ can be written with the Liouville measure:
\begin{align*}
  H_{\epsilon,\delta}^{\left(\alpha\right)}f \left( x,v \right)=\frac{\displaystyle\int_{UTM}K_{\epsilon,\delta}^{\left(\alpha\right)}\left( x,v;y,w \right)f \left( y,w \right)p \left( y,w \right)\,d\Theta \left( y,w \right)}{\displaystyle\int_{UTM}K_{\epsilon,\delta}^{\left(\alpha\right)}\left( x,v;y,w \right)p \left( y,w \right)\,d\Theta \left( y,w \right)},\quad\forall f\in C^{\infty}\left( UT\!M \right).
\end{align*}
The horizontal and vertical Laplacians on $UT\!M$ can be defined from $\Delta_{T\!M}^H$ and $\Delta_{T\!M}^V$ by extending $f\in C^{\infty}\left( UT\!M \right)$ to $C^{\infty}\left( T\!M \right)$ and restricting the result back to $U\!TM$. Therefore, for any $f\in C^{\infty}\left( T\!M \right)$, if $\delta=O \left( \epsilon \right)$,
\begin{align*}
  H_{\epsilon,\delta}^{\left(\alpha\right)} & f\left( x,v \right) = f \left( x,v \right)+\epsilon \frac{m_{21}}{2m_0}\frac{\left[\Delta_{UT\!M}^H\left(fp^{1-\alpha}\right)-f\Delta_{UT\!M}^Hp^{1-\alpha}\right]\left( x,v \right)}{p^{1-\alpha}\left( x,v \right)}\\
    &+\delta \frac{m_{22}}{2m_0}\frac{\left[\Delta_{UT\!M}^V\left(fp^{1-\alpha}\right)-f\Delta_{UT\!M}^Vp^{1-\alpha}\right]\left( x,v \right)}{p^{1-\alpha}\left( x,v \right)}+O\left(\epsilon^2+\epsilon\delta+\delta^2 \right).
\end{align*}
This is consistent with the conclusion obtained in \cite[Chapter 3]{Gao2015Thesis}.

The theory of HDM on tangent and unit tangent bundles are parallel to each other, but a general theory for sampling from fibre bundles of arbitrary fibre type will find it easier to consider sampling from the unit tangent bundle due to the compactness of its fibres. Sampling from tangent bundles is special, since its fibres are vector spaces and thus determined by estimating a basis; this is considered in \cite[\S5]{SingerWu2012VDM}. We thus study the behavior of HDM on unit tangent bundles under finite sampling. In this section, we first consider sampling without noise, i.e. where we sample exactly on unit tangent bundles; next, we study the case where the tangent spaces are empirically estimated from samples on the base manifold. The latter scenario is a proof-of-concept for applying HDM to general fibre bundles in practical situations where data representing each fibre are often acquired with noise. The proofs of Theorem~\ref{thm:utm_finite_sampling_noiseless} and Theorem~\ref{thm:utm_finite_sampling_noise} can be found in Appendix~\ref{sec:proof_variance}. In Section~\ref{sec:unit-tangent-bundles}, we shall demonstrate a numerical experiment on $\mathrm{SO(3)}$ (the unit tangent bundle of the $2$-sphere in $\mathbb{R}^3$) that addresses the two sampling strategies. Throughout this section, recall from Remark~\ref{rem:different_horizontal_laplacian} that $\Delta_H=\Delta^H_{UT\!M}$ since the fibres of $UT\!M$ are totally geodesic.

\subsection{Rate of Convergence from Finite Samples}
\label{sec:conv-rate-from}

\subsubsection{Sampling without Noise}
\label{sec:sampl-with-noise}

We begin with some assumptions and definitions. Assumption~\ref{assum:technical} includes our technical assumptions, and Assumption~\ref{assum:utm_finite_sampling_noiseless} specifies the noiseless sampling strategy.
\begin{assumption}
\label{assum:technical}
\begin{enumerate}[(1)]
\item\label{item:13} $\iota: M\hookrightarrow \mathbb{R}^D$ is an isometric embedding of a $d$-dimensional closed Riemannian manifold into $\mathbb{R}^D$, with $D\gg d$.

\item\label{item:20} Let the bi-variate smooth kernel function $K:\mathbb{R}^2\rightarrow\mathbb{R}^{\geq0}$ be compactly supported within the unit square $\left[ 0,1 \right]\times \left[ 0,1 \right]$. The partial derivatives $\partial_1K$, $\partial_2K$ are therefore automatically compactly supported on the unit square as well. (In fact, a similar result still holds if $K$ and its first order derivatives decay faster at infinity than any inverse polynomials; to avoid technicalities and focus on demonstrating the idea, we use compactly supported $K$.)
\end{enumerate}
\end{assumption}

\begin{assumption}
\label{assum:utm_finite_sampling_noiseless}
The $\left( N_B\times N_F \right)$ data points
  \begin{equation*}
    \begin{matrix}
      x_{1,1}, &x_{1,2}, &\cdots, &x_{1,N_F}\\
      x_{2,1}, &x_{2,2}, &\cdots, &x_{2,N_F}\\
      \vdots & \vdots &\cdots &\vdots\\
      x_{N_B,1}, &x_{N_B,2}, &\cdots, &x_{N_B,N_F}
    \end{matrix}
  \end{equation*}
are sampled from $UTM$ with respect to a probability density function $p \left( x,v \right)$ satisfying \eqref{eq:density_positivity_assumption}, following a two-step strategy: (i) sample $N_B$ points $\xi_1,\cdots,\xi_{N_B}$ i.i.d. on $M$ with respect to $\langle p\rangle$, the fibre average of $p$ on $M$; (ii) sample $N_F$ points $x_{j,1},\cdots,x_{j,N_F}$ on $S_{\xi_j}$ with respect to $p \left(\cdot\mid\xi_j\right)$, the conditional probability density.
\end{assumption}

\begin{definition}
\label{defn:utm_finite_sampling_noiseless}
\begin{enumerate}[(1)]
\item\label{item:34} For $\epsilon>0$, $\delta>0$ and $1\leq i,j\leq N_B$, $1\leq r,s\leq N_F$, define
  \begin{equation*}
    \hat{K}_{\epsilon,\delta} \left( x_{i,r}, x_{j,s} \right) =
   \begin{cases}
    \displaystyle K \left( \frac{\|\xi_i-\xi_j\|^2}{\epsilon}, \frac{\|P_{\xi_j,\xi_i}x_{i,r}-x_{j,s}\|^2}{\delta} \right), & i\neq j,\\
   0,&i=j.
    \end{cases}
  \end{equation*} 
where $P_{\xi_j,\xi_i}:S_{\xi_i}\rightarrow S_{\xi_j}$ is the parallel transport from $S_{\xi_i}$ to $S_{\xi_j}$. Note the difference between $\hat{K}_{\epsilon,\delta}$ and $K_{\epsilon,\delta}$ defined in~\eqref{eq:warped_kernel}: $\hat{K}_{\epsilon,\delta}$ uses Euclidean distance while $K_{\epsilon,\delta}$ uses geodesic distance.
\item\label{item:35} For $0\leq \alpha\leq 1$, define
\begin{equation*}
    \hat{p}_{\epsilon,\delta}\left( x_{i,r} \right) = \sum_{j=1}^{N_B}\sum_{s=1}^{N_F}\hat{K}_{\epsilon,\delta} \left( x_{i,r}, x_{j,s} \right)
  \end{equation*}
and the \emph{empirical $\alpha$-normalized kernel} $\hat{K}_{\epsilon,\delta}^{\alpha}$
  \begin{equation*}
    \hat{K}_{\epsilon,\delta}^{\alpha} \left( x_{i,r}, x_{j,s} \right) = \frac{\hat{K}_{\epsilon,\delta}\left( x_{i,r}, x_{j,s} \right)}{\hat{p}^{\alpha}_{\epsilon,\delta}\left( x_{i,r}\right)\hat{p}^{\alpha}_{\epsilon,\delta}\left( x_{j,s} \right)},\quad 1\leq i, j\leq N_B, 1\leq r,s\leq N_F.
  \end{equation*}
\item\label{item:36} For $0\leq \alpha\leq 1$ and $f\in C^{\infty}\left( UTM \right)$, denote the \emph{$\alpha$-normalized empirical horizontal diffusion operator} by
  \begin{equation*}
    \hat{H}_{\epsilon,\delta}^{\alpha}f\left( x_{i,r} \right)=\frac{\displaystyle \sum_{j=1}^{N_B}\sum_{s=1}^{N_F}\hat{K}_{\epsilon,\delta}^{\alpha} \left( x_{i,r}, x_{j,s}\right)f \left( x_{j,s} \right)}{\displaystyle \sum_{j=1}^{N_B}\sum_{s=1}^{N_F}\hat{K}_{\epsilon,\delta}^{\alpha} \left( x_{i,r}, x_{j,s}\right)}.
  \end{equation*}
\end{enumerate}
\end{definition}

\begin{theorem}[Finite Sampling without Noise]
\label{thm:utm_finite_sampling_noiseless}
  Under Assumption~{\rm\ref{assum:technical}} and Assumption~{\rm\ref{assum:utm_finite_sampling_noiseless}}, if
\begin{enumerate}[(i)]
\item\label{item:41} $\delta=O \left( \epsilon \right)$ as $\epsilon\rightarrow0$;
\item\label{item:42}
\begin{equation*}
  \lim_{N_B\rightarrow\infty\atop N_F\rightarrow\infty}\frac{N_F}{N_B}=\beta\in \left( 0,\infty \right),
\end{equation*}
\end{enumerate}
then for any $x_{i,r}$ with $1\leq i\leq N_B$ and $1\leq r\leq N_F$, as $\epsilon\rightarrow0$ (and thus $\delta\rightarrow0$), with high probability
\begin{equation}
\label{eq:finite_sample_theorem}
\begin{aligned}
    \hat{H}_{\epsilon,\delta}^{\alpha}f\left( x_{i,r} \right) &= f \left( x_{i,r} \right)+\epsilon \frac{m_{21}}{2m_0}\left[\frac{\Delta_H\left[fp^{1-\alpha}\right]\left( x_{i,r} \right)}{p^{1-\alpha}\left( x_{i,r} \right)}-f \left( x_{i,r} \right) \frac{\Delta_Hp^{1-\alpha}\left( x_{i,r} \right)}{p^{1-\alpha}\left( x_{i,r} \right)}  \right]\\
    &+\delta \frac{m_{22}}{2m_0}\left[\frac{\Delta_{UTM}^V\left[fp^{1-\alpha}\right]\left( x_{i,r} \right)}{p^{1-\alpha}\left( x_{i,r} \right)}-f \left( x_{i,r} \right) \frac{\Delta_{UTM}^Vp^{1-\alpha}\left( x_{i,r} \right)}{p^{1-\alpha}\left( x_{i,r} \right)}  \right]\\
    &+O \left( \epsilon^2+\delta^2+\theta_{*}^{-1}N_B^{-\frac{1}{2}}\epsilon^{-\frac{d}{4}} \right),
\end{aligned}
\end{equation}
where
\begin{equation*}
  \theta_{*}=1-\frac{\displaystyle 1}{\displaystyle 1+\epsilon^{\frac{d}{4}}\delta^{\frac{d-1}{4}}\sqrt{\frac{N_F}{N_B}}}.
\end{equation*}
\end{theorem}
The proof of Theorem~\ref{thm:utm_finite_sampling_noiseless} is deferred to Appendix~\ref{sec:proof_variance}.

\begin{remark}
\label{rem:hdm-nontrivial-sampling}
  Theorem~\ref{thm:utm_finite_sampling_noiseless} reflects the difference in the finite-sample rate of convergence between considering horizontal diffusion and standard diffusion on the total manifold of the fibre bundle. For instance, in the special case $\epsilon=\delta$, by \cite{Singer2006ConvergenceRate}, the variance error associated with the standard diffusion maps on the total manifold is $O \left( N_F ^{-1/2} N_B ^{-1/2} \epsilon^{-\left( 2d-1 \right)/4} \right)$, while the variance error in\eqref{eq:finite_sample_theorem} is $O\left(\theta_{*}^{-1}N_B^{-1/2}\epsilon^{-d/4}\right)$. This is another evidence demonstrating the difference between horizontal diffusion maps and standard diffusion maps on the total manifold of the fibre bundle; see also Remark~\ref{rem:hdm-nontrivial}.
\end{remark}

\subsubsection{Sampling from Empirical Tangent Spaces}
\label{sec:sampl-estim-tang}

In practice, it has been shown in~\cite{SingerWu2012VDM} that, under the manifold assumption, a local PCA procedure can be used for estimating tangent spaces from a point cloud; we are using PCA here as a procedure that determines the dimension of a local good linear approximation to the manifold, and also, conveniently, provides a good basis, which can be viewed as a basis for each tangent plane. To sample on these tangent spaces, it suffices to repeatedly sample coordinate coefficients from a fixed standard unit sphere; each sample can be interpreted as giving the coordinates of a point (approximately) on the tangent space. Parallel-transports will take the corresponding point that truly lies on the tangent space at $\xi$ to the tangent space at $\zeta$, another point on the manifold. This new tangent space is, however, again known only approximately; points in this approximate space are characterized by coordinates with respect to the local PCA basis at $\zeta$. We can thus express the whole (approximate) parallel-transport procedure by maps between coordinates with respect to PCA basis at $\xi$ to sets of coordinates at $\zeta$; these changes of coordinates incorporate information on the choices of basis at each end as well as on the parallel-transport itself.

Let us now describe this in more detail, setting up notations along the way. Throughout this section, Assumption~\ref{assum:technical} still holds. Let $\left\{ \xi_1,\cdots,\xi_{N_B} \right\}$ be a collection of i.i.d. samples from $M$; then the local PCA procedure can be summarized as follows: for any $\xi_j$, $1\leq j\leq N_B$, let $\xi_{j_1},\cdots,\xi_{j_k}$ be its $k$ nearest neighboring points. Then
\begin{equation*}
  X_j=\left[ \xi_{j_1}-\xi_j,\cdots, \xi_{j_k}-\xi_j\right]
\end{equation*}
is a $D\times k$ matrix. Let $K_{\mathrm{PCA}}$ be a positive monotonic decreasing function supported on  the unit interval, e.g. the \emph{Epanechnikov kernel}~\cite{Epanechnikov1969}
\begin{equation*}
  K_{\mathrm{PCA}}\left( u \right)=\left( 1-u^2 \right)\chi_{\left[ 0,1 \right]},
\end{equation*}
where $\chi$ is the indicator function. Fix a scale parameter $\epsilon_{\mathrm{PCA}}>0$, let $D_j$ be the $k\times k$ diagonal matrix
\begin{equation*}
  D_j = \mathrm{diag}\left( \sqrt{K_{\mathrm{PCA}}\left( \frac{\left\|\xi_j-\xi_{j_1}\right\|}{\sqrt{\epsilon_{\mathrm{PCA}}}} \right)}, \cdots, \sqrt{K_{\mathrm{PCA}}\left( \frac{\left\|\xi_j-\xi_{j_k}\right\|}{\sqrt{\epsilon_{\mathrm{PCA}}}} \right)} \right)
\end{equation*}
and carry out the singular value decomposition (SVD) of matrix $X_jD_j$ as
\begin{equation*}
  X_jD_j=U_j\Sigma_jV_j^{\top}.
\end{equation*}
An estimated basis $B_j$ for the local tangent plane at $\xi_j$ is formed by the first $d$ left singular vectors (corresponding to the $d$ largest singular values in $\Sigma_j$), arranged into a matrix as follows:
\begin{equation*}
  B_j=\left[ u_j^{\left( 1 \right)},\cdots, u_j^{\left( d \right)} \right]\in \mathbb{R}^{D\times d}.
\end{equation*}
Note that the intrinsic dimension $d$ is generally not known \emph{a priori}. The authors of \cite{SingerWu2012VDM} proposed a procedure that first estimates local dimensions from the decay of singular values in $\Sigma_j$ and then sets $d$ to be the median of all local dimensions; \cite{LittleMaggioniRosasco2011} proposed a different approach based on multi-scale singular value decomposition.

Once a pair of estimated bases $B_i,B_j$ is obtained for neighboring points $\xi_i,\xi_j$, one estimates a parallel-transport from $T_{\xi_i}M$ to $T_{\xi_j}M$ as
\begin{equation*}
  O_{ji}:=\argmin_{O\in O \left( d \right)}\left\|O-B_j^{\top}B_i\right\|_{\mathrm{HS}},
\end{equation*}
where $\left\|\cdot\right\|_{\mathrm{HS}}$ is the Hilbert-Schmidt norm. Though this minimization problem is non-convex, it has an efficient closed-form solution via the SVD of $B_i^{\top}B_j$, namely
\begin{equation*}
  O_{ji}=UV^{\top},\quad\textrm{where }B_j^{\top}B_i=U\Sigma V^{\top} \textrm{ is the SVD of $B_j^{\top}B_i$.}
\end{equation*}
It is worth noting that $O_{ji}$ depends on the bases; it operates on the coordinates of tangent vectors under $B_i$ and $B_j$, as explained above. $O_{ji}$ approximates the true parallel-transport $P_{\xi_j,\xi_i}$ (composed with the bases-expansions) with an error of $O \left( \epsilon_{\mathrm{PCA}} \right)$, in the sense of \cite[Lemma B.1]{SingerWu2012VDM}.

We summarize our sampling strategy for this section (with some new notations) in the following definition.
\begin{definition}
\label{defn:utm_finite_sampling_noise}
\begin{enumerate}[(1)]
\item\label{item:31} Let $\left\{ \xi_1,\cdots,\xi_{N_B} \right\}$ be a collection of samples from the base manifold $M$, i.i.d. with respect to some probability density function $\overline{p}\in C^{\infty}\left( M \right)$. For each $\xi_j$, $1\leq j\leq N_B$, sample $N_F$ points uniformly from the $\left(d-1\right)$-dimensional standard unit sphere $S^{d-1}$ in $\mathbb{R}^d$, and denote the set of samples as $\mathscr{C}_j = \left\{ c_{j,1},\cdots,c_{j,N_F} \right\}$, where each $c_{j,s}$ is a $d\times 1$ column vector. Using the basis $B_j$  estimated from the local PCA procedure, each $c_{j,s}$ corresponds to an ``approximate tangent vector at $\xi_j$'', denoted as
  \begin{equation*}
    \tau_{j,s}:=B_jc_{j,s}.
  \end{equation*}
We use the notation $\mathscr{S}_j$ for the unit sphere in the estimated tangent space (i.e., the column space of $B_j$). Note that the $\tau_{j,1},\cdots,\tau_{j,N_F}$ are uniformly distributed on $\mathscr{S}_j$.
\item\label{item:40} By \cite[lemma B.1]{SingerWu2012VDM}, for any $B_j$ there exists a $D\times d$ matrix $Q_j$, such that the columns of $Q_j$ constitutes an orthonormal basis for $\iota_{*}T_{\xi_j}M$ and
  \begin{equation*}
    \left\|B_j-Q_j\right\|_{\mathrm{HS}}=O \left( \epsilon_{\mathrm{PCA}} \right).
  \end{equation*}
We define the \emph{tangent projection} from $\iota_{*}S_{\xi_j}$ to the estimated tangent plane as
\begin{equation*}
  \tau_{j,s}\mapsto\overline{\tau}_{j,s}= \frac{Q_jQ_j^{\top}\tau_{j,s}}{\left\|Q_jQ_j^{\top}\tau_{j,s}\right\|}.
\end{equation*}
This map is well-defined for sufficiently small $\epsilon_{\mathrm{PCA}}$, and then it is an isometry. Its inverse is given by
\begin{equation*}
  \overline{\tau}_{j,s}\mapsto \tau_{j,s}=\frac{B_jB_j^{\top}\overline{\tau}_{j,s}}{\left\|B_jB_j^{\top}\overline{\tau}_{j,s}\right\|}.
\end{equation*}
Note that we have
\begin{equation*}
  \left\|\tau_{j,s}-\overline{\tau}_{j,s}\right\|\leq C\epsilon_{\mathrm{PCA}}
\end{equation*}
for some constant $C>0$ independent of indices $j,s$. Since we sample each $\mathscr{S}_j$ uniformly and the projection map $\tau_{j,s}\mapsto\bar{\tau}_{j,s}$ is an isometry, the points $\left\{ \overline{\tau}_{j,1},\cdots,\overline{\tau}_{j,N_F} \right\}$ are also uniformly distributed on $S_{\xi_j}$. The points
\begin{equation*}
  \begin{matrix}
    \overline{\tau}_{1,1}, &\overline{\tau}_{1,2}, &\cdots, &\overline{\tau}_{1,N_F}\\
    \overline{\tau}_{2,1}, &\overline{\tau}_{2,2}, &\cdots, &\overline{\tau}_{2,N_F}\\
    \vdots & \vdots &\cdots &\vdots\\
    \overline{\tau}_{N_B,1}, &\overline{\tau}_{N_B,2}, &\cdots, &\overline{\tau}_{N_B,N_F}
  \end{matrix}
\end{equation*}
are therefore distributed on $UTM$ according to a joint probability density function $p$ on $UTM$ defined as
\begin{equation*}
  p \left( x,v \right)=\overline{p}\left( x \right),\quad \forall \left( x,v \right)\in UTM.
\end{equation*}
As in Assumption~\ref{assum:utm_finite_sampling_noiseless}, we assume $p$ satisfies \eqref{eq:density_positivity_assumption}, i.e.,
\begin{equation*}
  0<p_m\leq p \left( x,v \right)=\overline{p}\left( x \right)\leq p_M <\infty,\quad \forall \left( x,v \right)\in UTM
\end{equation*}
for positive constants $p_m,p_M$.
\item\label{item:14} For $\epsilon>0$, $\delta>0$ and $1\leq i,j\leq N_B$, $1\leq r,s\leq N_F$, define
  \begin{equation*}
    \mathscr{K}_{\epsilon,\delta} \left( \overline{\tau}_{i,r}, \overline{\tau}_{j,s} \right) =
   \begin{cases}
    \displaystyle K \left( \frac{\|\xi_i-\xi_j\|^2}{\epsilon}, \frac{\|O_{ji}c_{i,r}-c_{j,s}\|^2}{\delta} \right), & i\neq j,\\
   0,&i=j.
    \end{cases}
  \end{equation*} 
where $O_{ji}$ is the estimated parallel-transport from $T_{\xi_i}M$ to $T_{\xi_j}M$.
\item\label{item:29} For $0\leq \alpha\leq 1$, define
\begin{equation*}
    \hat{q}_{\epsilon,\delta}\left( \overline{\tau}_{i,r} \right) = \sum_{j=1}^{N_B}\sum_{s=1}^{N_F}\mathscr{K}_{\epsilon,\delta} \left( \overline{\tau}_{i,r}, \overline{\tau}_{j,s} \right)
  \end{equation*}
and
  \begin{equation*}
    \mathscr{K}_{\epsilon,\delta}^{\alpha} \left( \overline{\tau}_{i,r}, \overline{\tau}_{j,s} \right) = \frac{\mathscr{K}_{\epsilon,\delta}\left( \overline{\tau}_{i,r}, \overline{\tau}_{j,s} \right)}{\hat{q}^{\alpha}_{\epsilon,\delta}\left( \overline{\tau}_{i,r}\right)\hat{q}^{\alpha}_{\epsilon,\delta}\left( \overline{\tau}_{j,s} \right)},\quad 1\leq i, j\leq N_B, 1\leq r,s\leq N_F.
  \end{equation*}
\item\label{item:30} For $0\leq \alpha\leq 1$ and $f\in C^{\infty}\left( UTM \right)$, denote
  \begin{equation*}
    \mathscr{H}_{\epsilon,\delta}^{\alpha}f\left( \overline{\tau}_{i,r} \right)=\frac{\displaystyle \sum_{j=1}^{N_B}\sum_{s=1}^{N_F}\mathscr{K}_{\epsilon,\delta}^{\alpha} \left( \overline{\tau}_{i,r}, \overline{\tau}_{j,s}\right)f \left( \overline{\tau}_{j,s} \right)}{\displaystyle \sum_{j=1}^{N_B}\sum_{s=1}^{N_F}\mathscr{K}_{\epsilon,\delta}^{\alpha} \left( \overline{\tau}_{i,r}, \overline{\tau}_{j,s}\right)}.
  \end{equation*}
\end{enumerate}
\end{definition}

\begin{theorem}[Finite Sampling from Empirical Tangent Planes]
\label{thm:utm_finite_sampling_noise}
  In addition to Assumption~{\rm\ref{assum:technical}}, suppose
\begin{enumerate}[(i)]
\item\label{item:33} $\epsilon_{\textrm{PCA}} = O \left( N_B^{-\frac{2}{d+2}} \right)$ as $N_B\rightarrow\infty$;
\item\label{item:32} As $\epsilon\rightarrow0$, $\delta=O \left( \epsilon \right)$ and $\delta\gg \left( \epsilon_{\mathrm{PCA}}^{\frac{1}{2}}+\epsilon^{\frac{3}{2}} \right)$;
\item\label{item:39}
\begin{equation*}
  \lim_{N_B\rightarrow\infty\atop N_F\rightarrow\infty}\frac{N_F}{N_B}=\beta\in \left( 0,\infty \right).
\end{equation*}
\end{enumerate}
Then for any $\tau_{i,r}$ with $1\leq i\leq N_B$ and $1\leq r\leq N_F$, as $\epsilon\rightarrow0$ (and thus $\delta\rightarrow0$), with high probability
\begin{equation}
\label{eq:finite_sample_theorem_noise}
\begin{aligned}
    \mathscr{H}_{\epsilon,\delta}^{\alpha}f\left( \overline{\tau}_{i,r} \right) &= f \left( \overline{\tau}_{i,r} \right)+\epsilon \frac{m_{21}}{2m_0}\left[\frac{\Delta_{UT\!M}^H\left[fp^{1-\alpha}\right]\left( \overline{\tau}_{i,r} \right)}{p^{1-\alpha}\left( \overline{\tau}_{i,r} \right)}-f \left( \overline{\tau}_{i,r} \right) \frac{\Delta_{UT\!M}^Hp^{1-\alpha}\left( \overline{\tau}_{i,r} \right)}{p^{1-\alpha}\left( \overline{\tau}_{i,r} \right)}  \right]\\
    &+\delta \frac{m_{22}}{2m_0}\left[\frac{\Delta_{UT\!M}^V\left[fp^{1-\alpha}\right]\left( \overline{\tau}_{i,r} \right)}{p^{1-\alpha}\left( \overline{\tau}_{i,r} \right)}-f \left( \overline{\tau}_{i,r} \right) \frac{\Delta_{UT\!M}^Vp^{1-\alpha}\left( \overline{\tau}_{i,r} \right)}{p^{1-\alpha}\left( \overline{\tau}_{i,r} \right)}  \right]\\
    &+O \left( \epsilon^2+\epsilon\delta+\delta^2+\theta_{*}^{-1}N_B^{-\frac{1}{2}}\epsilon^{-\frac{d}{4}}+\delta^{-1}\left( \epsilon_{\mathrm{PCA}}^{\frac{1}{2}}+\epsilon^{\frac{3}{2}} \right) \right),
\end{aligned}
\end{equation}
where
\begin{equation*}
  \theta_{*}=1-\frac{\displaystyle 1}{\displaystyle 1+\epsilon^{\frac{d}{4}}\delta^{\frac{d-1}{4}}\sqrt{\frac{N_F}{N_B}}}.
\end{equation*}
\end{theorem}
We give a proof of Theorem~\ref{thm:utm_finite_sampling_noise} in Appendix~\ref{sec:proof_variance}.

\subsection{Numerical Experiments}
\label{sec:unit-tangent-bundles}

The unit tangent bundle is of special interest since $UT\!M$ is a compact Riemannian manifold whenever $M$ is compact, enabling finite sampling and numerically validating Theorem~\ref{thm:warped_diffusion_generator} and Theorem~\ref{thm:warped_diffusion_generator_unbounded_ratio}. We present in below a numerical experiment on $\mathrm{SO(3)}$, the unit tangent bundle of the standard two-dimensional sphere in $\mathbb{R}^3$, along with an analysis of sampling errors on general unit tangent bundles. In the first step, we uniformly sample $N_B=2,000$ points $\left\{\xi_1,\cdots,\xi_{N_B} \right\}$ on the unit sphere $S^2$, and find for each sample point the $K_B=100$ nearest neighbors in the point cloud. Next, we sample $N_F=50$ vectors of unit length tangent to the unit sphere at each sample point (which in this case is a circle), thus collecting a total of $N_B\times N_F=100,000$ points on $UTS^2=\mathrm{SO(3)}$, denoted as
\begin{equation*}
\left\{ x_{j,s}\mid 1\leq j\leq N_B,1\leq s\leq N_F \right\}.
\end{equation*}
The horizontal diffusion matrix $H$ is then constructed as an $N_B\times N_B$ block matrix with block size $N_F\times N_F$, and $H_{ij}$ (the $\left( i,j \right)$-th block of $H$) is non-zero only if the sample points $\xi_i,\xi_j$ are each among the $K_B$-nearest neighbors of the other; when $H_{ij}$ is non-zero, its $\left( r,s \right)$-entry ($1\leq r,s\leq N_F$) is non-zero only if $P_{\xi_j,\xi_i}x_{i,r}$ and $x_{j,s}$ are each among the $K_F=50$ nearest neighbors of the other, and in that case for all $i\neq j$
\begin{equation}
\label{eq:matrix_entry_noiseless}
  H_{ij} \left( r,s \right)=\exp \left[ -\left(\frac{\left\|\xi_i-\xi_j\right\|^2}{\epsilon}+\frac{\left\|P_{\xi_j,\xi_i}x_{i,r}-x_{j,s}\right\|^2}{\delta}\right) \right],
\end{equation}
where the choices of $\epsilon,\delta$ will be explained below. The diagonal blocks are set to zero as in Definition~\ref{defn:utm_finite_sampling_noiseless}. Note that for the unit sphere $S^2$ the parallel-transport from $T_{\xi_i}S^2$ to $T_{\xi_j}S^2$ can be explicitly constructed as a rotation along the axis $\xi_i\times \xi_j$. Finally, we form the $\alpha$-normalized horizontal diffusion matrix $H_{\alpha}$ by
\begin{equation}
\label{eq:normalize_diffusion_matrix}
  \left(H_{\alpha}\right)_{ij}\left( r,s \right)=\frac{H_{ij} \left( r,s \right)}{\left( \displaystyle \sum_{l=1}^{N_B}\sum_{m=1}^{N_F}H_{il} \left( r,m \right) \right)^{\alpha}\left( \displaystyle \sum_{k=1}^{N_B}\sum_{n=1}^{ N_F}H_{jk} \left( r,n \right) \right)^{\alpha}},
\end{equation}
and solve the eigenvalue problem
\begin{equation}
\label{eq:generalized_eigen_problem}
  \left(D^{-\frac{1}{2}}H_{\alpha}D^{-\frac{1}{2}}\right)U = U\Lambda
\end{equation}
where $D$ is the $\left(N_BN_F\right)\times\left(N_BN_F\right)$ diagonal matrix with entry $\left( k,k \right)$ equal to the $k$-th column sum of $H_{\alpha}$:
\begin{equation*}
  D \left( k,k \right) = \sum_{v=1}^{N_B N_F}H_{\alpha}\left( k,v \right),
\end{equation*}
and $\Lambda$ is a diagonal matrix of the same dimensions. Throughout this experiment, we fix $\alpha=1$, $\epsilon=0.2$ and choose various values of $\delta$ ranging from $0.0005$ to $50$, and observe the spacing of the eigenvalues stored in $\Lambda$.

The purpose of this experiment is to investigate the influence of the ratio $\gamma=\delta/\epsilon$ on the spectral behavior of graph horizontal Laplacians. As shown in Figure~\ref{fig:laplacians}, the spacing in the spectrum of these graph horizontal Laplacians follow patterns similar to the multiplicities of the eigenvalues of corresponding Laplacians on $\mathrm{SO(3)}$ (governed by the relative size of $\delta$ and $\epsilon$). In Figure~\ref{fig:laplacians}(a), $\delta\ll\epsilon$, hence the graph horizontal Laplacian approximates the horizontal Laplacian on $\mathrm{SO(3)}$ (according to Theorem~\ref{thm:warped_diffusion_generator} and Corollary~\ref{cor:warped_diffusion_generator}), in which the smallest eigenvalues have  multiplicities $1,6,13,\cdots$; in Figure~\ref{fig:laplacians}(b), $\delta=O \left(\epsilon \right)$, hence the graph horizontal Laplacian approximates the total Laplacian on $\mathrm{SO(3)}$ (again, according to Theorem~\ref{thm:warped_diffusion_generator} and Corollary~\ref{cor:warped_diffusion_generator}), with eigenvalue multiplicities $1,9,25,\cdots$); in Figure~\ref{fig:laplacians}(c), $\delta\gg\epsilon$, hence the graph horizontal Laplacian approximates the Laplacian on the base manifold $S^2$ (according to Theorem~\ref{thm:warped_diffusion_generator_unbounded_ratio}), with eigenvalue multiplicities $1,3,5,\cdots$). Note that in Figure~\ref{fig:laplacians}(c) we fixed $\epsilon$ and pushed $\delta$ to $\infty$, which essentially corresponds to the limit process in \eqref{eq:warped_generator_thm_unbounded_ratio} rather than \eqref{eq:horizontal_generator_theorem}. Moreover, if in each figure we divide the sequence of eigenvalues by the smallest non-zero eigenvalue, the resulting sequence coincides with the list of eigenvalues of the corresponding manifold Laplacian up to numerical error. For a description of the spectrum of these partial differential operators, see \cite[Chapter 2]{Taylor1990NHA}.
\begin{figure}[htp]
  \centering
\begin{tabular}{ccc}
  \includegraphics[width=0.31\textwidth]{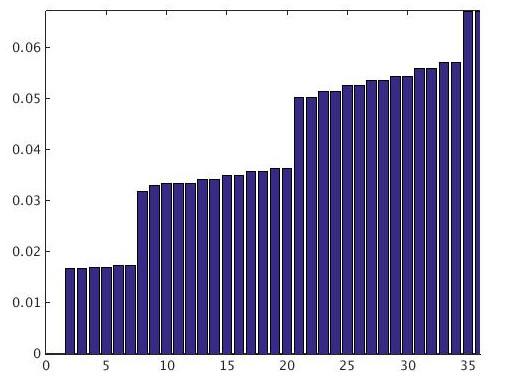}&
  \includegraphics[width=0.31\textwidth]{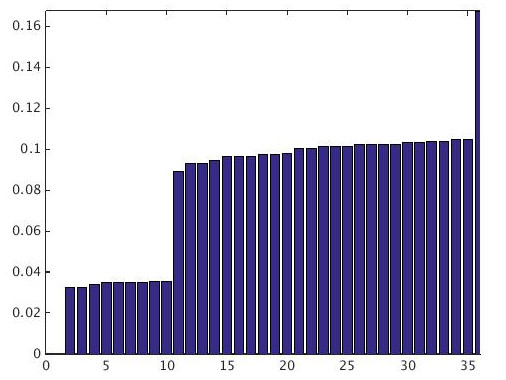}&
  \includegraphics[width=0.31\textwidth]{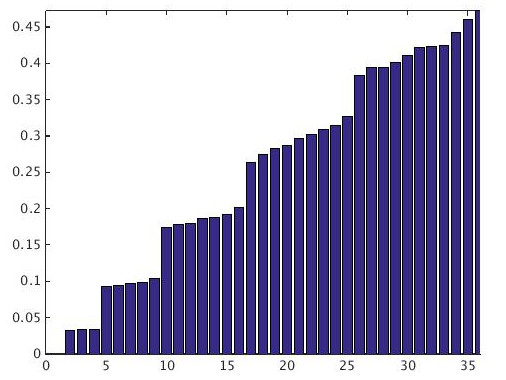}\\
  (a) $\delta=0.002, \Delta_{\mathrm{SO(3)}}^H$ & (b) $\delta=0.015, \Delta_{\mathrm{SO(3)}}$ & (c) $\delta=20, \Delta_{S^2}$\\
\end{tabular}
  \caption{Bar plots of the smallest $36$ eigenvalues of $I-D^{-1}H_{\alpha}$ with $\alpha=1, \epsilon=0.2$, and varying $\delta$ values (sampling without noise). \emph{Left}: When $\delta\ll \epsilon$, $H_{\epsilon,\delta}^1$ approximates the heat kernel of $\Delta_{\mathrm{SO(3)}}^H$, of which the multiplicities of largest eigenvalues are $1,6,13,\cdots$; \emph{Middle}: When $\delta\approx \epsilon$, $H_{\epsilon,\delta}^1$ approximates the heat kernel of $\Delta_{\mathrm{SO(3)}}$, of which the multiplicities of largest eigenvalues are $1,9,25,\cdots$; \emph{Right}: When $\delta\gg \epsilon$, $H_{\epsilon,\delta}^1$ approximates the heat kernel of $\Delta_{S^2}$, of which the multiplicities of largest eigenvalues are $1,3,5,\cdots$.}
  \label{fig:laplacians}
\end{figure}

Similar numerical results have been observed for sampling from empirically estimated tangent spaces; we refer interested readers to \cite[\S3.5.2]{Gao2015Thesis}.

\section{Application to Automated Geometric Morphometrics}
\label{sec:autogm}

The HDM framework can be applied to any data set with pairwise structural correspondences. In many applications, such structural correspondences are readily available through a registration procedure, and have been used to compute similarity scores or distances between objects of interest. In this section, we sketch the application of HDM to automated geometric morphometrics. In a nutshell, this is an unsupervised learning problem with heterogeneous or unorganized data, for which feature engineering is particularly difficult; moreover, it is hard to apply kernel methods due to the lack of an informative kernel function. We expect problems arising from machine learning, pattern recognition, and computer vision facing similar difficulties to benefit from the proposed approach.

Geometric morphometrics is the quantitative analysis of shape variation and their correlation with other traits for biological organisms. For instance, it is often of interest to geometric morphometricians to understand quantitatively the amount of the shape variation explained by geometric features within a collection of shapes. They typically select equal numbers of consistently \emph{homologous} landmark points on each surface \cite{Mitteroecker2009}, corresponding to a mental model of a latent ``template,'' of which every individual shape is an instantiation. In statistical shape analysis, this landmark-based approach is developed in the framework of \emph{Procrustes analysis} \cite{DrydenMardia1998SSA}. Obviously, such an analysis is limited by the knowledge of landmark placement. From a mathematical point of view, extracting a limited number of landmarks from a continuous surface inevitably loses geometric information, unless when the shapes under consideration are solely determined by the landmarks (e.g. polygonal shapes, as considered in \cite{Kendall1984}\cite{FryThesis1993}), which is rarely the case for geometric morphometricians in biology; from a practical point of view, the requirement that an equal number of landmarks must be chosen on each shape is sometimes unrealistic due to the complex evolutionary and developmental process. Manually placing landmarks on each shape among a large collection is also a tedious task, and the skill to perform it ``correctly'' typically requires years of professional training; even then the ``correctness'' or the number of landmarks one should fix for a collection of shapes can be subject to debate among experts. These difficulties are gradually and continuously being addressed by a recent trend that advocates \emph{automated} workflows to bypass the repetitive, laborious, and time-consuming process of manual landmark placement on large collections of 3D digitized anatomical surfaces (see \cite{CP13,PNAS2011,PuenteThesis2013,LipmanDaubechies2011,Boyer2012,LipmanPuenteDaubechies2013,Auto3dGM2015,KoehlHass2015,Gao2015Thesis,VMGBSB2017,HLB2016,GYDMB2018,GKD2019,GKBD2019} and references therein).

The digitized morphological data set contains hundreds of triangular meshes (see Figure~\ref{fig:mesh_data}) of diverse size, topology, and quality, each representing an anatomical surface reconstructed from MicroCT images.
\begin{figure}[h]
  \centering
  \includegraphics[width=0.3\textwidth]{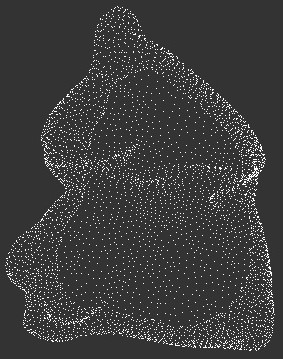}
  \includegraphics[width=0.3\textwidth]{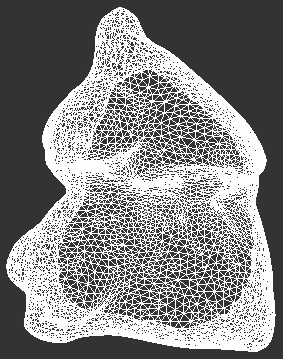}
  \includegraphics[width=0.3\textwidth]{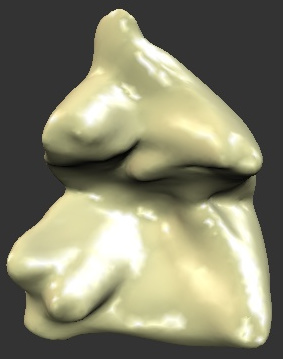}
  \caption{A second mandibular molar of a Philippine flying lemur (\emph{Cynocephalus volans}), represented as a point cloud (left), wireframe (middle), and a piecewise linear surface (right).}
  \label{fig:mesh_data}
\end{figure}
In \cite{CP13}, the authors introduced \emph{Continuous Procrustes Distance} (CPD) between surfaces with disk-type topology, and used conformal parameterization (uniformization) to design a fast algorithm that computes automatically (without landmarks) distances between pairs of morphological surfaces that would be at least as effective, for species discrimination, as Procrustes distances computed from user-defined landmarks \cite{PNAS2011}. Upon computing a distance between pairs of surfaces, the algorithm in \cite{CP13} minimize an energy functional depending on the pair, over an admissible set of \emph{correspondence maps}; the distance is indeed the value of the functional at the optimal correspondence map (Figure~\ref{fig:correspondence}). This approach has recently been followed by other authors as well \cite{Koehl2015}.

Detailed inspection of the optimal correspondence maps generated from the algorithm \cite{CP13} showed that some of them presented serious anomalies, such as reversed alignments of the anteroposterior/buccolingual axes \cite{GYDMB2018}. It may seem surprising that the algorithms, despite sometimes producing these erroneous maps, nevertheless were sufficiently successful in capturing sample geometry to achieve the success rate reported in \cite{PNAS2011}. As we extend the algorithm in \cite{CP13} in different directions, the correspondence maps became an important explicit goal of the algorithm, as opposed to an interesting by-product. While CPD automates the traditional Procrustes analysis, the optimal correspondence maps parallel the landmark-identification procedure performed mentally by geometric morphometricians. Moreover, these maps provide detailed information about correlations, often not fully retained when an energy functional summarizes a similarity measure, between functional or developmental regions on different shapes.

HDM is a natural algorithmic framework for unsupervised learning from structural correspondences maps. In this section, we apply HDM to a data set consisting of $50$ discretized triangular meshes of the second mandibular molar of prosimian primates and nonprimate close relatives. The $50$ meshes are evenly divided into $5$ genus groups: \emph{Alouatta}, \emph{Ateles}, \emph{Brachyteles}, \emph{Callicebus}, and \emph{Saimiri}; each mesh contains about $5,000$ vertices and $10,000$ faces. We compute first all pairwise CPD and correspondence maps, then all pairwise \emph{Horizontal Base Diffusion Distance} (HBDD) from the distances and maps. The $50\times 50$ distance matrices are finally embedded into $\mathbb{R}^3$ for comparison via \emph{multi-dimensional scaling} (MDS).

The HBDD is constructed from CPD as follows. For each pair of triangular meshes $S_i,S_j$ in the data set, denote their CPD as $d_{ij}$, and the optimal correspondence from $S_i$ to $S_j$ as $f_{ij}$. Note that $d_{ij}=d_{ji}$ and $f_{ij}=f_{ji}^{-1}$. In the first step, we discretize each surface area measure $\mu_j=d\mathrm{vol}_{S_j}$ into a linear combination of Dirac delta measures supported on vertices of $S_j$, where each vertex of $S_j$ is assigned $1/3$ of the surface area of its \emph{one-ring neighborhood}. We then \emph{soften} each bijective smooth map $f_{ij}$ into a \emph{transport plan} matrix $w_{ij}$, the $s$-th row of which records the transition probability from vertex $x_{i,s}$ of $S_i$ to each vertex on $S_j$; moreover, the specific softening we choose here allows each $x_{i,s}$ to jump (in one step) only to the three vertices of the unique\footnote{It is conceivable that $f_{ij}\left( x_{i,s} \right)$ could fall on the edge shared by two triangles in $S_j$, or even on a vertex of $S_j$ shared by more than $2$ triangles. While this rarely happens in practice, in our implementation for this application we resolve such conflicts by assigning $f_{ij}\left( x_{i,s} \right)$ randomly to any of the qualified triangles. This is because we express $f_{ij}\left( x_{i,s} \right)$ as a barycentric combination of the vertices of the triangle to which it is assigned, and thus the softening is in fact independent of the specific choice made.} triangular face on $S_j$ that contains $f_{ij} \left( x_{i,s} \right)$. If $x_{j,r}$ is a vertex on $S_j$ that can be reached from $x_{i,s}$ in one step of the random walk, we set the transition probability between $x_{i,s}$ and $x_{j,r}$ proportional to
\begin{equation*}
  \exp \left( -\frac{\left\|f_{ij}\left( x_{i,s} \right)-x_{j,r}\right\|^2}{\epsilon_F} \right),
\end{equation*}
where $\epsilon_F$ is a prescribed positive constant playing the role of the vertical bandwidth parameter $\delta$ in \eqref{eq:fibre_bundle_diffusion_operator}. For this specific data set, we choose $\epsilon_F=0.001$ which is the order of magnitude of the average distance between adjacent vertices on each mesh in the data set. Next, we construct the horizontal diffusion matrix $H$ as a $50\times 50$ block matrix, with block $\left( i,j \right)$
\begin{equation*}
  H \left( i,j \right)=
  \begin{cases}
    \displaystyle \exp \left( -\frac{d^2_{ij}}{\epsilon_B} \right)\cdot w_{ij}&\textrm{if $S_j$ is within the $N_B$-neighborhood of $S_i$,}\\
    0 &\textrm{otherwise.}
  \end{cases}
\end{equation*}
We chose for this data set $N_B=4$ and $\epsilon_B=0.03$. These parameters are picked empirically, where $0.03$ is usually the maximum CPD between surfaces that belong to the same species group. We then construct the normalized graph horizontal Laplacian $L_{\alpha,*}^H$ from $H$, as in \eqref{eq:hypoelliptic_graph_laplacian_normalized_alpha}, and solve for its largest $100$ eigenvalues and corresponding eigenvectors. From this eigen-decomposition we compute the \emph{horizontal base diffusion map} (HBDM) as in \eqref{eq:hbdm}, obtaining an embedding of the data set into $\mathbb{R}^{100 \choose 2}=\mathbb{R}^{4950}$. Though this embedding is still high dimensional, it is only $1/3$ of the original dimensionality (approximately $5000\times 3=15000$). The HBDD between each pair $S_i,S_j$ is then defined as the Euclidean distance between their images embedded in $\mathbb{R}^{4950}$, as in \eqref{eq:hbdd}. For comparison, we also embed the standard \emph{Diffusion Distance} matrix in to $\mathbb{R}^3$ using MDS. As shown in Figure~\ref{fig:MDS_comparison}, HBDD demonstrates the most clear pattern of species clusters among the three distances. It is even more interesting to notice that HBDD reflects the dietary groups within the data set (see Figure~\ref{fig:form_function}): folivores Alouatta (red) and Brachyteles (green) are adjacent to each other in the rightmost panel of Figure~\ref{fig:MDS_comparison}, so are frugivores Ateles (blue) and Callicebus (purple); the insectivore Saimiri (yellow) is far from the other herbivorous groups.
\begin{figure}[htp]
  \centering
  \includegraphics[width=.32\textwidth]{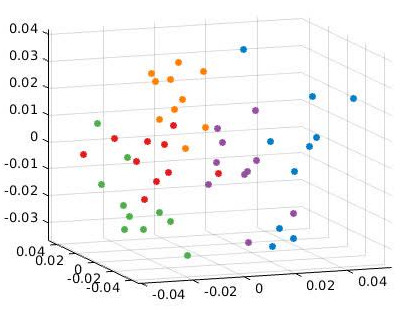}
  \includegraphics[width=.32\textwidth]{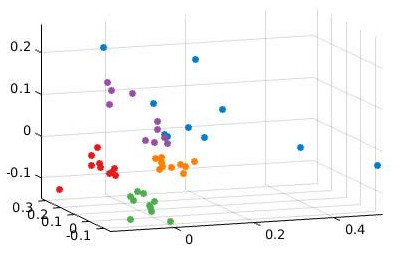}
  \includegraphics[width=.32\textwidth]{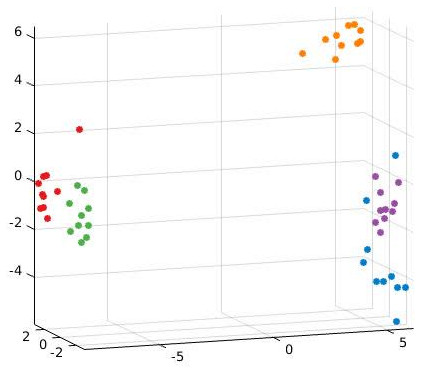}
  \caption{Embeddings of CPD (left), DM (middle), and HBDD (right) matrices into $\mathbb{R}^3$ using \emph{Multi-dimensional Scaling} (MDS).}
  \label{fig:MDS_comparison}
\end{figure}\vspace{-0.0in}
\begin{figure}[htp]
  \centering
  \includegraphics[width=.5\textwidth]{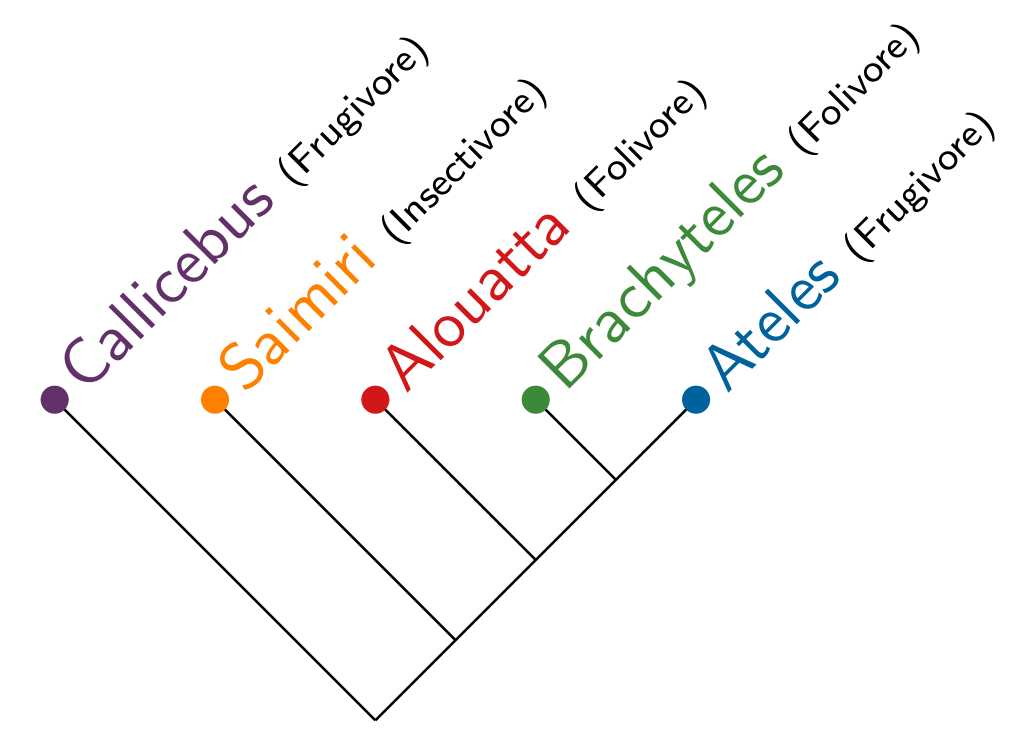}
  \caption{Phylogeny of the $5$ species groups \emph{Alouatta}, \emph{Ateles}, \emph{Brachyteles}, \emph{Callicebus}, and \emph{Saimiri}. HBDD (see Figure~\ref{fig:MDS_comparison}) reflects the dietary categories but not the phylogeny.}
  \label{fig:form_function}
\end{figure}\vspace{-0.1in}

For applications in geometric morphometrics, a major advantage of HDM over persistence-diagram-based methods is the morphological interpretability. This interpretability amounts to a globally consistent manner to identify corresponding regions on each shape in the data set and is potentially useful for subsequent studies of the evolutionary and developmental history. In standard morphologists' practice, such correspondences are assessed visually and manually; recent progress in techniques for generating and analyzing digital representations led to major advances~\cite{Zelditch2004,Wiley2005,PollyMacLeod2008} but still require the input of anatomical landmarks from the user. In contrast, by spectral clustering on the point cloud embedded into $\mathbb{R}^{100}$ by HDM, we can easily obtain a globally consistent segmentation for all surfaces, see Figure~\ref{fig:hdm_spectral_clustering}.
\begin{figure}[htp]
  \centering
  \includegraphics[width=1.0\textwidth]{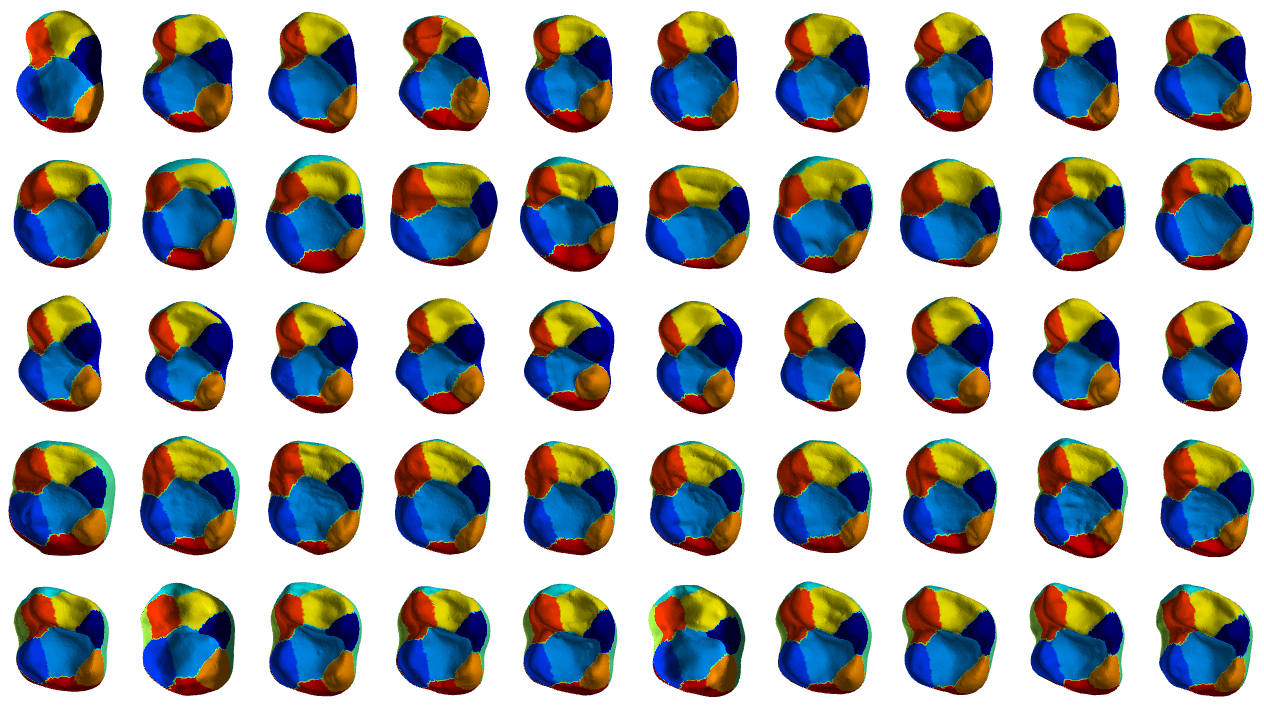}
  \caption{Automated landmarking: consistent segmentation of $50$ lemur teeth by spectral clustering in the Euclidean space to which HDM embeds. From the top row to the bottom row: Alouatta, Ateles, Brachyteles, Callicebus, Saimiri.}
  \label{fig:hdm_spectral_clustering}
\end{figure}

\section{Discussion and Future Work}
\label{sec:conlusion}
This paper introduced \emph{horizontal diffusion maps} (HDM), a novel semi-supervised learning framework for the analysis and organization of a class of complex data sets, in which individual structures at each data point carry abundant information that can not be easily abstracted away by a pairwise similarity measure. We also introduced the fibre bundle assumption, a generalization of the manifold assumption, and showed that under this assumption HDM provides embeddings for both the base and the total manifold; furthermore, the flexibility of the HDM framework enables us to view VDM and the standard diffusion maps (DM) as special cases. The rest of the paper focused on analyzing the asymptotic behavior of HDM, with convergence rate estimated for finite sampling on unit tangent bundles. These results provide the mathematical foundation for HDM on fibre bundles, and motivate further studies concerning both wider applicability and deeper mathematical understanding of the algorithmic framework. We conclude this paper by listing a few potential directions for further exploration.

\begin{enumerate}[1)]
\item\label{item:17} \emph{Spectral Convergence of HDM.} The convergence results in this paper are pointwise; as in~\cite{BelkinNiyogi2007,SW2016}, we believe that it is possible to show the convergence of the eigenvalues and eigenvectors of the graph horizontal Laplacians to the eigenvalues and eigenvectors of the manifold horizontal Laplacians, thus establishing the mathematical foundation for the spectral analysis of the HDM framework. Moreover, the horizontal diffusion maps differ from diffusion maps and vector diffusion maps in that the fibres tend to be registered to a common ``template'', which, to our knowledge, is a new phenomenon addressed here for the first time.
\item\label{item:18} \emph{Spectral Clustering and Cheeger-Type Inequalities.} An important application of graph Laplacian is spectral clustering (graph partitioning). In a simple case, for a connected graph, the eigenvector corresponding to the smallest positive eigenvalue of the graph Laplacian partitions the graph vertices into two similarly sized subsets, in such a way that the number of edges across the subsets is as small as possible. In spectral graph theory~\cite{Chung1997}, the classical Cheeger's Inequality provides upper and lower bounds for the performance of the partition; recently, \cite{Bandeira2013} established similar results for the graph connection Laplacian, the central object of VDM. We believe that similar inequalities can be established for graph horizontal Laplacians as well, with potentially more interesting behavior of the eigenvectors. For instance, we observed in practice that the eigenvector corresponding to the smallest positive eigenvalue of the graph horizontal Laplacian stably partitions all the fibres in a globally consistent manner.
\item\label{item:19} \emph{Multiscale Analysis and Hierarchical Coarse-Graining.} Multiscale representation of massive, complex data sets based on similarity graphs is an interesting and fruitful application of diffusion operators~\cite{LafonLee2006,CoifmanMaggioni2006}. Based on HDM, one can build a similar theory for data sets possessing fibre bundle structures, providing a natural framework for coarse-graining that is meaningful (or even possible) only when performed simultaneously on the base and fibre manifolds. Moreover, since the horizontal diffusion matrix is often of high dimensionality, an efficient approach to store and compute its powers will significantly improve the applicability of the HDM algorithm. We thus expect to develop a theory of \emph{horizontal diffusion wavelets} and investigate their performance on real data sets with underlying fibre bundle structures.
\end{enumerate}

\appendix
\renewcommand*{\thesection}{\Alph{section}}

\section{Fibre Bundles and Connections}
\label{sec:connections}

There are different ways to define a connection on a fibre bundle. For the sake of generality, we adopt here the treatment in~\cite{Michor2008} or \cite{Ehresmann1950Connexions} from a Riemannian submersion point of view; see also \cite{Besse2007EinsteinManifolds, GLP1999} for more detailed discussions.

For any fibre bundle $\mathscr{E}=\left(E,M,F,\pi\right)$, the bundle projection map $\pi:E\rightarrow M$ descends canonically to its differential $d\pi$ between tangent bundles $T\!E$ and $T\!M$, defining linear surjective homomorphisms between tangent planes $T_eE$ and $T_{\pi \left( e \right)}M$ for any $e\in E$. We denote $V\!E$ for the \emph{vertical bundle}, a sub-bundle of $T\!E$ defined as the kernel of the differential map $d\pi:T\!E\rightarrow T\!M$. A \emph{horizontal bundle} $H\!E$ is a sub-bundle of $TE$ that is supplementary to $V\!E$ in the sense that $T\!E=H\!E\oplus V\!E$, or equivalently
\begin{equation*}
  T_eE = H_eE\oplus V_eE \quad\textrm{for all }e\in E.
\end{equation*}
Here $H_eE$, $V_eE$ stand for the fibres of $H\!E, V\!E$ above $e\in E$, respectively; we shall refer to $H_eE$, $V_eE$ as the \emph{horizontal tangent space} and \emph{vertical tangent space} at $e\in E$ for future convenience, and denote
\begin{equation}
  \label{eq:horizontal_vertical_projections}
  \mathscr{H}:T\!E\rightarrow H\!E,\quad \mathscr{V}:T\!E\rightarrow V\!E
\end{equation}
for the corresponding \emph{horizontal projection} and \emph{vertical projection}. Note that although $V\!E$ is canonically defined, the choice of $H\!E$ is arbitrary at this point. Since $d\pi\big|_{H_eE}: H_eE\rightarrow T_{\pi \left( e \right)}M$ is a linear isomorphism, for any tangent vector $u\in T_{\pi \left( e \right)}M$ there exists a unique tangent vector $\overline{u}\in H_eE$ such that $d\pi_e \left( \overline{u} \right)=u$; we call $\bar{u}$ the \emph{horizontal lift} of $u$. Furthermore, we know from simple ODE theory (and the smoothness of $H\!E$) that for any vector field $X\in \Gamma \left( M, T\!M \right)$ there exists a unique horizontal lift $\bar{X}\in\Gamma \left( E, H\!E \right)$ such that $d\pi_e \left( \bar{X}_e \right)=X_{\pi\left(e\right)}$ for all $e\in E$.

In the rest of this paper, a path $\gamma: \left[ 0,T \right]\rightarrow E$ is \emph{horizontal} if all tangent vectors along $\gamma$ are in $H\!E$. Given a path $c:\left[ 0,T \right]\rightarrow M$, a \emph{horizontal lift} of $c$ is any horizontal path $\bar{c}$ in $E$ that projects to $c$ under the bundle projection $\pi$, i.e. $\pi\circ\bar{c}=c$. Again, by horizontally lifting the tangent vector field along the path from $T\!M$ to $H\!E$ and solving the ODE system (where the overline again stands for horizontally lifted tangent vectors)
\begin{equation*}
  \frac{d\tilde{c}}{dt}=\overline{\left(\frac{dc}{dt}\right)},\quad t\in \left[ 0,T \right]
\end{equation*}
we can uniquely lift any piecewise smooth path $c$ in $M$ starting at $\pi \left( e \right)\in M$ to a horizontal path $\bar{c}$ in $E$ starting at $e\in E$, at least locally around $c \left( 0 \right)$. We call $H\!E$ a \emph{Ehresmann connection}~\cite{Ehresmann1950Connexions}, or \emph{connection} hereafter, if any path in $M$ starting at $m\in M$ can be \emph{globally} horizontally lifted to $E$ with any given initial point $e\in E$ satisfying $e\in \pi^{-1}\left( m \right)$. Such a lifting property is guaranteed, for instance, on any Riemannian submersion $\pi:E\rightarrow M$ with geodesically complete total space $E$, in which case the submersion is known to be a locally trivial fibration~\cite{Hermann1960}.

We shall focus on Ehresmann connections so that the horizontal lift of any path in $M$ is uniquely determined once the starting point on $E$ is specified. Therefore, given a smooth curve $\gamma:\left[ 0,T \right]\rightarrow M$ that connects $\gamma \left( 0 \right)$ to $\gamma \left( T \right)$ on $M$, there exists a smooth map from $F_{\gamma \left( 0 \right)}$ to $F_{\gamma \left( T \right)}$  (at least when $\gamma \left( 0 \right)$ and $\gamma \left( T \right)$ are sufficiently close), defined as
\begin{equation*}
  F_{\gamma \left( 0 \right)}\ni e\mapsto \bar{\gamma}_e\left( T \right)\in F_{\gamma \left( T \right)},
\end{equation*}
where $\bar{\gamma}_e$ denotes the horizontal lift of $\gamma$ with starting point $p$. We call this construction of maps between fibres, obviously depending on the choice of path $\gamma$, the \emph{parallel transport along $\gamma$ (with respect to the connection)}, and denote $P^{\gamma}_{yx}:F_x\rightarrow F_y$ for the parallel transport from fibre $F_x$ to fibre $F_y$. When $\gamma$ is a unique geodesic on $M$ that connects $x$ to $y$, we drop the superscript $\gamma$ and simply write $P_{yx}:F_x\rightarrow F_y$. For future reference, we give the precise definition of the operator $P_{yx}$ here.
\begin{definition}[Parallel Transport on Fibre Bundles]
\label{def:bundle-parallel-transport}
  Let $\mathscr{E}=\left(E,M,F,\pi\right)$ be a fibre bundle, $x\in M$, $v\in F_x$, and $U$ a geodesic normal neighborhood of $x$ on the base manifold $M$. For any $y\in U$, denote the geodesic distance between $x$ and $y$ as $d_M \left( x,y \right)$. Let $\gamma:\left[ 0, d_M \left( x,y \right)\right]\rightarrow M$ be the unique unit-speed geodesic on $M$ connecting $x$ to $y$, i.e., $\gamma \left( 0 \right)=x, \gamma \left( d_M \left( x,y \right) \right)=y$; let $\bar{\gamma}$ be the unique horizontal lift of $\gamma$ starting at $v\in F_x$, i.e.,
  \begin{equation*}
    \begin{cases}
      \bar{\gamma}' \left( t \right) = \overline{\gamma'} \left( t \right),&\, t\in \left[ 0, d_M \left( x,y \right) \right],\\
      \bar{\gamma} \left( 0 \right)=v.&
    \end{cases}
  \end{equation*}
The parallel-transport of $v$ from $x$ to $y$, denoted as $P_{yx}v$, is defined as
\begin{equation*}
  P_{yx}v = \bar{\gamma}\left( d_M \left( x,y \right) \right)\in F_y.
\end{equation*}
\end{definition}

The probabilistic interpretation of HDM (and even VDM) implicitly depends on lifting from the base manifold a path that is continuous but not necessarily smooth. Though this can not be trivially achieved by the ODE-based approach, stochastic differential geometers developed tools appropriate for tackling this technicality (see e.g. ~\cite[\S 5.1.2]{Stroock2005AnalysisPaths}).

\section{Horizontal and Vertical Laplacians}
\label{sec:horiz-lapl}

Assume $\left(M, g^M\right)$ is a $d$-dimensional Riemannian manifold, and denote $\nabla^M$ for the canonical Levi-Civita connection on $M$. The \emph{Laplace-Beltrami operator} on $M$, or \emph{Laplacian} for short, is the analogy of the usual Laplace operator on the Euclidean space defined by
\begin{equation*}
  \Delta_Mf \left( x \right) = \mathrm{Trace}\,\nabla^M\nabla f \left( x \right)
\end{equation*}
for all $f\in C^{\infty}\left( M \right)$, $x\in M$. For an orthonormal local frame $\left\{ X_1,\cdots,X_d \right\}$ near $x\in M$, $\Delta_M$ can also be written as
\begin{equation}
\label{eq:base_laplacian}
  \Delta_Mf \left( x \right)=\sum_{j=1}^d g^M\left(\nabla_{X_j}^M\nabla f, X_j \right) \left( x \right)=\sum_{j=1}^dX_j^2f \left( x \right)-\left(\sum_{j=1}^d\nabla_{X_j}^MX_j \left( x \right)\right)f \left( x \right).
\end{equation}
If we further pick the frame to be a local geodesic frame centered at $x\in M$, then $\nabla^M_{X_j}X_k \left( x \right)=0$ for all $1\leq j,k\leq d$ and thus $\Delta_M$ takes the following sum-of-squares form
\begin{equation}
\label{eq:laplacian_geodesic_frame}
  \Delta_Mf \left( x \right)=\sum_{j=1}^dX_j^2f \left( x \right).
\end{equation}
The infinitesimal generator of the horizontal diffusion~\eqref{eq:horizontal_diffusion_operator} turns out to be a differential operator on $E$ that is a ``horizontal lift'' of $\Delta_M$ in a sense to be made clear later in this section. To characterize this infinitesimal generator, let us first introduce a Riemannian metric on $E$ that is adapted to the connection $HE$. For any $x\in M$, recall from Section~\ref{sec:fibre-bundles} that $F_x$ (the fibre at point $x\in M$) is a Riemannian submanifold of $E$, thus vertical tangent vectors at $e\in \pi^{-1}\left( x \right)$ can be canonically identified with tangent vectors to $F_x$; if each $F_x$ is equipped with a Riemannian metric $g^{F_x}$, we define for any $U,V\in V_eE$
\begin{equation}
  \label{eq:vertical_metric}
  g^E\left( U, V \right) = g^{F_x} \left( U, V \right).
\end{equation}
For any $X,Y\in H_eE$, by the linear isomorphism between $H_eE$ and $T_xM$ we define
\begin{equation}
  \label{eq:horizontal_metric}
  g^E \left( X,Y \right) = g^M\left( d\pi_e \left( X\right), d\pi_e\left(Y \right) \right)
\end{equation}
where $g^M$ stands for the Riemannian metric on $M$. Finally, impose orthogonality between $H_eE$ and $V_eE$ by setting for any $X\in H_eE$, $U\in V_eE$
\begin{equation}
  \label{eq:orthogonality}
  g^E \left( X, U \right) = 0.
\end{equation}
The smoothness of $g^E$ with respect to $e\in E$ follows from the smoothness of $g^M$ and $g^{F_x}$. In other words, $g^E$ is constructed so as to make the decomposition $T\!E=H\!E\oplus V\!E$ orthogonal. Some authors~\cite{Bismut2013,Baudoin2014} abbreviate this construction as
\begin{equation}
\label{eq:metric_on_total_manifold}
  g^E=g^M\oplus g^F.
\end{equation}
For future convenience, let us use superscripts to denote the horizontal and vertical components of tangent vectors to $E$, i.e. for any $Z\in T_eE$
\begin{equation*}
  Z = Z^H+Z^V
\end{equation*}
where $Z^H\in H_eE$, $Z^V\in V_eE$ are uniquely determined due to the direct sum decomposition $T_eE=H_eE\oplus V_eE$. Thus for any $W,Z\in T_eE$
\begin{equation*}
  g^E\left( W,Z\right) = g^M \left( d\pi_e \left( W^H\right), d\pi_e\left(Z^H \right) \right)+g^{F_{\pi \left( e \right)}}\left( W^V,Z^V \right).
\end{equation*}
We also write the horizontal and vertical components of the gradient of any smooth function $f\in C^{\infty}\left( E \right)$ as
\begin{equation}
  \label{eq:horizontal_vertical_gradient}
    \nabla^Hf := \left( \nabla f \right)^H,\quad \nabla^Vf := \left( \nabla f \right)^V.
\end{equation}

Let $\nabla^E$ denote the Levi-Civita connection with respect to $g^E$. Define the \emph{rough horizontal Laplacian} $\Delta_H$ on $E$ for $f \in C^{\infty}\left( E \right)$ as the following second order partial differential operator:
\begin{equation}
  \label{eq:bochner_horizontal_laplacian}
  \Delta_H f \left( e \right)=\mathrm{Trace}\,\left(\nabla^E\nabla^H f\right)^H \left( e \right)\quad\textrm{for all }e\in E.
\end{equation}
Let $\left\{ \bar{X}_1,\cdots,\bar{X}_d \right\}$ be the horizontal lift of an orthonormal frame $\left\{ X_1,\cdots,X_d \right\}$ near $\pi \left( e \right)=x\in M$. Since $g^E \left( \bar{X}_j,\bar{X}_k \right)=g^M \left( X_j,X_k \right)$ for $1\leq j,k\leq d$, the tangent vectors $\bar{X}_1 \left( e' \right), \cdots,\bar{X}_d \left( e' \right)$ form an orthonormal basis for $H_{e'}E$ for all $e'$ sufficiently close to $e$. We can write \eqref{eq:bochner_horizontal_laplacian} in terms of these horizontally lifted vector fields as
\begin{equation}
  \label{eq:bochner_lifted_orthonormal_frame}
  \begin{aligned}
    \Delta_H f \left( e \right)&=\sum_{j=1}^dg^E \left(\nabla_{\bar{X}_j}^E\nabla^H f, \bar{X}_j  \right) \left( e \right)\\
    &=\sum_{j=1}^d \bar{X}_j^2f \left( e \right)-\left( \sum_{j=1}^d\left(\nabla^E_{\bar{X}_j}\bar{X}_j\right)^H \right)f \left( e \right).
  \end{aligned}
\end{equation}
Loosely speaking, $\Delta_H$ is the ``horizontal lift'' of $\Delta_M$ from $M$ to $E$, since \eqref{eq:bochner_lifted_orthonormal_frame} can be obtained from \eqref{eq:base_laplacian} by replacing each $X_j$ with its horizontal lift $\bar{X}_j$ and noting that $\left(\nabla^E_{\bar{X}_j}\bar{X}_j\right)^H$ is the horizontal lift of $\nabla^M_{X_j}X_j$ (see e.g. \cite[Proposition 3.1]{Hermann1960}). More precisely, for any $g\in C^{\infty}\left( M \right)$, denote $\bar{g}=g\circ\pi\in C^{\infty}\left( E \right)$, then for any $e\in E$ and $x=\pi \left( e \right)\in M$ we have
\begin{equation}
  \label{eq:laplacian_lift}
  \Delta_M g \left( x \right) = \Delta_H\bar{g}\left( e \right).
\end{equation}
\begin{remark}
  When $E=\mathcal{O}\left( M \right)$ is the frame bundle of $M$, the rough horizontal Laplacian $\Delta_H$ coincides with the \emph{Bochner horizontal Laplacian} $\Delta_{\mathcal{O}\left( M \right)}$ in stochastic differential geometry \cite[Chapter 3]{Hsu2002Book}. The classical Eells-Elworthy-Malliavin approach intrinsically defines a Brownian motion on manifolds as a \emph{horizontal Brownnian motion} on $\mathcal{O}\left( M \right)$ generated by $\Delta_{\mathcal{O}\left( M \right)}$.
\end{remark}
\begin{remark}
\label{rem:different_horizontal_laplacian}
  In general, the rough horizontal Laplacian $\Delta_H$ differs from the concept of ``horizontal Laplacian'' commonly seen in sub-Riemannian geometry by a mean curvature term~\cite{Baudoin2014,BerardBourguignon1982}; the two types of horizontal Laplacian coincide only when the fibres of $E$ are minimal submanifolds of $E$. In fact, for any $f\in C^{\infty}\left( E \right)$, the Laplace-Beltrami operator on $E$ with respect to $g^E$ splits into two parts
\begin{equation*}
  \begin{aligned}
    \Delta_Ef &= \mathrm{Trace}\,\nabla^E\nabla f=\mathrm{Trace}\,\nabla^E\nabla^H f+\mathrm{Trace}\,\nabla^E\nabla^Vf
  \end{aligned}
\end{equation*}
Define the \emph{horizontal Laplacian} $\Delta_E^H$ and the \emph{vertical Laplacian} $\Delta_E^V$ as
\begin{equation}
  \label{eq:horizontal_vertical_laplacians}
  \Delta_E^Hf := \mathrm{Trace}\,\nabla^E\nabla^H f,\quad \Delta_E^Vf:=\mathrm{Trace}\,\nabla^E\nabla^V f,
\end{equation}
then
\begin{equation}
  \label{eq:laplacian_split}
  \Delta_E = \Delta_E^H+\Delta_E^V.
\end{equation}
Recalling the definition of $\Delta_H$ from \eqref{eq:bochner_horizontal_laplacian}, we have
\begin{equation*}
  \begin{aligned}
    \Delta_E^Hf &= \mathrm{Trace}\left(\nabla^E\nabla^Hf\right)^H+\mathrm{Trace}\left(\nabla^E\nabla^Hf\right)^V\\
    & = \Delta_Hf+\mathrm{Trace}\left(\nabla^E\nabla^Hf\right)^V
  \end{aligned}
\end{equation*}
and $\Delta_E^H=\Delta_H$ if and only if
\begin{equation*}
  \mathrm{Trace}\left(\nabla^E\nabla^H f\right)^V=0\quad\textrm{for all }f\in C^{\infty}\left( E \right)
\end{equation*}
which turns out to be equivalent to the requirement that $F_x$ are minimal submanifolds of $E$ for all $x\in M$. This holds, for instance, when all fibres of the Riemannian submersion $\pi:E\rightarrow M$ are totally geodesic, a scenario of great theoretic interest since it implies that all fibres are isometric~\cite{Hermann1960}; we do not make such an assumption in the HDM framework since this particularly simple case is obviously too restricted for practical purposes.
\end{remark}
\begin{remark}
  For any $x\in M$ and $e\in \pi^{-1}\left( x \right)$, if $\left\{ X_1,\cdots,X_d \right\}$ is a geodesic frame on $M$ near $x$, then the horizontal lifts $\left\{ \bar{X}_1,\cdots,\bar{X}_d \right\}$ near $e$ also constitute a ``horizontal geodesic frame'' in the sense that
  \begin{equation*}
    \left(\nabla_{\bar{X}_j}^E\bar{X}_k\right)^H \left( e \right)=0,\quad\textrm{for all }1\leq j,k\leq d,
  \end{equation*}
which simplifies \eqref{eq:bochner_lifted_orthonormal_frame} into a sum-of-squares form analogous to \eqref{eq:laplacian_geodesic_frame}
\begin{equation}
  \label{eq:bochner_geodesic_frame}
  \Delta_H f \left( e \right) = \sum_{j=1}^d\bar{X}_j^2 f \left( e \right)\quad\textrm{for all }f\in C^{\infty}\left( E \right).
\end{equation}
\end{remark}

\begin{remark}
  We make the observation that the vertical Laplacian $\Delta_E^V$, which turns out to characterize the ``vertical component'' of the coupled diffusion operator $H_{\epsilon,\delta}^{\left( \alpha \right)}$ on the fibre bundle, coincides with the Laplace-Beltrami operator on each fibre $F_x$. This fact will be needed in the proof of Theorem~\ref{thm:warped_diffusion_generator} in \ref{sec:proof_biase}. 
More precisely, for any $f\in C^{\infty}\left( E \right)$ and $e\in E$,
\begin{equation}
\label{eq:restriction_vertical_laplacian}
  \Delta_E^Vf \left( e \right) = \left[\Delta_{F_{\pi\left( e \right)}}\left(f\restriction F_{\pi \left( e \right)}\right)\right]\left( e \right)=\mathrm{Trace}\left(\nabla^E\nabla^Vf\right)^V.
\end{equation}
At a first glance this might seem a bit surprising since one may expect a mean curvature term in $\Delta_E^V$ from~\eqref{eq:horizontal_vertical_laplacians} (as is the case for $\Delta_E^H$):
\begin{equation}
\label{eq:vertical_laplacian_decomposition}
    \Delta_E^Vf = \mathrm{Trace}\,\nabla^E\nabla^Vf = \mathrm{Trace}\left(\nabla^E\nabla^Vf\right)^H+\mathrm{Trace}\left(\nabla^E\nabla^Vf\right)^V.
\end{equation}
However, the first trace term in \eqref{eq:vertical_laplacian_decomposition} vanishes for the following reason. Let $\left\{ X_1, \cdots, X_d \right\}$ be a local horizontal orthonormal frame around $e\in E$, and $\left\{ U_1,\cdots,U_n \right\}$ a local vertical orthonormal frame (recall that $\mathrm{dim}\left(F\right) = n$); $\left\{ X_1,\cdots,X_d,U_1,\cdots,U_n \right\}$ is then a local orthonormal frame on $E$. We have
\begin{align*}
  \mathrm{Trace}\left(\nabla^E\nabla^Vf\right)^H&=\sum_{j=1}^d \langle \left(\nabla^E_{X_j}\nabla^Vf\right)^H, X_j \rangle+\sum_{k=1}^n\langle \left(\nabla^E_{U_k}\nabla^Vf\right)^H, U_k \rangle=\sum_{j=1}^d \langle \left(\nabla^E_{X_j}\nabla^Vf\right)^H, X_j \rangle\\
  &=\sum_{j=1}^d \langle \nabla^E_{X_j}\nabla^Vf, X_j \rangle=\langle \nabla^Vf, \sum_{j=1}^d-\nabla^E_{X_j}X_j \rangle=\langle \nabla^Vf, \sum_{j=1}^d-\left(\nabla^E_{X_j}X_j\right)^V \rangle=0,
\end{align*}
where the last equality follows from~\cite[Lemma 2]{ONeill1966}:
\begin{equation*}
  \left(\nabla^E_{X_j}X_j\right)^V = \frac{1}{2}\left(\left[ X_j,X_j \right]\right)^V=0\quad\textrm{for all }1\leq j\leq d.
\end{equation*}
\end{remark}

In the remaining section we consider horizontal and coupled diffusion operators on a few classical examples. All fibre bundles in this section are Riemannian submersions with totally geodesic fibres, for which, as explained in Remark~\ref{rem:different_horizontal_laplacian}, the rough horizontal Laplacian $\Delta_H$ equals to the horizontal Laplacian $\Delta^H_E$. See \cite{Baudoin2014}\cite[\S9.F]{Besse2007EinsteinManifolds} for more details about Riemannian submersions with totally geodesic fibres.

\begin{example}[Heisenberg Group]
\label{sec:heisenberg-group}
  The Heisenberg group
$$\mathbb{H}^{2n+1}=\left\{ \left( x,y,z \right)\in \mathbb{R}^{2n+1}\mid x\in\mathbb{R}^n,y\in \mathbb{R}^n,z\in\mathbb{R} \right\}$$
is essentially $\mathbb{R}^{2n+1}$ endowed with the following group structure:
\begin{equation*}
  \left( x_1,y_1,z_1 \right)\cdot \left( x_2,y_2,z_2 \right) = \left( x_1+x_2,y_1+y_2,z_1+z_2+\frac{1}{2}\left( x_1\cdot y_2-x_2\cdot y_1 \right) \right).
\end{equation*}
The projection
\begin{equation*}
  \begin{aligned}
    \pi:\mathbb{H}^{2n+1}&\longrightarrow \mathbb{R}^2\\
    \left( x,y,z \right)&\longmapsto \left( x,y \right)
  \end{aligned}
\end{equation*}
is a Riemannian submersion with totally geodesic fibres \cite{Baudoin2014}. Since $\mathbb{H}^{2n+1}$ is complete, it follows from \cite[Theorem 1]{Hermann1960} that $\left( \mathbb{H}^{2n+1},\mathbb{R}^2,\mathbb{R},\pi\right)$ is a fibre bundle. In fact, $\mathbb{H}^{2n+1}$ is a Lie group, and its Lie algebra of left invariant vector fields at $\left( x,y,z \right)$ is spanned by
\begin{equation*}
  \frac{\partial}{\partial z}, \quad \frac{\partial}{\partial x_j}-\frac{1}{2}y_j \frac{\partial}{\partial z},\quad \frac{\partial}{\partial y_j}-\frac{1}{2}x_j \frac{\partial}{\partial z},\quad j=1,\cdots,n.
\end{equation*}
These invariant vector fields define a connection on $\mathbb{H}^{2n+1}$ in the sense of Ehresmann~\cite{Ehresmann1950Connexions}. The horizontal and vertical Laplacians on $\mathbb{H}^{2n+1}$ with respect to this connection are
\begin{equation*}
  \begin{aligned}
    \Delta_{\mathbb{H}^{2n+1}}^H&=\sum_{j=1}^n \left[ \frac{\partial^2}{\partial x_j^2}+\frac{\partial^2}{\partial y_j^2}+\frac{1}{4}\left( x_j^2+ y_j^2 \right)\frac{\partial^2}{\partial z^2}-y_j \frac{\partial^2}{\partial x_j\partial z}+x_j \frac{\partial^2}{\partial y_j\partial z} \right],\\
    \Delta_{\mathbb{H}^{2n+1}}^V&=\frac{\partial^2}{\partial z^2}.
  \end{aligned}
\end{equation*}
By Theorem~\ref{thm:warped_diffusion_generator}, for any $f\in C^{\infty}\left( \mathbb{H}^{2n+1} \right)$, if $\delta = O \left( \epsilon \right)$,
\begin{align*}
  H_{\epsilon,\delta}^{\left( \alpha \right)} & f \left( x,v \right) = f \left( x,v \right) + \epsilon \frac{m_{21}}{2m_0} \frac{\left[\Delta_{\mathbb{H}^{2n+1}}^H \left( fp^{1-\alpha} \right)-f\Delta_{\mathbb{H}^{2n+1}}^Hp^{1-\alpha} \right]\left( x,v \right)}{p^{1-\alpha}\left( x,v \right)}\\
      &+\delta\frac{m_{22}}{2m_0} \frac{\left[\Delta_{\mathbb{H}^{2n+1}}^V \left( fp^{1-\alpha} \right)-f\Delta_{\mathbb{H}^{2n+1}}^Vp^{1-\alpha} \right]\left( x,v \right)}{p^{1-\alpha}\left( x,v \right)}+O \left( \epsilon^2+\epsilon\delta+\delta^2 \right).
\end{align*}
When $n=1$, this is consistent with the conclusion obtained in \cite[Chapter 4]{Gao2015Thesis}.
\end{example}


\begin{example}[Tangent Bundles]
  \label{sec:tangent-bundles}
Tangent bundles play an important role in Riemannian geometry. For a closed $d$-dimensional Riemannian manifold $\left(M,g\right)$, its tangent bundle $T\!M$ is defined as
\begin{equation*}
  T\!M=\coprod_{x\in M}T_xM
\end{equation*}
equipped with a natural smooth structure (see e.g. \cite{doCarmo1992RG}). In a local coordinate chart $\left( U; x_1,\cdots,x_d \right)$ of $M$, $\left\{ E_j=\partial\slash\partial x_j\mid 1\leq j\leq d\right\}$ is a local frame on $M$, and we write $v\in T_xM$ as $v=v_jE_j \left( x \right)$. A local trivialization on $U$ can be chosen as
\begin{equation*}
  \begin{aligned}
  \left( x,v \right)&\mapsto \left( x_1,\cdots,x_d,v_1,\cdots,v_d \right),\quad \forall x\in U, v\in T_xM,
  \end{aligned}
\end{equation*}
and the corresponding basis for $T_{\left( x,v \right)}T\!M$ can be written as
\begin{equation*}
  \left\{\frac{\partial}{\partial x_1}\Big|_{\left( x,v \right)},\cdots,\frac{\partial}{\partial x_d}\Big|_{\left( x,v \right)},\frac{\partial}{\partial v_1}\Big|_{\left( x,v \right)},\cdots,\frac{\partial}{\partial v_d}\Big|_{\left( x,v \right)}\right\}.
\end{equation*}
Let $\Gamma_{\alpha j}^{\beta}$ be the connection coefficients of the Levi-Civita connection on $M$. The horizontal subbundle of $TT\!M$ determined by this connection is
\begin{equation*}
  \begin{aligned}
    HTM:=&\coprod_{\left(x,v\right)\in TM}HT_{\left( x,v \right)}M=\coprod_{\left(x,v\right)\in TM}\textrm{span}\left\{\frac{\partial}{\partial x_j}\Big|_{\left( x,v \right)}\!\!\!\!-\Gamma_{\alpha j}^{\beta}\left( x \right)v_{\alpha} \frac{\partial}{\partial v_{\beta}}\Big|_{\left( x,v \right)},\quad j=1\cdots,d \right\}.
  \end{aligned}
\end{equation*}
The metric \eqref{eq:metric_on_total_manifold} on $T\!M$ given by this construction is the \emph{Sasaki metric}~\cite{Sasaki1958,Sasaki1962}. The horizontal and vertical Laplacians acts on any $f\in C^{\infty}\left( T\!M \right)$ as
\begin{align*}
    &\Delta_{T\!M}^Hf \left( x,v \right)=\frac{1}{\sqrt{\left| g \left( x \right) \right|}}\left( \frac{\partial}{\partial x_j}-\Gamma_{\alpha j}^{\beta}\left( x \right)v_{\alpha} \frac{\partial}{\partial v_{\beta}}\right)\left[ \sqrt{\left| g \left( x \right) \right|}\,g^{jk}\left( x \right)\left(\frac{\partial f}{\partial x_k}-\Gamma_{\alpha k}^{\beta}\left( x \right)v_{\alpha} \frac{\partial f}{\partial v_{\beta}}\right) \right],\\
    &\Delta_{T\!M}^Vf \left( x,v \right)=\frac{1}{\sqrt{\left| g \left( x \right) \right|}}\frac{\partial}{\partial v_j}\left( \sqrt{\left| g \left( x \right) \right|}\,g^{jk}\left( x \right)\frac{\partial f}{\partial v_k}\right)=g^{jk}\left( x \right)\frac{\partial^2 f}{\partial v_j\partial v_k}.
\end{align*}
According to Theorem~\ref{thm:warped_diffusion_generator}, for any $f\in C^{\infty}\left( T\!M \right)$, if $\delta=O \left( \epsilon \right)$,
\begin{align*}
  H_{\epsilon,\delta}^{\left(\alpha\right)} & f\left( x,v \right) = f \left( x,v \right)+\epsilon \frac{m_{21}}{2m_0}\frac{\left[\Delta_{T\!M}^H\left(fp^{1-\alpha}\right)-f\Delta_{T\!M}^Hp^{1-\alpha}\right]\left( x,v \right)}{p^{1-\alpha}\left( x,v \right)}\\
    &+\delta \frac{m_{22}}{2m_0}\frac{\left[\Delta_{T\!M}^V\left(fp^{1-\alpha}\right)-f\Delta_{T\!M}^Vp^{1-\alpha}\right]\left( x,v \right)}{p^{1-\alpha}\left( x,v \right)}+O\left(\epsilon^2+\epsilon\delta+\delta^2 \right).
\end{align*}
This is consistent with the conclusion obtained in \cite[Chapter 3]{Gao2015Thesis}.
\end{example}


\section{Proofs of Theorem~\ref{thm:horizontal_diffusion_generator}, Theorem~\ref{thm:warped_diffusion_generator}, and Theorem~\ref{thm:warped_diffusion_generator_unbounded_ratio}}
\label{sec:proof_biase}
Throughout this appendix we assume the Einstein summation convention unless otherwise specified. Our starting point is the following lemma, in reminiscent of \cite[Lemma 8]{CoifmanLafon2006} and \cite[Lemma B.10]{SingerWu2012VDM}.

\begin{lemma}
\label{lem:basic_kernel_integral}
  Let $\Phi:\mathbb{R}\rightarrow\mathbb{R}$ be a smooth function compactly supported in $\left[ 0,1 \right]$. Assume $M$ is a $d$-dimensional compact Riemannian manifold without boundary, with injectivity radius $\mathrm{Inj}\left( M \right)>0$. For any $\epsilon>0$, define kernel function
  \begin{equation}
    \label{eq:epsilon_kernel_family}
    \Phi_{\epsilon}\left( x,y \right)=\Phi \left( \frac{d^2_M \left( x,y \right)}{\epsilon} \right)
  \end{equation}
on $M\times M$, where $d^2_M \left( \cdot,\cdot \right)$ is the geodesic distance on $M$. For sufficiently small $\epsilon$ satisfying $0\leq\epsilon\leq\sqrt{\mathrm{Inj}\left( M \right)}$, the integral operator associated with kernel $\Phi_{\epsilon}$
\begin{equation}
  \label{eq:integral_operator}
  \left(\Phi_{\epsilon}\,g\right)\left( x \right):=\int_M \Phi_{\epsilon}\left( x,y \right)g \left( y \right)d\mathrm{vol}_M \left( y \right)
\end{equation}
has the following asymptotic expansion as $\epsilon\rightarrow 0$:
\begin{equation}
  \label{eq:basic_asymp_expansion}
  \left(\Phi_{\epsilon}\,g\right)\left( x \right) = \epsilon^{\frac{d}{2}}\left[ m_0 g \left( x \right)+\epsilon \frac{m_2}{2}\left( \Delta_M g \left( x \right)-\frac{1}{3}\mathrm{Scal}^M\left( x \right)g \left( x \right) \right)+O \left( \epsilon^2 \right) \right],
\end{equation}
where $m_0,m_2$ are constants that depend on the moments of $\Phi$ and the dimension $d$ of the Riemannian manifold $M$, $\Delta_M$ is the Laplace-Beltrami operator on $M$, and $\mathrm{Scal}^M\left( x \right)$ is the scalar curvature of $M$ at $x$.
\end{lemma}

\begin{proof}
  Consider geodesic normal coordinates near $x\in M$. Let $\left\{e_1,\cdots,e_d\right\}$ be an orthonormal basis for $T_xM$, $\left(s_1,\cdots,s_d\right)$  the geodesic normal coordinates, and write $r=d_M \left( x,y \right)$. Then $r^2=s_1^2+\cdots+s_d^2$. Note that
  \begin{equation}
  \label{lem1:geodesicnormalcoords}
    \begin{aligned}
      \int_M \Phi_{\epsilon}\left( x,y \right)g \left( y \right)\,d\mathrm{vol}_M \left( y \right) &= \int_{B_{\sqrt{\epsilon}}\left( x \right)} \Phi\left(\frac{d^2_M \left( x,y \right)}{\epsilon}\right)g \left( y \right)\,d\mathrm{vol}_M \left( y \right)\\
      &=\int_{B_{\sqrt{\epsilon}}\left( 0 \right)}\Phi \left( \frac{r^2}{\epsilon} \right)\tilde{g} \left(s\right)\,d\mathrm{vol}_M \left(s\right)
    \end{aligned}
  \end{equation}
where
\begin{equation*}
  \begin{aligned}
    \tilde{g}\left( s\right)=\tilde{g}\left( s_1,\cdots,s_d \right)&=g \circ \mathrm{exp}_x \left( s_1e_1+\cdots+s_de_d \right),\\
    d\mathrm{vol}_M \left( s \right)=d\mathrm{vol}_M \left( s_1,\cdots,s_d \right)&=d\mathrm{vol}_M \left( \mathrm{exp}_x \left( s_1e_1+\cdots+s_de_d \right) \right).
  \end{aligned}
\end{equation*}
By a further change of variables
\begin{equation}
\label{eq:square_root_coc}
  \tilde{s}_1=\frac{s_1}{\sqrt{\epsilon}},\cdots,\tilde{s}_d=\frac{s_d}{\sqrt{\epsilon}}; \quad\tilde{r}=\frac{r}{\sqrt{\epsilon}},
\end{equation}
we have
\begin{equation}
\label{lem1:changeofcoordinates}
  \begin{aligned}
    \int_{B_{\sqrt{\epsilon}}\left( 0 \right)} & \Phi \left( \frac{r^2}{\epsilon} \right)\tilde{g} \left(s \right)\,d\mathrm{vol}_M \left(s \right)=\int_{B_1\left( 0 \right)}\Phi \left(\tilde{r}^2 \right)\tilde{g} \left(\sqrt{\epsilon}\,\tilde{s} \right)d\mathrm{vol}_M \left(\sqrt{\epsilon}\,\tilde{s}\right).
  \end{aligned}
\end{equation}
On the other hand, in geodesic normal coordinates the Riemannian volume form has asymptotic expansion (see e.g.~\cite{Petersen2006})
\begin{equation}
\label{eq:volume_form_expansion}
  d\mathrm{vol}_M \left( s_1,\cdots,s_d \right)=\left[1-\frac{1}{6}R_{k\ell}\left( x \right)s_ks_{\ell}+O \left( r^3 \right)\right]ds_1\cdots ds_d
\end{equation}
where $R_{k\ell}$ is the Ricci curvature tensor
\begin{equation*}
  R_{k\ell} \left( x \right)=g^{ij}R_{ki\ell j}\left( x \right).
\end{equation*}
Thus
\begin{equation}
\label{lem1:dvol}
  d\mathrm{vol}_M \left( \sqrt{\epsilon}\,\tilde{s}_1,\cdots,\sqrt{\epsilon}\,\tilde{s}_d \right)=\left[1-\frac{\epsilon}{6}R_{k\ell}\left( x \right)\tilde{s}_k\tilde{s}_{\ell}+O \left(\epsilon^{\frac{3}{2}}\tilde{r}^3 \right)\right]\cdot \epsilon^{\frac{d}{2}}d\tilde{s}_1\cdots d\tilde{s}_d.
\end{equation}
In the meanwhile, the Taylor expansion of $\tilde{g}\left( s \right)$ near $x$ reads
\begin{equation*}
  \tilde{g}\left( s \right)=\tilde{g}\left( 0 \right)+\frac{\partial\tilde{g}}{\partial s_j}\left( 0 \right)s_j+ \frac{1}{2}\frac{\partial^2\tilde{g}}{\partial s_k\partial s_{\ell}}\left( 0 \right)s_ks_{\ell}+O \left( r^3 \right)
\end{equation*}
and thus
\begin{equation}
\label{lem1:taylorexpansion}
  \tilde{g} \left(\sqrt{\epsilon}\,\tilde{s} \right)=g \left( x \right)+\sqrt{\epsilon}\cdot \frac{\partial\tilde{g}}{\partial s_j}\left( 0 \right)\tilde{s}_j+\epsilon\cdot \frac{1}{2} \frac{\partial^2\tilde{g}}{\partial s_k\partial s_{\ell}}\left( 0 \right)\tilde{s}_k\tilde{s}_{\ell}+O \left( \epsilon^{\frac{3}{2}}\tilde{r}^3 \right).
\end{equation}
By the symmetry of the kernel and the domain of integration $B_1 \left( 0 \right)$,
\begin{equation}
\label{eq:symmetry_nil}
  \begin{aligned}
    &\int_{B_1 \left( 0 \right)}\Phi \left( \tilde{r}^2 \right)\tilde{s}_jd\tilde{s}_1\cdots d\tilde{s}_d=0\quad\textrm{for all }1\leq j\leq d,\\
    &\int_{B_1 \left( 0 \right)}\!\!\Phi \left( \tilde{r}^2 \right)\tilde{s}_k\tilde{s}_{\ell}\,d\tilde{s}_1\cdots d\tilde{s}_d=0\quad\textrm{fir all }1\leq k\neq \ell\leq d.
  \end{aligned}
\end{equation}
Combining (\ref{lem1:geodesicnormalcoords})--(\ref{eq:symmetry_nil}), we have
\begin{equation*}
  \begin{aligned}
    &\int_M \Phi_{\epsilon}\left( x,y \right)g \left( y \right)\,d\mathrm{vol}_M \left( y \right)=\int_{B_1\left( 0 \right)}\!\!\Phi \left(\tilde{r}^2\right)\tilde{g} \left(\sqrt{\epsilon}\,\tilde{s}_1,\cdots,\sqrt{\epsilon}\,\tilde{s}_d  \right)\,d\mathrm{vol}_M \left(\sqrt{\epsilon}\,\tilde{s} \right)\\
    &=\epsilon^{\frac{d}{2}}\Bigg[ g \left( x \right)\!\!\int_{B_1 \left( 0 \right)}\!\!\!\Phi \left( \tilde{r}^2\right)\,d\tilde{s}+\frac{\epsilon}{2}\sum_{k=1}^d\left( \frac{\partial^2\tilde{g}}{\partial s_k^2}\left( 0 \right)-\frac{1}{3}g \left( x \right)R_{kk}\left( x \right) \right)\int_{B_1 \left( 0 \right)}\!\!\!\!\Phi \left( \tilde{r}^2 \right)\tilde{s}_k^2\,d\tilde{s}+O \left( \epsilon^2 \right) \Bigg]
  \end{aligned}
\end{equation*}
Note that $O \left( \epsilon^{3/2} \right)$ term vanishes again by symmetry (the same argument given in~\cite[\S 2]{Singer2006ConvergenceRate} applies). Define constants
\begin{equation}
\label{eq:base_constants}
  \begin{aligned}
    m_0&:=\int_{B_1 \left( 0 \right)}\!\!\Phi \left( \tilde{r}^2 \right)\,d\tilde{s}_1\cdots d\tilde{s}_d=\omega^{d-1}\int_0^1\Phi \left( \tilde{r}^2 \right)\tilde{r}^{d-1}d\tilde{r},\\
    m_2 &:= \int_{B_1 \left( 0 \right)}\!\!\Phi \left( \tilde{r}^2 \right)\left( \tilde{s}_k \right)^2\,d\tilde{s}_1\cdots d\tilde{s}_d\quad\textrm{for any $k\in \left\{ 1,\cdots,d \right\}$.}
  \end{aligned}
\end{equation}
Then
\begin{align*}
    \int_M &\Phi_{\epsilon}\left( x,y \right)g \left( y \right)\,d\mathrm{vol}_M \left( y \right)\\
    &=\epsilon^{\frac{d}{2}}\left[ m_0 g \left(x \right) +\epsilon \frac{m_2}{2}\left( \Delta_Mg \left( x \right)-\frac{1}{3}\mathrm{Scal}^M\left( x \right)g \left( x \right) \right) +O \left( \epsilon^2 \right) \right],
\end{align*}
where we used the fact that in geodesic normal coordinates
\begin{equation*}
  \begin{aligned}
    \sum_{k=1}^d \frac{\partial^2\tilde{g}}{\partial s_k^2}\left( 0 \right)=\Delta_M g \left( x \right),\quad\sum_{k=1}^dR_{kk}\left( x \right)=\mathrm{Scal}^M\left( x \right).
  \end{aligned}
\end{equation*}
\end{proof}

Before applying Lemma~\ref{lem:basic_kernel_integral} to compute the infinitesimal generators of $H_{\epsilon}^{\left( \alpha \right)}$ and $H_{\epsilon,\delta}^{\left( \alpha \right)}$, we need more local information about $f \left( x, P_{yx}v \right)$ near $\left( x,v \right)$. To this end, let $\left\{ X_1,\cdots,X_d \right\}$ be a local geodesic frame on $U$ at $x$, and denote $\left\{\bar{X}_1,\cdots,\bar{X}_d\right\}$ for the horizontal lift of this frame; in addition, let $\left\{ V_1,\cdots,V_n \right\}$ be vertical vector fields on $E$ such that
$$\left\{ \bar{X}_1 \left( e \right),\cdots,\bar{X}_d \left( e \right), V_1 \left( e \right),\cdots,V_n \left( e \right) \right\}$$
constitutes an orthonormal basis for all $e$ in a sufficiently small neighborhood of $\left(x,v\right)$ contained in $\pi^{-1}\left( U \right)$. Write $\left\{ \theta^1,\cdots,\theta^d,\phi^1,\cdots\phi^n \right\}$ for the $1$-forms dual to the vector fields $\left\{ \bar{X}_1,\cdots,\bar{X}_d,V_1,\cdots,V_n \right\}$, i.e.,
\begin{equation*}
  \begin{aligned}
    &\theta^j \left( \bar{X}_k \right)=\delta^j_k,\quad \theta^j \left( V_\ell \right)=0,\\
    &\phi^m \left( \bar{X}_k \right)=0,\quad \phi^m \left( V_{\ell} \right)=\delta^m_{\ell},
  \end{aligned}
\end{equation*}
for all $1\leq j,k\leq d$, $1\leq \ell, m\leq n$.

If $\gamma$ is a unit speed geodesic on $M$ starting at $x$, recall from Definition~\ref{def:bundle-parallel-transport} that $t\mapsto P_{\gamma \left( t \right), x}v$ is the unique horizontal lift of $\gamma$ with starting point $v\in F_x$, i.e.,
$$\bar{\gamma} \left( t \right) = P_{\gamma \left( t \right),x}v.$$
Since $\bar{\gamma}$ is horizontal, $\phi^m \left( \bar{\gamma}' \left( t \right) \right)=0$ for all $1\leq m\leq n$ and thus (adopting Einstein summation convention)
\begin{equation}
\label{eq:expansion_under_horizontal_frame}
  \bar{\gamma}' \left( t \right)=\theta^j \left( \bar{\gamma}' \left( t \right) \right)\bar{X}_j \left( t \right).
\end{equation}
Here, as well as in the rest of this appendix, we set
$$X_j \left( t \right)=X_j \left( \gamma \left( t \right) \right),\quad \bar{X}_j \left( t \right)=\bar{X}_j \left( \gamma \left( t \right) \right).$$
By \cite[Proposition 3.1]{Hermann1960}, $\bar{\gamma}\left( t \right)$ is a geodesic on $E$, thus
\begin{equation*}
  \begin{aligned}
    0&=\nabla^E_{\bar{\gamma}' \left( t \right)}\bar{\gamma}' \left( t \right)
    =\frac{d}{dt} \left[ \theta^j \left( \bar{\gamma}' \left( t \right) \right) \right]\bar{X}_j \left( t \right)+\theta^j \left( \bar{\gamma}' \left( t \right) \right)\theta^k \left( \bar{\gamma}' \left( t \right) \right)\nabla^E_{\bar{X}_k \left( t \right)}\bar{X}_j \left( t \right),
  \end{aligned}
\end{equation*}
which implies
\begin{equation}
\label{eq:bundle_geodesic}
  \frac{d}{dt} \left[ \theta^j \left( \bar{\gamma}' \left( t \right) \right) \right]\bar{X}_j \left( t \right) = -\theta^j \left( \bar{\gamma}' \left( t \right) \right)\theta^k \left( \bar{\gamma}' \left( t \right) \right)\nabla^E_{\bar{X}_k \left( t \right)}\bar{X}_j \left( t \right).
\end{equation}
In particular, the right hand side of \eqref{eq:bundle_geodesic} is horizontal. It follows that
\begin{equation*}
  \begin{aligned}
    \frac{d}{dt} \left[ \theta^j \left( \bar{\gamma}' \left( t \right) \right) \right]\bar{X}_j \left( t \right)&=-\theta^j \left( \bar{\gamma}' \left( t \right) \right)\theta^k \left( \bar{\gamma}' \left( t \right) \right)\mathscr{H}\nabla^E_{\bar{X}_k \left( t \right)}\bar{X}_j \left( t \right)\\
  &=-\theta^j \left( \bar{\gamma}' \left( t \right) \right)\theta^k \left( \bar{\gamma}' \left( t \right) \right)\langle\nabla^E_{\bar{X}_k \left( t \right)}\bar{X}_j \left( t \right), \bar{X}_i \left( t \right)\rangle_{\bar{\gamma}\left( t \right)} \bar{X}_i \left( t \right)\\
  &=-\theta^j \left( \bar{\gamma}' \left( t \right) \right)\theta^k \left( \bar{\gamma}' \left( t \right) \right)\langle\nabla^M_{X_k \left( t \right)}X_j \left( t \right), X_i \left( t \right)\rangle_{\gamma\left( t \right)} \bar{X}_i \left( t \right),
  \end{aligned}
\end{equation*}
where $\mathscr{H}$ is the horizontal projection as defined in \eqref{eq:horizontal_vertical_projections}. By linear independence,
\begin{equation}
  \label{eq:first_order_derivative}
  \frac{d}{dt} \left[ \theta^j \left( \bar{\gamma}' \left( t \right) \right) \right]=-\theta^i \left( \bar{\gamma}' \left( t \right) \right)\theta^k \left( \bar{\gamma}' \left( t \right) \right)\Gamma_{ik}^j \left( \gamma \left( t \right) \right),
\end{equation}
where $\Gamma_{ik}^j$ are the connection coefficients for the frame $\left\{ X_1,\cdots,X_d \right\}$ on $M$
\begin{equation*}
  \Gamma_{ik}^j = \langle\nabla^M_{X_k}X_j, X_i\rangle,\quad \forall 1\leq i,j,k\leq d.
\end{equation*}
Setting $t=0$ in \eqref{eq:first_order_derivative} to get
\begin{equation}
  \label{eq:first_order_derivative_at_zero}
  \frac{d}{dt}\Big|_{t=0} \left[ \theta^j \left( \bar{\gamma}' \left( t \right) \right) \right]=-\theta^i \left( \bar{\gamma}' \left( 0 \right) \right)\theta^k \left( \bar{\gamma}' \left( 0 \right) \right)\Gamma_{ik}^j \left( \gamma \left( 0 \right) \right) = 0
\end{equation}
where $\Gamma_{ik}^j \left( x \right)=0$ since we picked $\left\{ X_j\mid 1\leq j\leq d \right\}$ as a geodesic frame at $x$.

Now for any $f\in C^{\infty}\left( E \right)$ write
\begin{equation*}
  f \left( t \right):=f \left( \bar{\gamma} \left( t \right) \right)=f \left( \gamma \left( t \right), P_{\gamma \left( t \right), x}v \right).
\end{equation*}
Using \eqref{eq:expansion_under_horizontal_frame} and \eqref{eq:first_order_derivative_at_zero}, the first and second derivatives of $f \left( t \right)$ at $t=0$ can be written as
\begin{equation*}
  \begin{aligned}
    f'\left( 0 \right)&=\theta^j \left( \bar{\gamma}' \left( 0 \right) \right)\bar{X}_jf \left( 0 \right),\\
    f''\left( 0 \right)&=\frac{d}{dt}\Big|_{t=0} \left[ \theta^j \left( \bar{\gamma}' \left( t \right) \right) \right]\bar{X}_j f \left( t \right)+\theta^i \left( \bar{\gamma}' \left( 0 \right) \right)\theta^k \left( \bar{\gamma}' \left( 0 \right) \right)\bar{X}_k\bar{X}_j f \left( 0 \right)\\
    &=\theta^i \left( \bar{\gamma}' \left( 0 \right) \right)\theta^k \left( \bar{\gamma}' \left( 0 \right) \right)\bar{X}_k\bar{X}_i f \left( 0 \right).
  \end{aligned}
\end{equation*}
Furthermore, if we denote $\pi^{*}:\Lambda^{*}M\rightarrow\Lambda^{*}E$ for the pullback map, and write $\left\{\psi^j\mid 1\leq j\leq d\right\}$ for the dual $1$-forms to the geodesic frame $\left\{ X_j\mid 1\leq j\leq d \right\}$ on $M$, then $\theta^j=\pi^{*}\psi^j$ for all $1\leq j\leq d$ and
\begin{equation*}
  \theta^j \left( \bar{\gamma}' \left( 0 \right) \right)=\pi^{*}\psi^j \left( \bar{\gamma}' \left( 0 \right) \right)=\psi^j \left( \gamma' \left( 0 \right) \right).
\end{equation*}
Thus $\left( \theta^1 \left( \bar{\gamma}' \left( 0 \right) \right), \cdots, \theta^d \left( \bar{\gamma}' \left( 0 \right) \right) \right)$ is $\gamma' \left( 0 \right)$ represented in the geodesic normal coordinate system associated with the geodesic frame $\left\{ X_j\mid 1\leq j\leq d \right\}$. If we write $\sigma_j=\theta^j \left( \bar{\gamma}' \left( 0 \right) \right)$ and $s_j \left( t \right) = t\sigma_j$ for all $j=1,\cdots, d$, then $\sum_{j=1}^d\sigma_j^2=1$ and $\left( s_1,\cdots,s_d \right)$ are the geodesic coordinates of $\gamma \left( t \right)$ on $M$ with respect to the geodesic frame $\left\{ X_j\mid 1\leq j\leq d \right\}$. With this notation,
\begin{equation*}
  \begin{aligned}
    f'\left( 0 \right)&=\sigma_j\bar{X}_jf \left( 0 \right),\\
    f''\left( 0 \right)&=\sigma_i\sigma_k\bar{X}_i\bar{X}_kf \left( 0 \right).
  \end{aligned}
\end{equation*}
Using \eqref{eq:expansion_under_horizontal_frame}, \eqref{eq:first_order_derivative_at_zero}, and $f'' \left( t \right)$, it is straightforward to compute the third order derivative of $f$ at $t=0$:
\begin{equation*}
  f''' \left( 0 \right)=\sigma_i\sigma_j\sigma_k\bar{X}_i\bar{X}_j\bar{X}_kf \left( 0 \right),
\end{equation*}
hence the Taylor expansion of $f \left( t \right)$ near $t=0$ is
\begin{equation}
  \label{eq:taylor_function_compose_geodesic}
  \begin{aligned}
    &f \left( t \right)=f \left( 0 \right)+tf' \left( 0 \right)+\frac{t^2}{2}f'' \left( 0 \right)+\frac{t^3}{6}f''' \left( 0 \right)+O \left( t^4 \right)\\
    &=f \left( x,v \right)+t\sigma_j\bar{X}_jf \left( x,v \right)+\frac{t^2}{2}\sigma_i\sigma_k\bar{X}_i\bar{X}_kf \left( x,v \right)+\frac{t^3}{6}\sigma_i\sigma_j\sigma_k\bar{X}_i\bar{X}_j\bar{X}_kf \left( x,v \right)+O \left( t^4 \right).
  \end{aligned}
\end{equation}
This expansion immediately leads to the following lemma:
\begin{lemma}
\label{lem:kernel_parallel_transport}
  Following Lemma~\ref{lem:basic_kernel_integral}, let $P_{yx}:F_x\rightarrow F_y$ be as defined in Definition~\ref{def:bundle-parallel-transport}. For any $f\in C^{\infty}\left( E \right)$ and $v\in F_x$, as $\epsilon\rightarrow 0$,
\begin{equation}
  \label{eq:kernel_parallel_transport}
  \begin{aligned}
    \int_M &\Phi_{\epsilon}\left( x,y \right)f \left( y,P_{yx}v \right)\,d\mathrm{vol}_M \left( y \right)\\
    &=\epsilon^{\frac{d}{2}}\left\{m_0 f \left( x,v \right)+\epsilon \frac{m_2}{2}\left[ \Delta_H f \left( x,v \right)-\frac{1}{3}\mathrm{Scal}^M\left( x \right)f \left( x,v \right) \right]+O \left( \epsilon^2 \right) \right\},
  \end{aligned}
\end{equation}
where $m_0$, $m_2$ are constants, $\mathrm{Scal}^M\left( x \right)$ is the scalar curvature of $M$ at $x$, and $\Delta_H$ is the rough horizontal Laplacian on $E$ defined in \eqref{eq:bochner_horizontal_laplacian}.
\end{lemma}

\begin{proof}
  Let $U\subset M$ be a geodesic normal neighborhood around $x\in M$, and $\epsilon$ sufficiently small that any point in $U$ can be connected to $x$ with a geodesic of length less than $\epsilon^{1/2}$. Let $\left\{ X_j\mid 1\leq j\leq d \right\}$ be a geodesic frame on $E$, $s_1,\cdots,s_d$ geodesic normal coordinates on $U$ with respect to this geodesic frame, and
  \begin{equation*}
    r>0,\quad r^2 = \sum_{j=1}^{d}s_j^2, \quad \sigma_j=\frac{s_j}{r}\quad \textrm{for all } 1\leq j\leq d.
  \end{equation*}
Following the proof of Lemma~\ref{lem:basic_kernel_integral}, let $\tilde{s}_j, \tilde{r}$ be as defined in \eqref{eq:square_root_coc} and use \eqref{eq:taylor_function_compose_geodesic} in place of \eqref{eq:expansion_under_horizontal_frame},
\begin{align*}
    &\int_M \Phi_{\epsilon}\left( x,y \right)f \left( y, P_{yx}v \right)\,d\mathrm{vol}_M \left( y \right)=\int_{B_1\left( 0 \right)}\!\!\Phi \left(\tilde{r}^2\right)\tilde{f} \left(\sqrt{\epsilon}\,\tilde{s},v \right)\,d\mathrm{vol}_M \left(\sqrt{\epsilon}\,\tilde{s} \right)\\
    &=f \left( x,v \right)\!\!\int_{B_1 \left( 0 \right)}\!\!\!\Phi \left( \tilde{r}^2\right)d\mathrm{vol}_M \left(\sqrt{\epsilon}\,\tilde{s} \right)+\epsilon^{\frac{1}{2}}\bar{X}_jf \left( x,v \right)\int_{B_1 \left( 0 \right)}\!\!\!\tilde{s}_j\Phi \left( \tilde{r}^2\right)d\mathrm{vol}_M \left(\sqrt{\epsilon}\,\tilde{s} \right) \\
    &\qquad+\frac{\epsilon}{2}\bar{X}_i\bar{X}_kf \left( x,v \right)\int_{B_1 \left( 0 \right)}\!\!\!\tilde{s}_i\tilde{s}_k\Phi \left( \tilde{r}^2\right)d\mathrm{vol}_M \left(\sqrt{\epsilon}\,\tilde{s} \right)\\
    &\qquad+\frac{\epsilon^{\frac{3}{2}}}{6}\bar{X}_i\bar{X}_j\bar{X}_kf \left( x,v \right)\int_{B_1 \left( 0 \right)}\!\!\!\tilde{s}_i\tilde{s}_j\tilde{s}_k\Phi \left( \tilde{r}^2\right)d\mathrm{vol}_M \left(\sqrt{\epsilon}\,\tilde{s} \right)+O \left( \epsilon^2 \right).
\end{align*}
Again by symmetry of these integrals and \eqref{eq:volume_form_expansion}, this reduces to
\begin{equation*}
  \begin{aligned}
    &\epsilon^{\frac{d}{2}}\left[ m_0f \left( x,v \right)-\frac{\epsilon}{6}m_2\mathrm{Scal}^M\left( x \right)f \left( x,v \right)+\frac{\epsilon}{2}m_2\sum_{k=1}^d\bar{X}^2_kf \left( x,v \right)+O \left( \epsilon^2 \right) \right]\\
    &=\epsilon^{\frac{d}{2}}\left[ m_0f \left( x,v \right)+\epsilon \frac{m_2}{2}\left(\Delta_H f \left( x,v \right)-\frac{1}{3}\mathrm{Scal}^M\left( x \right)f \left( x,v \right)\right)+O \left( \epsilon^2 \right) \right],
  \end{aligned}
\end{equation*}
where $m_0, m_2$ are constants defined in \eqref{eq:base_constants}, and
\begin{equation*}
  \sum_{k=1}^d\bar{X}_k^2f \left( x,v \right) = \Delta_Hf \left( x,v \right)
\end{equation*}
as explained in \eqref{eq:bochner_geodesic_frame}.
\end{proof}

We are now ready to give the proof of Theorem~{\rm\ref{thm:horizontal_diffusion_generator}}.

\begin{proof}[Proof of Theorem{\rm~\ref{thm:horizontal_diffusion_generator}}]
  By the definition of $H_{\epsilon}^{\left( \alpha \right)}$ in \eqref{eq:horizontal_diffusion_operator}, for any $f\in C^{\infty}\left( E \right)$,
  \begin{equation*}
    \begin{aligned}
      H_{\epsilon}^{\left( \alpha \right)}f \left( x,v \right)&=\frac{\displaystyle\int_MK_{\epsilon}^{\left( \alpha \right)}\left( x,y \right)f \left( y, P_{yx}v \right)p \left( y \right)d\mathrm{vol}_M \left( y \right)}{\displaystyle\int_MK_{\epsilon}^{\left( \alpha \right)}\left( x,y \right)p \left( y \right)d\mathrm{vol}_M \left( y \right)}\\
    &=\frac{\displaystyle\int_MK_{\epsilon}\left( x,y \right)f \left( y, P_{yx}v \right)p \left( y \right)p_{\epsilon}^{-\alpha} \left( y \right)d\mathrm{vol}_M \left( y \right)}{\displaystyle\int_MK_{\epsilon}\left( x,y \right)p \left( y \right)p_{\epsilon}^{-\alpha} \left( y \right)d\mathrm{vol}_M \left( y \right)}.
    \end{aligned}
  \end{equation*}
By Lemma~\ref{lem:basic_kernel_integral},
\begin{equation*}
  \begin{aligned}
    p_{\epsilon}\left( y \right) &= \int_MK_{\epsilon}\left( x,y \right)p \left( \eta \right)d\mathrm{vol}_M \left( \eta \right)\\
    &=\epsilon^{\frac{d}{2}}\left\{m_0p \left( y \right)+\epsilon\frac{m_2}{2}\left( \Delta_Mp \left( y \right)-\frac{1}{3}\mathrm{Scal}^M \left( y \right)p \left( y \right) \right)+O \left( \epsilon^2 \right)\right\}.
  \end{aligned}
\end{equation*}
Using this expansion of $p_{\epsilon}$ and applying Lemma~\ref{lem:basic_kernel_integral} to the denominator of $H_{\epsilon}^{\left( \alpha \right)}$,
\begin{align*}
  &\int_MK_{\epsilon}\left( x,y \right)p \left( y \right)p_{\epsilon}^{-\alpha} \left( y \right)d\mathrm{vol}_M \left( y \right)\\
  &=\epsilon^{\frac{\left( 1-\alpha \right)d}{2}}m_0^{-\alpha}\Bigg[ m_0 p^{1-\alpha}\left( x \right)+\epsilon \frac{m_2}{2}\left( \Delta_M p^{1-\alpha}\left( x \right)-\frac{1}{3}\mathrm{Scal}^M\left( x \right)p^{1-\alpha}\left( x \right) \right)\\
  &\qquad\qquad\qquad-\alpha\epsilon \frac{m_2}{2}p^{-\alpha}\left( x \right)\left( \Delta_Mp \left( x \right)-\frac{1}{3}\mathrm{Scal}^M \left( x \right)p \left( x \right) \right) +O\left( \epsilon^2 \right) \Bigg].
\end{align*}
Similarly, apply Lemma~\ref{lem:kernel_parallel_transport} to the numerator of $H_{\epsilon}^{\left( \alpha \right)}$ to get
\begin{align*}
    &\int_MK_{\epsilon}\left( x,y \right)f \left( y,P_{yx}v \right)p \left( y \right)p_{\epsilon}^{-\alpha} \left( y \right)d\mathrm{vol}_M \left( y \right)\\
    &=\epsilon^{\frac{\left( 1-\alpha \right)d}{2}}m_0^{-\alpha}\Bigg\{m_0 \left(f\bar{p}^{1-\alpha}\right) \left( x,v \right)\\
    &\qquad\qquad+\epsilon \frac{m_2}{2}\left[ \Delta_H \left(f\bar{p}^{1-\alpha}\right) \left( x,v \right)-\frac{1}{3}\mathrm{Scal}^M\left( x \right)\left(f\bar{p}^{1-\alpha}\right) \left( x,v \right) \right]\\
    &\qquad\qquad-\alpha\epsilon \frac{m_2}{2}f \left( x,v \right)p^{-\alpha}\left( x \right)\left( \Delta_Mp \left( x \right)-\frac{1}{3}\mathrm{Scal}^M \left( x \right)p \left( x \right) \right)+O \left( \epsilon^2 \right) \Bigg\}.
\end{align*}
Noting that $\bar{p}=p\circ \pi$ and by \eqref{eq:laplacian_lift}
$$\Delta_H\bar{p}^{1-\alpha}=\Delta_Mp^{1-\alpha},$$
a direct computation (plus assumption~\eqref{eq:assum_density_below} for the density $p$) concludes
\begin{align*}
  H_{\epsilon}^{\left( \alpha \right)}f \left( x,v \right)=f \left( x,v \right)+\epsilon \frac{m_2}{2m_0}\frac{\left[\Delta_H \left( f\bar{p}^{1-\alpha}\right)-f\Delta_H\bar{p}^{1-\alpha}\right] \left( x,v \right)}{\bar{p}^{1-\alpha} \left( x,v \right)}+O \left( \epsilon^2 \right),
\end{align*}
whence \eqref{eq:horizontal_generator_theorem} follows.
\end{proof}

We now turn to the proof of Theorem~\ref{thm:warped_diffusion_generator}. The basic idea is to apply Lemma~\ref{lem:basic_kernel_integral} and Lemma~\ref{lem:kernel_parallel_transport} repeatedly in both vertical and horizontal directions.
\begin{lemma}
\label{lem:warped_kernel_integral}
Suppose $\mathscr{E}=\left(E,M,F,\pi\right)$ is a fibre bundle, $M$ is a smooth closed Riemannian manifold with $\mathrm{Inj}\left( M \right)>0$, and $E$ equipped with the Riemannian metric \eqref{eq:metric_on_total_manifold}. Assume $\mathrm{dim}\,M=d$ and $\mathrm{dim}\,F=n$. Let $K_{\epsilon,\delta}$ be defined as in \eqref{eq:warped_kernel} with $\epsilon\in \left(0,\mathrm{Inj}\left( M \right)^2  \right)$, $\delta= O \left( \epsilon \right)$. For any function $f\in C^{\infty}\left( E \right)$,
\begin{equation}
  \label{eq:numerator_taylor_expansion}
  \begin{aligned}
   &\int_M\!\int_{F_y}K_{\epsilon,\delta}\left( x,v; y,w \right)f \left( y,w \right)\,d\mathrm{vol}_{F_y}\left(w\right)d\mathrm{vol}_M\left(y\right)\\
    &=\epsilon^{\frac{d}{2}}\delta^{\frac{n}{2}}\Bigg\{m_0f \left( x,v \right)+\epsilon\frac{m_{21}}{2}\left( \Delta_Hf \left( x,v \right)-\frac{1}{3}\mathrm{Scal}^M\left( x \right)f \left( x,v \right) \right)\\
    &\phantom{aaaaaaaaaaaa}+\delta\frac{m_{22}}{2}\left(\Delta_E^Vf \left( x,v \right)-\frac{1}{3}\mathrm{Scal}^{F_x} \left( v \right)f \left( x,v \right)  \right)+O \left( \epsilon^2+\epsilon\delta+\delta^2 \right)\Bigg\},
  \end{aligned}
\end{equation}
where $m_0,m_{21},m_{22}$ are positive constants depending only on the kernel $K$ and the fibre bundle, $\mathrm{Scal}^M$, $\mathrm{Scal}^{F_x}$ are scalar curvatures of $M$, $F_x$ respectively, and $\Delta_H$, $\Delta_E^V$ are defined in \eqref{eq:bochner_horizontal_laplacian} and \eqref{eq:horizontal_vertical_laplacians}.
\end{lemma}
\begin{proof}
  By definition of $K_{\epsilon,\delta}$,
\begin{equation*}
  \begin{aligned}
    &\int_M\!\int_{F_y}K_{\epsilon,\delta}\left( x,v; y,w \right)\,f \left( y,w \right)\,d\mathrm{vol}_{F_y}\left(w\right)d\mathrm{vol}_M\left(y\right)\\
    &=\int_M\!\int_{F_y}K\left( \frac{d^2_M \left( x,y \right)}{\epsilon}, \frac{d^2_{F_y}\left(P_{yx}v,w  \right)}{\delta} \right)f \left( y,w \right)\,d\mathrm{vol}_{F_y}\left(w\right)d\mathrm{vol}_M\left(y\right).
  \end{aligned}
\end{equation*}
For any fixed $y\in M$, apply Lemma~\ref{lem:basic_kernel_integral} to the inner integral over $F_y$ with
\begin{equation*}
  \Phi \left( p,q \right)=K\left( \frac{d^2_M \left( x,y \right)}{\epsilon}, \frac{d^2_{F_y}\left(p,q\right)}{\delta} \right)
\end{equation*}
then the constants $m_0,m_2$ will depend on $d^2_M \left( x,y \right)/\epsilon$. More specifically, if we set
\begin{align*}
    M_0 \left( r^2 \right)&=\int_{B_1^n \left( 0 \right)}K \left( r^2,\rho^2 \right)d\theta_1\cdots d\theta_n,\quad  M_2\left( r^2 \right)=\int_{B_1^n \left( 0 \right)}\theta_1^2K \left( r^2,\rho^2 \right)d\theta_1\cdots d\theta_n,\\
    M_3 \left( r^2 \right)&=\int_{B_1^n \left( 0 \right)}\theta_1^4K \left( r^2,\rho^2 \right)d\theta_1\cdots d\theta_n,\quad\textrm{where }\rho^2 = \sum_{j=1}^n\theta_j^2
\end{align*}
and recall from \eqref{eq:restriction_vertical_laplacian} that $\Delta_E^V$ coincides with $\Delta_{F_y}$ if one restricts a smooth function in $C^{\infty}\left( E \right)$ to $F_y$, then Lemma~\ref{lem:basic_kernel_integral} leads to
\begin{align*}
  &\int_{F_y}K_{\epsilon,\delta}\left( x,v; y,w \right)f \left( y,w \right)\,d\mathrm{vol}_{F_y}\left(w\right)\\
  &=\delta^{\frac{n}{2}}\Bigg\{ M_0 \left( \frac{d^2_M \left( x,y \right)}{\epsilon} \right)f \left( y,P_{yx}v \right)+\frac{\delta}{2}M_2 \left( \frac{d^2_M \left( x,y \right)}{\epsilon} \right)\times\\
  &\quad\qquad\left[ \Delta_E^V\,f \left( y,P_{y,x}v \right)-\frac{1}{3}\mathrm{Scal}^{F_y} \left( P_{yx}v \right)f \left( y,P_{yx}v \right) \right]+O \left( \delta^2M_3 \left( \frac{d^2_M \left( x,y \right)}{\epsilon} \right) \right) \Bigg\},
\end{align*}
Now integrate over $M$ and apply Lemma~\ref{lem:kernel_parallel_transport} multiple times:
\begin{align*}
    \int_M & M_0 \left( \frac{d^2_M \left( x,y \right)}{\epsilon} \right)f \left( y,P_{yx}v \right)d\mathrm{vol}_M \left( y \right)\\
    & = \epsilon^{\frac{d}{2}}\left\{ m_0f \left( x,v \right)+\epsilon \frac{m_{21}}{2}\left( \Delta_Hf \left( x,v \right)-\frac{1}{3}\mathrm{Scal}^M \left( x \right)f \left( x,v \right) \right)+O \left( \epsilon^2 \right) \right\},\\
    \int_M & M_2 \left( \frac{d^2_M \left( x,y \right)}{\epsilon} \right)\left[ \Delta_E^V\,f \left( y,P_{y,x}v \right)-\frac{1}{3}\mathrm{Scal}^{F_y} \left( P_{yx}v \right)f \left( y,P_{yx}v \right) \right]d\mathrm{vol}_M \left( y \right)\\
    &=\epsilon^{\frac{d}{2}}\left\{ m_{22}\left[ \Delta_E^V\,f \left( x,v \right)-\frac{1}{3}\mathrm{Scal}^{F_x} \left( v \right)f \left( x, v \right) \right]+O \left( \epsilon \right) \right\},
\end{align*}
where the constants $m_0, m_{21}, m_{22}$ are determined by (writing $r^2=\sum_{j=1}^ds_j^2$)
\begin{equation*}
  m_0=\int_{B_1^d \left( 0 \right)}M_0 \left( r^2 \right)ds_1\cdots ds_d,
\end{equation*}
\begin{equation*}
  \begin{aligned}
   &m_{21}=\int_{B_1^d \left( 0 \right)}M_0 \left( r^2 \right)s_1^2\,ds_1\cdots ds_d,\quad m_{22}=\int_{B_1^d \left( 0 \right)}M_2 \left( r^2 \right)ds_1\cdots ds_d.
  \end{aligned}
\end{equation*}
Therefore
\begin{align*}
    &\int_M\!\int_{F_y}K_{\epsilon,\delta}\left( x,v; y,w \right)\,f \left( y,w \right)\,d\mathrm{vol}_{F_y}\left(w\right)d\mathrm{vol}_M\left(y\right)\\
    &=\epsilon^{\frac{d}{2}}\delta^{\frac{n}{2}}\Bigg\{m_0f \left( x,v \right)+\epsilon\frac{m_{21}}{2}\left( \Delta_Hf \left( x,v \right)-\frac{1}{3}\mathrm{Scal}^M\left( x \right)f \left( x,v \right) \right)\\
    &\phantom{aaaaaaaaaaaa}+\delta\frac{m_{22}}{2}\left(\Delta_E^Vf \left( x,v \right)-\frac{1}{3}\mathrm{Scal}^{F_x} \left( v \right)f \left( x,v \right)  \right)+O \left( \epsilon^2+\epsilon\delta+\delta^2 \right)\Bigg\}.
\end{align*}
\end{proof}

\begin{proof}[Proof of Theorem{\rm~\ref{thm:warped_diffusion_generator}}]
Note that
\begin{equation}
\begin{aligned}
\label{eq:rough_cancellation}
  &H_{\epsilon,\delta}^{\left( \alpha \right)} f \left( x,v \right) = \frac{\displaystyle\int_M\!\int_{F_y}K_{\epsilon,\delta}^{\left( \alpha \right)}\left( x,v;y,w \right)f \left( y,w \right)p \left( y,w \right)d\mathrm{vol}_{F_y}\left( w \right)d\mathrm{vol}_M \left( y \right)}{\displaystyle\int_M\!\int_{F_y}K_{\epsilon,\delta}^{\left( \alpha \right)}\left( x,v;y,w \right)p \left( y,w \right)d\mathrm{vol}_{F_y}\left( w \right)d\mathrm{vol}_M \left( y \right)}\\
  &=\frac{\displaystyle\int_M\!\int_{F_y}K_{\epsilon,\delta}\left( x,v;y,w \right)f \left( y,w \right)p_{\epsilon,\delta}^{-\alpha}\left( y,w \right)p \left( y,w \right)d\mathrm{vol}_{F_y}\left( w \right)d\mathrm{vol}_M \left( y \right)}{\displaystyle\int_M\!\int_{F_y}K_{\epsilon,\delta}\left( x,v;y,w \right)p_{\epsilon,\delta}^{-\alpha}\left( y,w \right)p \left( y,w \right)d\mathrm{vol}_{F_y}\left( w \right)d\mathrm{vol}_M \left( y \right)}.
\end{aligned}
\end{equation}
Applying Lemma~\ref{lem:warped_kernel_integral} to $p_{\epsilon,\delta}$ to get
\begin{align*}
 &p_{\epsilon,\delta}\left( y,w \right)=\epsilon^{\frac{d}{2}}\delta^{\frac{n}{2}}\Bigg\{m_0p \left( y,w \right)+\epsilon\frac{m_{21}}{2}\left( \Delta_Hp \left( y,w \right)-\frac{1}{3}\mathrm{Scal}^M\left( y \right)p \left( y,w \right) \right)\\
    &\phantom{aaaaaaaaaaaa}+\delta\frac{m_{22}}{2}\left(\Delta_E^Vp \left( y,w \right)-\frac{1}{3}\mathrm{Scal}^{F_y} \left( w \right)p \left( y,w \right)  \right)+O \left( \epsilon^2+\epsilon\delta+\delta^2 \right)\Bigg\}.
\end{align*}
Using this and applying Lemma~\ref{lem:basic_kernel_integral}, Lemma~\ref{lem:kernel_parallel_transport} to the denominator and numerator of \eqref{eq:rough_cancellation} respectively:
\begin{align*}
   &\int_M\!\int_{F_y}K_{\epsilon,\delta}\left( x,v;y,w \right)p_{\epsilon,\delta}^{-\alpha}\left( y,w \right)p \left( y,w \right)d\mathrm{vol}_{F_y}\left( w \right)d\mathrm{vol}_M \left( y \right)\\
   &=\epsilon^{\frac{\left( 1-\alpha \right)d}{2}}\delta^{\frac{\left( 1-\alpha \right)n}{2}}m_0^{1-\alpha}p^{1-\alpha} \left( x,v \right)\Bigg\{1+\epsilon \frac{m_{21}}{2m_0}\left( \frac{\Delta_Hp^{1-\alpha}\left( x,v \right)}{p^{1-\alpha}\left( x,v \right)}-\frac{1}{3}\mathrm{Scal}^M \left( x \right)\right)\\
   &+\delta\frac{m_{22}}{2m_0}\left( \frac{\Delta_E^Vp^{1-\alpha}\left( x,v \right)}{p^{1-\alpha}\left( x,v \right)}-\frac{1}{3}\mathrm{Scal}^{F_x} \left( x \right)\right)-\alpha\epsilon \frac{m_{21}}{2m_0}\left( x,v \right)\left( \frac{\Delta_Hp \left( x,v \right)}{p \left( x,v \right)}-\frac{1}{3}\mathrm{Scal}^M \left( x \right)\right)\\
   &-\alpha\delta \frac{m_{22}}{2m_0}\left( \frac{\Delta_E^Vp \left( x,v \right)}{p \left( x,v \right)}-\frac{1}{3}\mathrm{Scal}^{F_x}\left( w \right) \right)+O \left( \epsilon^2+\epsilon\delta+\delta^2 \right) \Bigg\},
\end{align*}
\begin{align*}
&\int_M\!\int_{F_y}K_{\epsilon,\delta}\left( x,v;y,w \right)f \left( y,w \right)p_{\epsilon,\delta}^{-\alpha}\left( y,w \right)p \left( y,w \right)d\mathrm{vol}_{F_y}\left( w \right)d\mathrm{vol}_M \left( y \right)\\
   &=\epsilon^{\frac{\left( 1-\alpha \right)d}{2}}\delta^{\frac{\left( 1-\alpha \right)n}{2}}m_0^{1-\alpha}\left(fp^{1-\alpha}\right) \left( x,v \right)\Bigg\{1+\epsilon \frac{m_{21}}{2m_0}\left( \frac{\Delta_H\left(fp^{1-\alpha}\right)\left( x,v \right)}{\left(fp^{1-\alpha}\right)\left( x,v \right)}-\frac{1}{3}\mathrm{Scal}^M \left( x \right)\right)\\
   &+\delta\frac{m_{22}}{2m_0}\left( \frac{\Delta_E^V\left(fp^{1-\alpha}\right)\left( x,v \right)}{\left(fp^{1-\alpha}\right)\left( x,v \right)}-\frac{1}{3}\mathrm{Scal}^{F_x} \left( x \right)\right)-\alpha\epsilon \frac{m_{21}}{2m_0}\left( x,v \right)\left( \frac{\Delta_Hp \left( x,v \right)}{p \left( x,v \right)}-\frac{1}{3}\mathrm{Scal}^M \left( x \right)\right)\\
   &-\alpha\delta \frac{m_{22}}{2m_0}\left( \frac{\Delta_E^Vp \left( x,v \right)}{p \left( x,v \right)}-\frac{1}{3}\mathrm{Scal}^{F_x}\left( w \right) \right)+O \left( \epsilon^2+\epsilon\delta+\delta^2 \right) \Bigg\}.
\end{align*}
Combining these two expansions, a direct computation concludes
\begin{align*}
  &H_{\epsilon,\delta}^{\left( \alpha \right)} f \left( x,v \right) = f \left( x,v \right) + \epsilon \frac{m_{21}}{2m_0} \frac{\left[\Delta_H \left( fp^{1-\alpha} \right)-f\Delta_Hp^{1-\alpha} \right]\left( x,v \right)}{p^{1-\alpha}\left( x,v \right)}\\
  &\qquad\qquad\quad+\delta\frac{m_{22}}{2m_0} \frac{\left[\Delta_E^V \left( fp^{1-\alpha} \right)-f\Delta_E^Vp^{1-\alpha} \right]\left( x,v \right)}{p^{1-\alpha}\left( x,v \right)}+O \left( \epsilon^2+\epsilon\delta+\delta^2 \right).
\end{align*}
\end{proof}

\begin{proof}[Proof of Theorem{\rm~\ref{thm:warped_diffusion_generator_unbounded_ratio}}]
Since $P_{y,x}v$ does not depend on $\gamma$,
\begin{align*}
    \lim_{\gamma\rightarrow\infty}p_{\epsilon,\gamma\epsilon}\left( x,v \right)&=\lim_{\gamma\rightarrow\infty}\int_M\!\int_{F_y}K \left( \frac{d^2_M \left( x,y \right)}{\epsilon}, \frac{d^2_{S_y}\left(P_{y,x}v,w  \right)}{\gamma\epsilon} \right) p \left( y,w \right)d\mathrm{vol}\left(w\right)d\mathrm{vol}_M\left(y\right)\\
    &=\int_M\!K \left( \frac{d^2_M \left( x,y \right)}{\epsilon},0 \right) \Bigg[\int_{F_y}p \left( y,w \right)d\mathrm{vol}\left(w\right)\Bigg]d\mathrm{vol}_M\left(y\right).
\end{align*}
Define
\begin{equation*}
  \overline{K}_{\epsilon}\left( x,y \right)=\overline{K}\left( \frac{d^2_M \left( x,y \right)}{\epsilon} \right)=K \left( \frac{d^2_M \left( x,y \right)}{\epsilon},0 \right),\quad \overline{K}_{\epsilon}^{\left(\alpha\right)}\left( x,y \right)=\frac{\overline{K}_{\epsilon}\left( x,y \right)}{\langle p\rangle_{\epsilon}^{\alpha}\left( x \right)\langle p\rangle_{\epsilon}^{\alpha}\left( y \right)}.
\end{equation*}
By direct computation,
\begin{equation}
\label{eq:homogenized_H_alpha}
  \begin{aligned}
    \lim_{\gamma\rightarrow\infty}H_{\epsilon,\gamma\epsilon}^{\left(\alpha\right)}f \left( x,v \right)&=\frac{\displaystyle\int_M\overline{K}_{\epsilon}^{\left(\alpha\right)}\left( x,y \right)\Bigg[\int_{F_y}f \left( y,w \right)\frac{p \left( y,w \right)}{\langle p\rangle\left( y \right)}d\mathrm{vol}\left(w\right)\Bigg]\langle p\rangle\left( y \right)d\mathrm{vol}_M\left(y\right)}{\displaystyle\int_M\overline{K}_{\epsilon}^{\left(\alpha\right)}\left( x,y \right)\langle p\rangle\left( y \right)d\mathrm{vol}_M\left(y\right)}\\
    &=\frac{\displaystyle\int_M\overline{K}_{\epsilon}^{\left(\alpha\right)}\left( x,y \right)\langle f\rangle_p\left( y \right)\langle p\rangle\left( y \right)d\mathrm{vol}_M\left(y\right)}{\displaystyle\int_M\overline{K}_{\epsilon}^{\left(\alpha\right)}\left( x,y \right)\langle p\rangle\left( y \right)d\mathrm{vol}_M\left(y\right)}.
  \end{aligned}
\end{equation}
By \cite[Theorem 2]{CoifmanLafon2006}, as $\epsilon\rightarrow0$
\begin{equation}
\label{eq:homogenized_limit}
\begin{aligned}
  \lim_{\gamma\rightarrow\infty}&H^{\left(\alpha\right)}_{\epsilon,\gamma\epsilon}f \left( x,v \right)= \langle f\rangle_p \left( x \right) + \epsilon \frac{m_{2}'}{2m_0'} \frac{\left[\Delta_M \left( \langle f\rangle_p\langle p\rangle^{1-\alpha} \right)-\langle f\rangle_p\Delta_M\langle p\rangle^{1-\alpha} \right]\left( x \right)}{\langle p\rangle^{1-\alpha}\left( x \right)}+O \left( \epsilon^2 \right),
\end{aligned}
\end{equation}
where
\begin{equation*}
    m'_0=\int_{B_1^d \left( 0 \right)}\overline{K}\left( r^2 \right)ds_1\cdots ds_d,\quad
    m'_2=\int_{B_1^d \left( 0 \right)}\overline{K}\left( r^2 \right)s_1^2ds_1\cdots ds_d.
\end{equation*}
\end{proof}

\section{Proofs of Theorem~\ref{thm:utm_finite_sampling_noiseless} and Theorem~\ref{thm:utm_finite_sampling_noise}}
\label{sec:proof_variance}
In this appendix, we prove the two finite sampling theorems on unit tangent bundles in Section~\ref{sec:conv-rate-from}, following the paths paved by \cite{BelkinNiyogi2005,HeinAudibertVonLuxburg2007,Singer2006ConvergenceRate,SingerWu2012VDM}. Recall from Section~\ref{sec:unit-tangent-bundles} that for any $f\in C^{\infty}\left( UT\!M \right)$
\begin{align*}
  H_{\epsilon,\delta}^{\left(\alpha\right)}f \left( x,v \right)=\frac{\displaystyle\int_{UT\!M}K_{\epsilon,\delta}^{\left(\alpha\right)}\left( x,v;y,w \right)f \left( y,w \right)p \left( y,w \right)\,d\Theta \left( y,w \right)}{\displaystyle\int_{UT\!M}K_{\epsilon,\delta}^{\left(\alpha\right)}\left( x,v;y,w \right)p \left( y,w \right)\,d\Theta \left( y,w \right)},
\end{align*}
where $d\Theta \left( x,v \right)=d\mathrm{vol}_{S_x}d\mathrm{vol}_M \left( v \right)$ is the \emph{Liouville measure}. Since $S_x$ is a unit ball in $T_xM$, we shall also write $d\sigma_x=d\mathrm{vol}_{S_x}$ for convenience.

\subsubsection{Sampling without Noise}
\label{sec:proof-theor-noiseless}

The following lemma builds the bridge between the geodesic distance on the manifold and the Euclidean distance in the ambient space.

\begin{lemma}
\label{lem:distance_expansion}
Let $\iota:M\hookrightarrow \mathbb{R}^D$ be an isometric embedding of the smooth $d$-dimensional closed Riemannian manifold $M$ into $\mathbb{R}^D$. For any $x,y\in M$ such that $d_M \left( x,y \right)<\mathrm{Inj}\left( M \right)$, we have
\begin{equation}
  \label{eq:taylor_geodesic_distance}
  d_M^2 \left( x,y \right)=\left\|\iota \left( x \right)-\iota \left( y \right)  \right\|^2+\frac{1}{12}d_M^4 \left( x,y \right)\left\|\Pi \left( \theta,\theta \right) \right\|^2+O \left( d_M^5 \left( x,y \right)\right),
\end{equation}
where $\theta\in T_xM$, $\left\|\theta\right\|_x=1$ comes from the geodesic polar coordinates of $y$ in a geodesic normal neighborhood of $x$:
\begin{equation*}
  y=\mathrm{exp}_xr\theta,\quad r=d_M \left( x,y \right).
\end{equation*}
\end{lemma}
\begin{proof}
  See \cite[Proposition 6]{SWW2007}.
\end{proof}

For proving Theorem~\ref{thm:utm_finite_sampling_noiseless}, it is convenient to introduce the ``Euclidean distance version'' of the diffusion operators introduced in Section~\ref{sec:horiz-rand-walk}. Note that in Definition~\ref{defn:utm_finite_sampling_noiseless} the hat ``$\hat{\phantom{a}}$'' is used for empirical quantities; for the remainder of this appendix, the tilde ``$\tilde{\phantom{a}}$'' will be used for quantities in the definition of $H^{\left( \alpha \right)}_{\epsilon}$ and $H^{\left( \alpha \right)}_{\epsilon,\delta}$ with Euclidean distance in place of geodesic distance. For instance\footnote{Note that here $\hat{K}^{\left( \alpha \right)}_{\epsilon,\delta}=\tilde{K}_{\epsilon,\delta}^{\left(\alpha\right)}$, but this equality no longer holds in next subsection where $\hat{K}_{\epsilon,\delta}^{\left(\alpha\right)}$ is constructed from estimated parallel-transports.},

\begin{align*}
    \tilde{K}_{\epsilon,\delta}\left( x,v;y,w \right)&=K \left( \frac{\left\|x-y\right\|^2}{\epsilon},\frac{\left\| P_{y,x}v-w \right\|^2_y}{\delta} \right),\\
    \tilde{p}_{\epsilon,\delta}\left( x,v \right)&=\int_{UT\!M}\tilde{K}_{\epsilon,\delta}\left( x,v;y,w \right)p \left( y,w \right)d\Theta \left( y,w \right),\\
    \tilde{K}_{\epsilon,\delta}^{\left(\alpha\right)}\left( x,v;y,w \right)&=\frac{\tilde{K}_{\epsilon,\delta}\left( x,v;y,w \right)}{\tilde{p}_{\epsilon,\delta}^{\alpha}\left( x,v \right)\tilde{p}_{\epsilon,\delta}^{\alpha}\left( y,w \right)},
\end{align*}
and eventually
\begin{equation*}
  \begin{aligned}
    \tilde{H}_{\epsilon,\delta}^{\left(\alpha\right)}f \left( x,v \right)=\frac{\displaystyle\int_{UT\!M}\tilde{K}_{\epsilon,\delta}^{\left(\alpha\right)}\left( x,v;y,w \right)f \left( y,w \right)p \left( y,w \right)\,d\Theta \left( y,w \right)}{\displaystyle\int_{UT\!M}\tilde{K}_{\epsilon,\delta}^{\left(\alpha\right)}\left( x,v;y,w \right)p \left( y,w \right)\,d\Theta \left( y,w \right)}.
  \end{aligned}
\end{equation*}

The next step is to establish an asymptotic expansion of type \eqref{eq:warped_horizontal_generator_theorem} for $\tilde{H}_{\epsilon,\delta}^{\left(\alpha\right)}$. We deduce the following Lemma~\ref{lem:euclidean_kernel_integral}, the ``Euclidean distance version'' of Lemma~\ref{lem:basic_kernel_integral}, from Lemma~\ref{lem:distance_expansion} and Lemma~\ref{lem:basic_kernel_integral} itself.
\begin{lemma}
\label{lem:euclidean_kernel_integral}
  Let $\Phi:\mathbb{R}\rightarrow\mathbb{R}$ be a smooth function compactly supported in $\left[ 0,1 \right]$. Assume $M$ is a $d$-dimensional closed Riemannian manifold isometrically embedded in $\mathbb{R}^D$, with injectivity radius $\mathrm{Inj}\left( M \right)>0$. For any $\epsilon>0$, define kernel function
  \begin{equation}
    \hat{\Phi}_{\epsilon}\left( x,y \right)=\Phi \left( \frac{\left\| x-y\right\|^2}{\epsilon} \right)
  \end{equation}
on $M\times M$, where $\left\| \cdot \right\|$ is the Euclidean distance on $\mathbb{R}^D$. If the parameter $\epsilon$ is sufficiently small such that $0\leq\epsilon\leq\sqrt{\mathrm{Inj}\left( M \right)}$, then the integral operator associated with kernel $\Phi_{\epsilon}$
\begin{equation}
  \left(\hat{\Phi}_{\epsilon}\,g\right)\left( x \right):=\int_M \Phi_{\epsilon}\left( x,y \right)g \left( y \right)d\mathrm{vol}_M \left( y \right)
\end{equation}
has the following asymptotic expansion as $\epsilon\rightarrow 0$
\begin{equation}
  \left(\hat{\Phi}_{\epsilon}\,g\right)\left( x \right) = \epsilon^{\frac{d}{2}}\left[ m_0 g \left( x \right)+\epsilon \frac{m_2}{2}\left( \Delta_M g \left( x \right)+E\left( x \right)g \left( x \right) \right)+O \left( \epsilon^2 \right) \right],
\end{equation}
with
\begin{equation*}
  E \left( x \right) = -\frac{1}{3}\mathrm{Scal}^M\left( x \right)+\frac{d \left( d+2 \right)}{12}A \left( x \right)
\end{equation*}
where $m_0,m_2$ are constants that depend on the moments of $\Phi$ and the dimension $d$ of the Riemannian manifold $M$, $\Delta_M$ is the Laplace-Beltrami operator on $M$, $\mathrm{Scal}^M\left( x \right)$ is the scalar curvature of $M$ at $x$, and $A \left( x \right)$ is a scalar function on $M$ that only depends on the intrinsic dimension $d$ and the second fundamental form of the isometric embedding $\iota: M\hookrightarrow\mathbb{R}^D$.
\end{lemma}

\begin{proof}
  Since we already established Lemma~\ref{lem:basic_kernel_integral}, it suffices to expand the difference
\begin{equation}
\label{eq:distance_approximation_error}
  \int_M \left[\Phi \left( \frac{\left\| x-y\right\|^2}{\epsilon} \right) - \Phi \left( \frac{d^2_M \left( x,y \right)}{\epsilon} \right)\right] g \left( y \right)d\mathrm{vol}_M \left( y \right).
\end{equation}
Put $y$ in geodesic polar coordinates in a geodesic normal neighborhood of $x\in M$:
\begin{equation*}
  y=\mathrm{exp}_xr\theta, \quad\textrm{with } r=d_M \left( x,y \right), \theta\in T_xM, \left\|\theta\right\|_x=1,
\end{equation*}
and denote the geodesic normal coordinates around $x$ as $\left( s_1,\cdots,s_d \right)$. By Lemma~\ref{lem:distance_expansion},
\begin{equation*}
  \left\|x-y\right\|^2-d_M^2 \left( x,y \right)=-\frac{1}{12}d_M^4 \left( x,y \right)\left\|\Pi \left( \theta,\theta \right)\right\|^2+O \left( d_M^5 \left( x,y \right) \right)
\end{equation*}
thus
\begin{equation}
\label{eq:taylor_kernel_function}
  \begin{aligned}
    &\Phi \left( \frac{\left\| x-y\right\|^2}{\epsilon} \right) - \Phi \left( \frac{d^2_M \left( x,y \right)}{\epsilon} \right)\\
    &=\Phi' \left(\frac{d_M^2 \left( x,y \right)}{\epsilon}  \right)\cdot \left( -\frac{1}{12\epsilon}d_M^4 \left( x,y \right)\left\|\Pi \left( \theta,\theta \right)\right\|^2 \right)+O \left( \frac{d_M^8 \left( x,y \right)}{\epsilon^2} \right).
  \end{aligned}
\end{equation}
Recall that $\Phi$ is supported on the unit interval, which implies that in \eqref{eq:distance_approximation_error} only those $y\in M$ satisfying $\left\|x-y \right\|\leq \sqrt{\epsilon}$ or $d_M \left( x,y \right)\leq \sqrt{\epsilon}$ are involved. According to Lemma~\ref{lem:distance_expansion}, for sufficiently small $\epsilon>0$, $\left\|x-y\right\|\leq \sqrt{\epsilon}$ implies $d_M \left( x,y \right)<2\sqrt{\epsilon}$, thus the higher order error in \eqref{eq:taylor_kernel_function} is indeed
\begin{equation*}
  O \left( \frac{d_M^8 \left( x,y \right)}{\epsilon^2} \right)=O \left(\frac{\left(\sqrt{\epsilon}\right)^8}{\epsilon^2} \right)=O \left( \epsilon^2 \right).
\end{equation*}
Therefore,
\begin{equation}
\label{eq:error_involves_second_fundamental_form}
  \begin{aligned}
    \int_M & \left[\Phi \left( \frac{\left\| x-y\right\|^2}{\epsilon} \right) - \Phi \left( \frac{d^2_M \left( x,y \right)}{\epsilon} \right)\right] g \left( y \right)d\mathrm{vol}_M \left( y \right)\\
    &=-\frac{1}{12\epsilon}\int_M\Phi' \left(\frac{d_M^2 \left( x,y \right)}{\epsilon}  \right)d_M^4 \left( x,y \right)\left\|\Pi \left( \theta,\theta \right)\right\|^2g \left( y \right)d\mathrm{vol}_M \left( y \right)+\epsilon^{\frac{d}{2}}\cdot O \left( \epsilon^2 \right)\\
    &=-\frac{1}{12\epsilon}\int_M\Phi' \left(\frac{r^2}{\epsilon}  \right)r^4\left\|\Pi \left( \theta,\theta \right)\right\|^2g \left( y \right)d\mathrm{vol}_M \left( y \right)+\epsilon^{\frac{d}{2}}\cdot O \left( \epsilon^2 \right).
  \end{aligned}
\end{equation}
In geodesic normal coordinates $\left( s_1,\cdots,s_d \right)$,
\begin{equation}
\label{eq:put_geodesic_coordinates}
  \begin{aligned}
    \int_M& \Phi' \left(\frac{r^2}{\epsilon}  \right)r^4\left\|\Pi \left(\theta,\theta \right)\right\|^2g \left( y \right)d\mathrm{vol}_M \left( y \right)\\
    &=\int_{B_{\sqrt{\epsilon}}\left( 0 \right)}\Phi' \left(\frac{r^2}{\epsilon}  \right)r^4\left\|\Pi \left( \theta,\theta \right)\right\|^2\tilde{g}\left( s \right)\left[ 1-\frac{1}{6}R_{k\ell}\left( x \right)s_ks_{\ell}+O \left( r^3 \right) \right]ds_1\cdots ds_d.
  \end{aligned}
\end{equation}
Using the Taylor expansion of $\tilde{g} \left( s \right)$ around $s=0$ and the symmetry of the integral, \eqref{eq:put_geodesic_coordinates} reduces to
\begin{align*}
    \int_{B_{\sqrt{\epsilon}}\left( 0 \right)} & \Phi' \left(\frac{r^2}{\epsilon}  \right)r^4\left\|\Pi \left( \theta,\theta \right)\right\|^2 g \left( x \right)ds_1\cdots ds_d\\
    &+\int_{B_{\sqrt{\epsilon}}\left( 0 \right)}\Phi' \left(\frac{r^2}{\epsilon}  \right)r^4\left\|\Pi \left( \theta,\theta \right)\right\|^2 O \left( r^2 \right)ds_1\cdots ds_d\\
    &=\epsilon^{\frac{d}{2}}\cdot\epsilon^2g \left( x \right)\int_{S_1 \left( 0 \right)}\left\|\Pi \left( \theta,\theta \right)\right\|^2d\theta\int_0^1\Phi' \left(\tilde{r}^2\right)\tilde{r}^{3+d}d\tilde{r}+\epsilon^{\frac{d}{2}}\cdot O \left( \epsilon^3 \right).
\end{align*}
Let $m_2$ be the constant as in Lemma~\ref{lem:kernel_parallel_transport}, $\omega_{d-1}$ the volume of the standard unit sphere in $\mathbb{R}^d$. Note that
\begin{equation*}
  \omega_{d-1}\int_0^1\Phi \left( \tilde{r}^2 \right)\tilde{r}^{d+1}dr = \int_{B_1 \left( 0 \right)}\Phi \left( \tilde{r}^2 \right)\tilde{r}^2d\tilde{s}^1\cdots d\tilde{s}^d = m_2d.
\end{equation*}
Let $A \left( x \right)$ be the average of the length of the second fundamental form over the standard unit sphere, i.e.,
\begin{equation*}
  A \left( x \right)= \frac{1}{\omega_{d-1}}\int_{S_1 \left( 0 \right)}\left\|\Pi \left( \theta,\theta \right)\right\|^2d\theta.
\end{equation*}
Integrating the term involving $\Phi' \left( \tilde{r}^2 \right)$ by parts to get
\begin{equation}
\label{eq:last_piece}
  \begin{aligned}
    \epsilon^{\frac{d}{2}}\cdot\epsilon^2g \left( x \right)\int_{S_1 \left( 0 \right)}\left\|\Pi \left( \theta,\theta \right)\right\|^2d\theta\int_0^1\Phi' \left(\tilde{r}^2\right)\tilde{r}^{3+d}d\tilde{r}=-\epsilon^{\frac{d}{2}}\cdot \epsilon^2 \frac{m_2}{2}d \left( d+2 \right)g \left( x \right)A \left( x \right).
  \end{aligned}
\end{equation}
Therefore,
\begin{align*}
  \int_M &\left[\Phi \left( \frac{\left\| x-y\right\|^2}{\epsilon} \right) - \Phi \left( \frac{d^2_M \left( x,y \right)}{\epsilon} \right)\right] g \left( y \right)d\mathrm{vol}_M \left( y \right)\\
  &=\epsilon^{\frac{d}{2}}\left[ \epsilon \frac{m_2}{24}d \left( d+2 \right)A \left( x \right)g \left( x \right)+O \left( \epsilon^2 \right) \right],
\end{align*}
and thus
\begin{align*}
    \left(\hat{\Phi}_{\epsilon}\,g\right)\left( x \right)=\epsilon^{\frac{d}{2}}\left[ m_0 g \left( x \right)+\epsilon \frac{m_2}{2}\left( \Delta_M g \left( x \right)+E \left( x \right)g \left( x \right)\right) +O \left( \epsilon^2 \right) \right]
\end{align*}
where
\begin{equation*}
  E \left( x \right) := -\frac{1}{3}\mathrm{Scal}\left( x \right)+\frac{1}{12}d \left( d+2 \right)A \left( x \right).
\end{equation*}
\end{proof}

\begin{remark}
\label{rem:lazy_remark}
  The only difference between the conclusions in Lemma~\ref{lem:euclidean_kernel_integral} and Lemma~\ref{lem:basic_kernel_integral} is that the scalar function $E \left( x \right)$ takes the place of the scalar curvature $\mathrm{Scal}^M\left( x \right)$; one can check, essentially by going through the proof of Theorem~\ref{thm:warped_diffusion_generator}, that the proof still works through, due to the cancellation of the terms involving $E \left( x \right)$. In fact, applying Lemma~\ref{lem:euclidean_kernel_integral} repeatedly, one has
\begin{equation}
  \label{eq:kernel_parallel_transport_utm_hat}
  \begin{aligned}
    \int_M &\hat{\Phi}_{\epsilon}\left( x,y \right)f \left( y,P_{y,x}v \right)d\mathrm{vol}_M \left( y \right)\\
    &=\epsilon^{\frac{d}{2}}\left\{m_0 f \left( x,v \right)+\epsilon \frac{m_2}{2}\left[ \Delta_H f \left( x,v \right)+E_1 \left( x \right) f \left( x,v \right) \right]+O \left( \epsilon^2 \right) \right\},
  \end{aligned}
\end{equation}
and
\begin{equation}
  \label{eq:numerator_taylor_expansion_utm_hat}
  \begin{aligned}
    &\int_{UT\!M}\tilde{K}_{\epsilon,\delta}\left( x,v; y,w \right)g \left( y,w \right)\,d\Theta \left( y,w \right)\\
    &=\epsilon^{\frac{d}{2}}\delta^{\frac{d-1}{2}}\Bigg\{m_0g \left( x,v \right)+\epsilon\frac{m_{21}}{2}\left[ \Delta_{UT\!M}^Hg \left( x,v \right)+E_1\left( x \right)g \left( x,v \right) \right]\\
    &\phantom{aaaaaaaaaaaaaaaaa}+\delta\frac{m_{22}}{2}\left[\Delta_{UT\!M}^Vg \left( x,v \right)+E_2\cdot g \left( x,v \right)  \right]+O \left( \epsilon^2+\delta^2 \right) \Bigg\},
  \end{aligned}
\end{equation}
where
\begin{equation*}
  E_1 \left( \xi_i \right)=-\frac{1}{3}\mathrm{Scal}^M\left( \xi_i \right)+\frac{d \left( d+2 \right)}{12}\cdot \frac{1}{\omega_{d-1}}\int_{S_1 \left( 0 \right)}\left\|\Pi_M \left( \theta,\theta \right)\right\|^2d\theta
\end{equation*}
only depends on the scalar curvature $\mathrm{Scal}^M$ and the second fundamental form $\Pi_M$ of the base manifold $M $at $\xi$, and
\begin{equation*}
  E_2=-\frac{1}{3}\mathrm{Scal}^S+\frac{\left( d-1 \right)\left( d+1 \right)}{12}\cdot \frac{1}{\omega_{d-2}}\int_{S_1 \left( 0 \right)}\left\|\Pi_S \left( \theta,\theta \right)\right\|^2d\theta
\end{equation*}
is a constant because
\begin{equation*}
    \mathrm{Scal}^S \equiv \left( d-1 \right)\left( d-2 \right),\quad \left\|\Pi_S \left( \theta,\theta \right)\right\|^2 \equiv1\quad\textrm{for any unit tangent vector $\theta$}.
\end{equation*}
These expansions are essentially the equivalents of Lemma~\ref{lem:kernel_parallel_transport} and Lemma~\ref{lem:warped_kernel_integral} for $\tilde{K}_{\epsilon,\delta}$. Using \eqref{eq:kernel_parallel_transport_utm_hat}, \eqref{eq:numerator_taylor_expansion_utm_hat} and picking $\delta=O \left( \epsilon \right)$ as $\epsilon\rightarrow0$, a version of Theorem~\ref{thm:warped_diffusion_generator} holds true when $K_{\epsilon,\delta}^{\left(\alpha\right)}$ is replaced with $\tilde{K}_{\epsilon,\delta}^{\left(\alpha\right)}$, i.e., as $\epsilon\rightarrow0$ (and thus $\delta\rightarrow0$),
\begin{equation}
\label{eq:bias_error_term}
  \begin{aligned}
    \tilde{H}_{\epsilon,\delta}^{\left(\alpha\right)}&f\left( x,v \right) = f \left( x,v \right)+\epsilon \frac{m_{21}}{2m_0}\left[\frac{\Delta_H\left[fp^{1-\alpha}\right]\left( x,v \right)}{p^{1-\alpha}\left( x,v \right)}-f \left( x,v \right) \frac{\Delta_Hp^{1-\alpha}\left( x,v \right)}{p^{1-\alpha}\left( x,v \right)}  \right]\\
    &+\delta \frac{m_{22}}{2m_0}\left[\frac{\Delta_E^V\left[fp^{1-\alpha}\right]\left( x,v \right)}{p^{1-\alpha}\left( x,v \right)}-f \left( x,v \right) \frac{\Delta_E^Vp^{1-\alpha}\left( x,v \right)}{p^{1-\alpha}\left( x,v \right)}  \right]+O\left(\epsilon^2+\epsilon\delta+\delta^2 \right).\\
\end{aligned}
\end{equation}
As we shall see below, this observation is the key to establishing estimates for the bias error in the proof of Theorem~\ref{thm:utm_finite_sampling_noiseless}.
\end{remark}

Before we present the proof of Theorem~\ref{thm:utm_finite_sampling_noiseless}, we establish a large deviation bound for our two-step sampling strategy. Recall from Assumption~\ref{assum:utm_finite_sampling_noiseless} that we first sample $N_B$ points $\xi_1,\cdots,\xi_{N_B}$ i.i.d. with respect to $\langle p\rangle$ on the base manifold $M$, then sample $N_F$ points on each fibre $S_{\xi_j}$ i.i.d. with respect to $p \left( \cdot\mid\xi_j \right)$. The resulting $N_B\times N_F$ points on $UT\!M$
\begin{equation*}
  \begin{matrix}
    x_{1,1}, &x_{1,2}, &\cdots, &x_{1,N_F}\\
    x_{2,1}, &x_{2,2}, &\cdots, &x_{2,N_F}\\
    \vdots & \vdots &\cdots &\vdots\\
    x_{N_B,1}, &x_{N_B,2}, &\cdots, &x_{N_B,N_F}
  \end{matrix}
\end{equation*}
are generally not i.i.d. sampled from $UT\!M$. This forbids applying the \emph{Law of Large Numbers} directly to quantities that take the form of an average over the entire unit tangent bundle, such as
\begin{equation*}
  \frac{1}{N_BN_F}\sum_{j=1}^{N_B}\sum_{s=1}^{N_F}\hat{K}_{\epsilon,\delta}\left( x_{i,r},x_{j,s} \right)f \left( x_{j,s} \right).
\end{equation*}
However, due to the conditional i.i.d. fibrewise sampling, it makes sense to apply the law of large numbers to average quantities on a fixed fibre, e.g.,
\begin{equation*}
  \frac{1}{N_F}\sum_{s=1}^{N_F}\hat{K}_{\epsilon,\delta}\left( x_{i,r},x_{j,s} \right)f \left( x_{j,s} \right)\longrightarrow \mathbb{E}_Z \left[ \tilde{K}_{\epsilon,\delta}\left( x_{i,r},\left( \xi_j,Z \right) \right)f \left( \xi_j,Z \right) \right],
\end{equation*}
where $\mathbb{E}_Z$ stands for the expectation with respect to the ``fibre component'' of the coordinates of the points on $S_{\xi_j}$. Explicitly,
\begin{equation*}
  \mathbb{E}_Z \left[ \tilde{K}_{\epsilon,\delta}\left( x_{i,r},\left( \xi_j,\cdot \right) \right)f \left( \xi_j,\cdot \right) \right]=\int_{S_{\xi_j}}\tilde{K}_{\epsilon,\delta}\left( x_{i,r},\left( \xi_j,w \right) \right)f \left( \xi_j,w \right)p \left( w\mid \xi_j \right)d\sigma_{\xi_j} \left( w \right).
\end{equation*}
Next, note that $\xi_1,\cdots,\xi_{N_B}$ are i.i.d. sampled from the base manifold $M$, the partial expectations
\begin{equation*}
  \left\{ \mathbb{E}_Z \left[ \tilde{K}_{\epsilon,\delta}\left( x_{i,r},\left( \xi_j,Z \right) \right)f \left( \xi_j,Z \right) \right] \right\}_{j=1}^{N_B}
\end{equation*}
are i.i.d. random variables on $M$ with respect to $\langle p \rangle$. Thus
\begin{equation*}
  \frac{1}{N_B}\sum_{j=1}^{N_B}\mathbb{E}_Z \left[ \tilde{K}_{\epsilon,\delta}\left( x_{i,r},\left( \xi_j,Z \right) \right)f \left( \xi_j,Z \right) \right]\longrightarrow \mathbb{E}_Y\left[ \mathbb{E}_Z \left[\tilde{K}_{\epsilon,\delta}\left( x_{i,r},\left( Y,Z \right) \right)f \left(Y,Z \right) \right] \right],
\end{equation*}
where
\begin{equation*}
  \begin{aligned}
    &\mathbb{E}_Y\left[ \mathbb{E}_Z \left[\tilde{K}_{\epsilon,\delta}\left( x_{i,r},\left( Y,Z \right) \right)f \left(Y,Z \right) \right] \right]\\
   =&\int_M\langle p\rangle\left( y \right)\int_{S_{\xi_j}}\tilde{K}_{\epsilon,\delta}\left( x_{i,r},\left( \xi_j,w \right) \right)f \left( y,w \right)p \left( w\mid y \right)d\sigma_{y} \left( w \right)d\mathrm{vol}_M \left( y \right)\\
   =&\int_M\!\int_{S_{\xi_j}}\tilde{K}_{\epsilon,\delta}\left( x_{i,r},\left( \xi_j,w \right) \right)f \left( y,w \right)p \left( y,w \right)d\sigma_{y} \left( w \right)d\mathrm{vol}_M \left( y \right).
  \end{aligned}
\end{equation*}
This gives
\begin{equation*}
  \begin{aligned}
    \lim_{N_B\rightarrow\infty}&\lim_{N_F\rightarrow\infty}\frac{1}{N_BN_F}\sum_{j=1}^{N_B}\sum_{s=1}^{N_F}\hat{K}_{\epsilon,\delta}\left( x_{i,r},x_{j,s} \right)f \left( x_{j,s} \right)\\
    &=\int_M\!\int_{S_{\xi_j}}\tilde{K}_{\epsilon,\delta}\left( x_{i,r},\left( \xi_j,w \right) \right)f \left( y,w \right)p \left( y,w \right)d\sigma_{y} \left( w \right)d\mathrm{vol}_M \left( y \right),
  \end{aligned}
\end{equation*}
in which the two limits on the left hand side do not commute in general. Thus it is natural to consider iterated partial expectations rather than expectation on the entire $UT\!M$. From now on, we denote $\mathbb{E}_Y,\mathbb{E}_Z$ as $\mathbb{E}_1,\mathbb{E}_2$ for simplicity.
\begin{definition}
\label{defn:special_sampling}
Let $p$ be a probability density function on $UT\!M$, and
\begin{equation*}
  \langle p\rangle\left( x \right)=\int_{S_x}p \left( x,w \right)d\sigma_x \left( w \right),\quad p \left( v\mid x \right)=\frac{p \left( x,v \right)}{\langle p\rangle\left( x \right)}
\end{equation*}
be as defined in \eqref{eq:fibre_average}\eqref{eq:conditional_pdf}. For any $f\in C^{\infty}\left( M \right)$ and $g\in C^{\infty}\left( S_{\xi} \right)$, $\xi\in M$, define
\begin{align*}
  \mathbb{E}_1 f&:= \int_Mf \left( y \right)p \left( y \right)d\mathrm{vol}_M \left( y \right),\\
  \mathbb{E}_2^{\xi}g &:= \int_{S_{\xi}}g \left( \xi,w \right)p \left( w\mid \xi \right)d\sigma_{\xi}\left( w \right).
\end{align*}
\end{definition}

\begin{definition}
\label{defn:procrustesian}
Let $p$ be a probability density function on $UT\!M$. We call a collection of $N_B\times N_F$ real-valued random functions
\begin{equation*}
  \left\{ X_{j,s}\mid 1\leq j\leq N_B,1\leq s\leq N_F \right\}
\end{equation*}
\emph{Procrustean with respect to $p$ on $UT\!M$}, if
\begin{enumerate}[(i)]
\item\label{item:37} For each $1\leq j\leq N_B$, the subcollection $\left\{ X_{j,s}\mid 1\leq s\leq N_F \right\}$ are i.i.d. on $S_{\xi_j}$ for some $\xi_j\in M$, with respect to the conditional probability density $p \left( \cdot\mid \xi_j \right)$;
\item\label{item:38} The points $\left\{ \xi_j\mid 1\leq j\leq N_B \right\}$ are i.i.d. on $M$ with respect to the fibre average density $\langle p\rangle \left( \cdot \right)$.
\end{enumerate}
Due to \eqref{item:37}, we can drop the dependency of $X_{j,s}$ with respect to $s$ and simply write
\begin{equation*}
  \mathbb{E}_2^{\xi_j}X_j:=\mathbb{E}_2^{\xi_j}X_{j,s},\quad \mathbb{E}_2^{\xi_j}X_j^2:=\mathbb{E}_2^{\xi_j}X_{j,s}^2.
\end{equation*}
Similarly, because of \eqref{item:38} we can write
\begin{equation*}
  \mathbb{E}_1\mathbb{E}_2X:=\mathbb{E}_1\mathbb{E}_2^{\xi_j}X_j,\quad \mathbb{E}_1 \left( \mathbb{E}_2X \right)^2:=\mathbb{E}_1\left(\mathbb{E}_2^{\xi_j}X_j\right).
\end{equation*}

\end{definition}

\begin{lemma}
\label{lem:large_deviation_bound}
Let $\left\{ X_{j,s}\mid 1\leq j\leq N_B,1\leq s\leq N_F \right\}$ be a collection of Procrustean random functions with respect to some density function $p$ on $UT\!M$. If
\begin{equation*}
  \left| X_{j,s} \right|\leq M_0,\,\,\left| \mathbb{E}_2^{\xi_j}X_j \right|\leq M_1,\,\,\left| \mathbb{E}_1\mathbb{E}_2X \right|\leq M_2\quad\textrm{a.s. for all }1\leq j\leq N_B,1\leq s\leq N_F,
\end{equation*}
then for any $t>0$ and $0<\theta<1$,
\begin{align*}
    &\mathbb{P}\left\{ \frac{1}{N_BN_F}\sum_{j=1}^{N_B}\sum_{s=1}^{N_F}X_{j,s}-\mathbb{E}_1\mathbb{E}_2X>t \right\}\\
    &\leq \sum_{j=1}^{N_B}\exp \left\{-\frac{\displaystyle \frac{1}{2}\left(1-\theta\right)^2N_Ft^2}{\displaystyle \left[\mathbb{E}_2^{\xi_j}X_j^2-\left(\mathbb{E}_2^{\xi_j}X_j \right)^2\right]+ \frac{1}{3}\left(M_0+M_1\right)\left(1-\theta\right) t}  \right\}\\
    &+\exp \left\{ -\frac{\displaystyle \frac{1}{2}\theta^2N_Bt^2}{\displaystyle \left[\mathbb{E}_1 \left( \mathbb{E}_2X\right)^2-\left(\mathbb{E}_1\mathbb{E}_2X \right)^2\right]+\frac{1}{3}\left(M_1+M_2\right)\theta t} \right\}.
\end{align*}
\end{lemma}
\begin{proof}
  Note that for any $\theta\in \left( 0,1 \right)$
  \begin{align*}
      &\mathbb{P}\left\{ \frac{1}{N_BN_F}\sum_{j=1}^{N_B}\sum_{s=1}^{N_F}X_{j,s}-\mathbb{E}_1\mathbb{E}_2X>t \right\}\\
      &\leq \mathbb{P}\left\{\frac{1}{N_BN_F}\sum_{j=1}^{N_B}\sum_{s=1}^{N_F}X_{j,s}-\frac{1}{N_B}\sum_{j=1}^{N_B}\mathbb{E}_2^{\xi_j}X_j>\left(1-\theta\right) t \right\}\\
      &\quad+\mathbb{P}\left\{\frac{1}{N_B}\sum_{j=1}^{N_B}\mathbb{E}_2^{\xi_j}X_j-\mathbb{E}_1\mathbb{E}_2X> \theta t \right\}=: \left( \mathrm{I} \right)+\left( \mathrm{II} \right).
  \end{align*}
Since
\begin{equation*}
  \left| \mathbb{E}_2^{\xi_j}X_j-\mathbb{E}_1\mathbb{E}_2X \right|\leq M_1+M_2,
\end{equation*}
by Bernstein's Inequality~\cite[\S2.2]{Chung2006},
\begin{align*}
    \left( \mathrm{II} \right)&=\mathbb{P}\left\{ \sum_{j=1}^{N_B}\left( \mathbb{E}_2^{\xi_j}X_j -\mathbb{E}_1\mathbb{E}_2X\right)> \theta N_Bt \right\}\\
    &\leq \exp \left\{ -\frac{\displaystyle \frac{1}{2}\theta^2N_Bt}{\displaystyle\left[\mathbb{E}_1 \left(\mathbb{E}_2X\right)^2 -\left(\mathbb{E}_1\mathbb{E}_2X\right)^2\right]+\frac{1}{3} \left(M_1+M_2\right) \theta t} \right\}.
\end{align*}
For $\left( \mathrm{I} \right)$, a union bound plus Bernstein's Inequality gives
\begin{align*}
    \left( \mathrm{I} \right)&=\mathbb{P}\left\{\sum_{j=1}^{N_B}\left( \frac{1}{N_F}\sum_{s=1}^{N_F}X_{j,s}-\mathbb{E}_2^{\xi_j}X_j \right)>\left(1-\theta\right) N_Bt \right\}\\
    &\leq \sum_{j=1}^{N_B}\mathbb{P}\left\{ \sum_{s=1}^{N_F}\left( X_{j,s}-\mathbb{E}_2^{\xi_j}X_j \right)>\left(1-\theta\right) N_F t \right\}\\
    &\leq \sum_{j=1}^{N_B}\exp \left\{-\frac{\displaystyle \frac{1}{2}\left(1-\theta\right)^2N_F t^2}{\displaystyle \left[\mathbb{E}_2^{\xi_j}X_j^2-\left(\mathbb{E}_2^{\xi_j}X_j \right)^2\right]+\frac{1}{3}\left(M_0+M_1\right)\left(1-\theta\right) t}  \right\}.
\end{align*}
The conclusion follows from combining these two bounds.
\end{proof}
\begin{remark}
\label{rem:insight_large_deviation_bound}
  Intuitively, the second term in the bound comes from the sampling error on the base manifold, and is thus independent of $\delta$ and $N_F$; the first term in the bound comes from accumulating fibrewise sampling error across all $N_B$ fibres.
\end{remark}

\begin{proof}[Proof of Theorem{\rm ~\ref{thm:utm_finite_sampling_noiseless}}]
  We first establish the result for $\alpha=0$. In this case, $\hat{K}^{\left( 0 \right)}_{\epsilon,\delta}\left( \cdot,\cdot \right)=\hat{K}_{\epsilon,\delta}\left( \cdot,\cdot \right)$, and
\begin{align*}
    \hat{H}_{\epsilon,\delta}^{\left( 0 \right)}f \left( x_{i,r} \right)&=\frac{\displaystyle\sum_{j=1}^{N_B}\sum_{s=1}^{N_F}\hat{K}_{\epsilon,\delta}\left( x_{i,r},x_{j,s} \right)f \left( x_{j,s} \right)}{\displaystyle\sum_{j=1}^{N_B}\sum_{s=1}^{N_F}\hat{K}_{\epsilon,\delta}\left( x_{i,r},x_{j,s} \right)}\\
    &=\frac{\displaystyle \frac{1}{N_BN_F} \sum_{j=1}^{N_B}\sum_{s=1}^{N_F}K \left( \frac{\|\xi_i-\xi_j\|^2}{\epsilon}, \frac{\|P_{\xi_j,\xi_i}x_{i,r}-x_{j,s}\|^2}{\delta} \right)f \left( x_{j,s} \right)}{\displaystyle\frac{1}{N_BN_F}\sum_{j=1}^{N_B}\sum_{s=1}^{N_F}K \left( \frac{\|\xi_i-\xi_j\|^2}{\epsilon}, \frac{\|P_{\xi_j,\xi_i}x_{i,r}-x_{j,s}\|^2}{\delta} \right)}.
\end{align*}
Since $\left\{x_{j,s}\right\}_{s=1}^{N_F}$ are i.i.d. with respect to $p \left( \cdot\mid \xi_j \right)$, by the law of large numbers, for each fixed $j=1,\cdots,N_B$, as $N_F\rightarrow\infty$,
\begin{equation*}
  \begin{aligned}
    \lim_{N_F\rightarrow\infty}&\frac{1}{N_F}\sum_{s=1}^{N_F}K \left( \frac{\|\xi_i-\xi_j\|^2}{\epsilon}, \frac{\|P_{\xi_j,\xi_i}x_{i,r}-x_{j,s}\|^2}{\delta} \right)f \left( x_{j,s} \right)\\
    &= \int_{S_{\xi_j}}K \left( \frac{\|\xi_i-\xi_j\|^2}{\epsilon}, \frac{\|P_{\xi_j,\xi_i}x_{i,r}-w\|^2}{\delta} \right)f \left( \xi_j,w \right) p \left( w\mid\xi_j \right) d\sigma_{\xi_j} \left( w \right),
  \end{aligned}
\end{equation*}
Note that $\left\{\xi_j\right\}_{j=1}^{N_B}$ are i.i.d. with respect to $\langle p\rangle$, it follows again from the law of large numbers that
\begin{equation*}
  \begin{aligned}
    &\lim_{N_B\rightarrow\infty}\frac{1}{N_B}\sum_{j=1}^{N_B}\lim_{N_F\rightarrow\infty}\frac{1}{N_F}\sum_{s=1}^{N_F} K \left( \frac{\|\xi_i-\xi_j\|^2}{\epsilon}, \frac{\|P_{\xi_j,\xi_i}x_{i,r}-x_{j,s}\|^2}{\delta} \right)f \left( x_{j,s} \right)\\
    &=\int_{UT\!M}K \left( \frac{\|\xi_i-y\|^2}{\epsilon}, \frac{\|P_{y,\xi_i}x_{i,r}-w\|^2}{\delta} \right)f \left( y,w \right)p \left( y,w \right) d\Theta \left( y,w \right),
  \end{aligned}
\end{equation*}
where we used $p \left( y,w \right)=\langle p\rangle\left( y \right)p \left( w\mid y \right)$. For $f\equiv 1$,
\begin{equation*}
  \begin{aligned}
    &\lim_{N_B\rightarrow\infty}\frac{1}{N_B}\sum_{j=1}^{N_B}\lim_{N_F\rightarrow\infty}\frac{1}{N_F}\sum_{s=1}^{N_F} K \left( \frac{\|\xi_i-\xi_j\|^2}{\epsilon}, \frac{\|P_{\xi_j,\xi_i}x_{i,r}-x_{j,s}\|^2}{\delta} \right)\\
    &=\int_{UT\!M}K \left( \frac{\|\xi_i-y\|^2}{\epsilon}, \frac{\|P_{y,\xi_i}x_{i,r}-w\|^2}{\delta} \right)p \left( y,w \right) d\Theta \left( y,w \right).
  \end{aligned}
\end{equation*}
Therefore,
\begin{align*}
    \lim_{N_B\rightarrow\infty}&\lim_{N_F\rightarrow\infty}\hat{H}_{\epsilon,\delta}^{\left( 0 \right)}f \left( x_{i,r} \right)=\tilde{H}_{\epsilon,\delta}^{\left( 0 \right)}f \left( x_{i,r} \right)\\
    &=f \left( x_{i,r} \right)+\epsilon \frac{m_{21}}{2m_0}\left[\frac{\Delta_{UT\!M}^H\left[fp\right]\left( x_{i,r} \right)}{p\left( x_{i,r} \right)}-f \left( x_{i,r} \right) \frac{\Delta_{UT\!M}^Hp\left( x_{i,r} \right)}{p\left( x_{i,r} \right)}  \right]\\
    &+\delta \frac{m_{22}}{2m_0}\left[\frac{\Delta_{UT\!M}^V\left[fp\right]\left( x_{i,r} \right)}{p\left( x_{i,r} \right)}-f \left( x_{i,r} \right) \frac{\Delta_{UT\!M}^Vp\left( x_{i,r} \right)}{p\left( x_{i,r} \right)}  \right]+O \left( \epsilon^2+\epsilon\delta+\delta^2 \right).
\end{align*}
The last equality makes use of the assumption $\delta=O \left( \epsilon \right)$ as $\epsilon\rightarrow0$ and Remark~\ref{rem:lazy_remark}. This establishes the bias error for the special case $\alpha=0$ and it remains to estimate the variance error. To this end, denote
\begin{equation*}
  \begin{aligned}
    F_{j,s} = \hat{K}_{\epsilon,\delta}\left( x_{i,r},x_{j,s} \right)f \left( x_{j,s} \right),\qquad G_{j,s} = \hat{K}_{\epsilon,\delta}\left( x_{i,r},x_{j,s} \right)
  \end{aligned}
\end{equation*}
for any fixed $x_{i,r}\in UT\!M$. Note that $F_{i,s}=0$, $G_{i,s}=0$ for all $s=1,\cdots,N_F$, by Definition~\ref{defn:utm_finite_sampling_noiseless}~\eqref{item:34}; by the compactness of $UT\!M$ we have the following trivial bounds uniform in $j,s$:
\begin{equation*}
  \left| F_{j,s} \right|\leq \left\|K\right\|_{\infty}\left\|f\right\|_{\infty},\quad \left| G_{j,s} \right|\leq \left\|K\right\|_{\infty}.
\end{equation*}
Thus we already have
\begin{equation*}
  \lim_{N_B\rightarrow\infty}\lim_{N_F\rightarrow\infty}\hat{H}_{\epsilon,\delta}^{\left( 0 \right)}f \left( x_{i,r} \right)=\frac{\mathbb{E}_1\mathbb{E}_2 F}{\mathbb{E}_1\mathbb{E}_2 G},
\end{equation*}
and would like to estimate
\begin{equation*}
  p \left( N_B,N_F,\beta \right):=\mathbb{P} \left\{ \frac{\sum_j\sum_sF_{j,s}}{\sum_j\sum_sG_{j,s}}-\frac{\mathbb{E}_1\mathbb{E}_2 F}{\mathbb{E}_1\mathbb{E}_2 G}>\beta \right\}
\end{equation*}
for sufficiently small $\beta>0$. An upper bound for
\begin{equation*}
  \mathbb{P} \left\{ \frac{\sum_j\sum_sF_{j,s}}{\sum_j\sum_sG_{j,s}}-\frac{\mathbb{E}_1\mathbb{E}_2 F}{\mathbb{E}_1\mathbb{E}_2 G}<-\beta \right\}
\end{equation*}
can be obtained in a similar manner. Since $G_{j,s}>0$,
\begin{equation*}
  \begin{aligned}
    &p \left( N_B,N_F,\beta \right) = \mathbb{P} \left\{ \frac{\left(\sum_j\sum_sF_{j,s}\right)\mathbb{E}_1\mathbb{E}_2 G-\left( \sum_j\sum_sG_{j,s} \right)\mathbb{E}_1\mathbb{E}_2 F}{\left(\sum_j\sum_sG_{j,s}\right)\mathbb{E}_1\mathbb{E}_2 G} > \beta \right\}\\
    &=\mathbb{P} \left\{ \left(\sum_j\sum_sF_{j,s}\right)\mathbb{E}_1\mathbb{E}_2 G-\left( \sum_j\sum_sG_{j,s} \right)\mathbb{E}_1\mathbb{E}_2 F>\beta\left(\sum_j\sum_sG_{j,s}\right)\mathbb{E}_1\mathbb{E}_2 G \right\}.
  \end{aligned}
\end{equation*}
Denote
\begin{equation*}
  Y_{j,s}:=F_{j,s}\mathbb{E}_1\mathbb{E}_2G-G_{j,s}\mathbb{E}_1\mathbb{E}_2F+\beta \left( \mathbb{E}_1\mathbb{E}_2G-G_{j,s} \right)\mathbb{E}_1\mathbb{E}_2G,
\end{equation*}
then it is easily verifiable that $\mathbb{E}_1\mathbb{E}_2Y_{j,s}=0$ for all $1\leq j\leq N_B$, $1\leq s\leq N_F$, and
\begin{equation*}
  p \left( N_B,N_F,\beta \right)=\mathbb{P}\left\{ \frac{1}{N_BN_F}\sum_j\sum_s Y_{j,s}> \beta\left( \mathbb{E}_1\mathbb{E}_2G \right)^2 \right\}.
\end{equation*}
By Lemma~\ref{lem:large_deviation_bound}, bounding this quantity reduces to computing various moments. 
Define
\begin{equation*}
  X_j:=\mathbb{E}_2Y_j,
\end{equation*}
then $X_1,\cdots,X_{N_B}$ are i.i.d. on $M$ with respect to $\langle p \rangle$, and $\mathbb{E}_1X_j=0$ for $1\leq j\leq N_B$. Furthermore, $X_1,\cdots,X_{N_B}$ are uniformly bounded. To find this bound explicitly, note that
\begin{equation*}
  \begin{aligned}
    \left| X_j \right|&=\left| \mathbb{E}_2Y_j \right|=\left| \left(\mathbb{E}_2F_j\right)\mathbb{E}_1\mathbb{E}_2G-\left(\mathbb{E}_2G_j\right)\mathbb{E}_1\mathbb{E}_2F+\beta \left( \mathbb{E}_1\mathbb{E}_2G-\mathbb{E}_2G_j \right)\mathbb{E}_1\mathbb{E}_2G \right|\\
    &\leq \left|\left(\mathbb{E}_2F_j\right)\mathbb{E}_1\mathbb{E}_2G \right|+ \left|\left(\mathbb{E}_2G_j\right)\mathbb{E}_1\mathbb{E}_2F\right|+\beta \left( \mathbb{E}_1\mathbb{E}_2G \right)^2+\beta \left| \mathbb{E}_2G_j \right|\left| \mathbb{E}_1\mathbb{E}_2G \right|,
  \end{aligned}
\end{equation*}
and recall from Lemma~\ref{lem:euclidean_kernel_integral} and Remark~\ref{rem:lazy_remark} that
\begin{equation*}
  \begin{aligned}
    \mathbb{E}_1\mathbb{E}_2F &= O \left( \epsilon^{\frac{d}{2}}\delta^{\frac{d-1}{2}} \right),\quad \mathbb{E}_1\mathbb{E}_2G = O \left( \epsilon^{\frac{d}{2}}\delta^{\frac{d-1}{2}} \right),\\
    \mathbb{E}_2F_j&=O \left( \delta^{\frac{d-1}{2}} \right),\quad \mathbb{E}_2G_j=O \left( \delta^{\frac{d-1}{2}} \right),
  \end{aligned}
\end{equation*}
thus
\begin{equation*}
  \left| X_j \right|\leq  \tilde{C}\epsilon^{\frac{d}{2}}\delta^{d-1}+\beta \left( \epsilon^d\delta^{d-1}+\epsilon^{\frac{d}{2}}\delta^{d-1} \right)
\end{equation*}
where $\tilde{C}$ is some positive constant depending on the pointwise bounds of $K$, $p$, and $f$. Since we will be mostly interested in small $\beta>0$, let us pick $\beta = O \left( \epsilon^2+\epsilon\delta+\delta^2 \right)$ and rewrite the upper bound as
\begin{equation}
\label{eq:constant_C}
  \left| X_j \right|\leq C\epsilon^{\frac{d}{2}}\delta^{d-1}, \quad C=C \left(\left\|K\right\|_{\infty}, \left\|f\right\|_{\infty}, p_m, p_M\right)>0.
\end{equation}
We then need to bound $\mathbb{E}_1X_j^2$. Since
\begin{align*}
    \mathbb{E}_1 &X_j^2=\left[\mathbb{E}_1\left(\mathbb{E}_2F_j\right)^2\right]\left(\mathbb{E}_1\mathbb{E}_2G\right)^2+\left[\mathbb{E}_1\left(\mathbb{E}_2G_j\right)^2\right]\left(\mathbb{E}_1\mathbb{E}_2F\right)^2\\
    &-2 \mathbb{E}_1\left[\left( \mathbb{E}_2F_j \right)\left( \mathbb{E}_2G_j \right)\right]\left( \mathbb{E}_1\mathbb{E}_2F \right)\left( \mathbb{E}_1\mathbb{E}_2G \right)+\beta^2 \left( \mathbb{E}_1\mathbb{E}_2G \right)^2 \left[ \mathbb{E}_1 \left(\mathbb{E}_2G\right)^2-\left( \mathbb{E}_1\mathbb{E}_2G \right)^2 \right]\\
    &+2\beta\left(\mathbb{E}_1\mathbb{E}_2G\right) \left[ \mathbb{E}_1 \left( \mathbb{E}_2G \right)^2\mathbb{E}_1\mathbb{E}_2F-\left(\mathbb{E}_1\mathbb{E}_2G\right)\mathbb{E}_1 \left( \mathbb{E}_2F_j\mathbb{E}_2G_j \right) \right],
\end{align*}
it suffices to compute the first and second moments of $\mathbb{E}_2F_j$, $\mathbb{E}_2G_j$ for $1\leq j\leq N_B$. By~\eqref{eq:numerator_taylor_expansion_utm_hat},
\begin{align*}
  \mathbb{E}_1\mathbb{E}_2F&=\epsilon^{\frac{d}{2}}\delta^{\frac{d-1}{2}}\bigg\{m_0\left[fp\right] \left( x_{i,r} \right)+\epsilon \frac{m_{21}}{2}\left( \Delta_{UT\!M}^H\left[fp\right]\left( x_{i,r} \right)+E_1 \left( \xi_i \right)\left[ fp\right]\left( x_{i,r} \right) \right)\\
   &+\delta \frac{m_{22}}{2}\left( \Delta_{UT\!M}^V \left[ fp\right]\left( x_{i,r} \right)+E_2\cdot\left[fp\right]\left( x_{i,r} \right)\right) +O \left( \epsilon^2+\epsilon\delta+\delta^2 \right) \bigg\},\\
   \mathbb{E}_1\mathbb{E}_2G&=\epsilon^{\frac{d}{2}}\delta^{\frac{d-1}{2}}\bigg\{m_0p \left( x_{i,r} \right)+\epsilon \frac{m_{21}}{2}\left( \Delta_{UT\!M}^Hp\left( x_{i,r} \right)+E_1 \left( \xi_i \right)p\left( x_{i,r} \right) \right)\\
   &+\delta \frac{m_{22}}{2}\left( \Delta_{UT\!M}^Vp\left( x_{i,r} \right)+E_2 p\left( x_{i,r} \right)\right) +O \left( \epsilon^2+\epsilon\delta+\delta^2 \right) \bigg\}.
\end{align*}
Using the notation and applying Lemma~\ref{lem:euclidean_kernel_integral} onc,e
\begin{equation*}
  p \left( v\mid x \right)=\frac{p \left( x,v \right)}{\langle p\rangle\left( x \right)}=\frac{p \left( x,v \right)}{\langle p\rangle\circ\pi \left( x,v \right)}=\left[ \frac{p}{\langle p\rangle\circ\pi} \right]\left( x,v \right).
\end{equation*}
we have
\begin{align*}
  &\mathbb{E}_2F_j=\delta^{\frac{d-1}{2}}\bigg\{ M_0 \left( \frac{\left\|\xi_i-\xi_j\right\|^2}{\epsilon} \right)\left[ \frac{fp}{\langle p\rangle\circ\pi} \right]\left(P_{\xi_j,\xi_i}x_{i,r} \right)\\
  &+\frac{\delta}{2}M_2 \left( \frac{\left\|\xi_i-\xi_j\right\|^2}{\epsilon} \right)\left( \Delta_{UT\!M}^V \left[ \frac{fp}{\langle p\rangle\circ\pi} \right]+E_2\cdot\left[ \frac{fp}{\langle p\rangle\circ\pi} \right] \right)\left(P_{\xi_j,\xi_i}x_{i,r} \right)+O \left( \delta^2 \right) \bigg\},
\end{align*}
where $M_0 \left( \cdot \right)$, $M_2 \left( \cdot \right)$ are functions depending only on the kernel $K$, as in the proof of Lemma~\ref{lem:warped_kernel_integral}. By a direct computation using Lemma~\ref{lem:euclidean_kernel_integral},
\begin{align*}
  &\mathbb{E}_1\left( \mathbb{E}_2F_j \right)^2\\
  &=\epsilon^{\frac{d}{2}}\delta^{d-1}\bigg\{m'_0\left[ \frac{\left(fp\right)^2}{\langle p\rangle\circ\pi} \right]\left( x_{i,r} \right)+\epsilon \frac{m_{21}'}{2} \left( \Delta_{UT\!M}^H\left[ \frac{\left(fp\right)^2}{\langle p\rangle\circ\pi} \right]\left( x_{i,r} \right)+E_1 \left( \xi_i \right)\left[ \frac{\left(fp\right)^2}{\langle p\rangle\circ\pi} \right]\left( x_{i,r} \right) \right)\\
  &+\delta m_{22}' \left[ fp \right] \left( x_{i,r} \right)\left( \Delta_{UT\!M}^V \left[ \frac{fp}{\langle p\rangle\circ\pi} \right]\left( x_{i,r} \right)+E_2\cdot\left[ \frac{fp}{\langle p\rangle\circ\pi} \right]\left( x_{i,r} \right) \right)+O \left( \epsilon^2+\epsilon\delta+\delta^2 \right)\bigg\},
\end{align*}
where $m_0'$, $m_{21}'$, $m_{22}'$ are positive constants determined by the kernel function $K$ and dimension $d$:
\begin{align*}
    m_0'&=\int_{B_1^d \left( 0 \right)}M_0^2 \left( r^2 \right)ds^1\cdots ds^d,\quad r^2=\sum_{k=1}^d\left(s^k\right)^2,\\
    m_{21}'&=\int_{B_1^d \left( 0 \right)}M_0 \left( r^2 \right)\left( s^1 \right)^2ds^1\cdots ds^d,\quad m_{22}'=\int_{B_1^d \left( 0 \right)}M_0 \left( r^2 \right)M_2 \left( r^2 \right)ds^1\cdots ds^d.
\end{align*}
Setting $f\equiv 1$,
\begin{align*}
  &\mathbb{E}_1\left( \mathbb{E}_2G_j \right)^2\\
  &=\epsilon^{\frac{d}{2}}\delta^{d-1}\bigg\{m'_0\left[ \frac{p^2}{\langle p\rangle\circ\pi} \right]\left( x_{i,r} \right)+\epsilon \frac{m_{21}'}{2} \left( \Delta_{UT\!M}^H\left[ \frac{p^2}{\langle p\rangle\circ\pi} \right]\left( x_{i,r} \right)+E_1 \left( \xi_i \right)\left[ \frac{p^2}{\langle p\rangle\circ\pi} \right]\left( x_{i,r} \right) \right)\\
  &+\delta m_{22}' p \left( x_{i,r} \right) \left(\Delta_{UT\!M}^V \left[ \frac{p}{\langle p\rangle\circ\pi} \right]\left( x_{i,r} \right)+E_2\cdot\left[ \frac{p}{\langle p\rangle\circ\pi} \right]\left( x_{i,r} \right) \right)+O \left( \epsilon^2+\epsilon\delta+\delta^2 \right)\bigg\}.
\end{align*}
Similarly,
\begin{align*}
  &\mathbb{E}_1 \left[ \left(\mathbb{E}_2F_j\right)\left(\mathbb{E}_2G_j\right) \right]=\epsilon^{\frac{d}{2}}\delta^{d-1}\bigg\{m'_0 \left[ \frac{fp^2}{\langle p\rangle\circ\pi}\right]\left( x_{i,r} \right)\\
  &+\epsilon \frac{m_{21}'}{2} \left( \Delta_{UT\!M}^H\left[ \frac{fp^2}{\langle p\rangle\circ\pi} \right]\left( x_{i,r} \right)+E_1 \left( \xi_i \right)\left[ \frac{fp^2}{\langle p\rangle\circ\pi} \right]\right)\\
  &+\delta \frac{m_{22}'}{2} \bigg( p\left( x_{i,r} \right) \Delta_{UT\!M}^V \left[ \frac{fp}{\langle p\rangle\circ\pi} \right]\left( x_{i,r} \right)+\left[fp\right]\left( x_{i,r} \right) \Delta_{UT\!M}^V \left[ \frac{p}{\langle p\rangle\circ\pi} \right]\left( x_{i,r} \right)\\
  &+2E_2\cdot \left[ \frac{fp^2}{\langle p\rangle\circ\pi} \right]\left( x_{i,r} \right) \bigg)+O \left( \epsilon^2+\epsilon\delta+\delta^2 \right)\bigg\}.
\end{align*}
Take $\beta=O \left( \epsilon^2+\epsilon\delta+\delta^2 \right)$ so that $O \left( \beta \right)$ and $O \left( \beta^2 \right)$ terms are absorbed into
\begin{align*}
  O \left[\epsilon^{\frac{3d}{2}}\delta^{2 \left( d-1 \right)}\left( \epsilon^2+\epsilon\delta+\delta^2 \right)\right].
\end{align*}
Direct computation using
\begin{align*}
  \Delta_{UT\!M}^H \left( f^2g \right)+f^2\Delta_{UT\!M}^Hg-2f\Delta_{UT\!M}^H \left( fg \right)=2\left\|\nabla^H_{UT\!M}f\right\|^2g
\end{align*}
gives
\begin{align*}
  \mathbb{E}_1X_j^2&=\left[\mathbb{E}_1\left(\mathbb{E}_2F_j\right)^2\right]\left(\mathbb{E}_1\mathbb{E}_2G\right)^2+\left[\mathbb{E}_1\left(\mathbb{E}_2G_j\right)^2\right]\left(\mathbb{E}_1\mathbb{E}_2F\right)^2\\
   &\qquad-2 \mathbb{E}_1\left[\left( \mathbb{E}_2F_j \right)\left( \mathbb{E}_2G_j \right)\right]\left( \mathbb{E}_1\mathbb{E}_2F \right)\left( \mathbb{E}_1\mathbb{E}_2G \right)+O \left( \epsilon^2+\delta^2 \right)\\
   &=\epsilon^{\frac{3d}{2}}\delta^{2 \left( d-1 \right)}\bigg\{\epsilon m_0^2m_{21}'\left[ \frac{p^4}{\langle p\rangle\circ\pi} \right]\left( x_{i,r} \right)\left\|\nabla_{UT\!M}^H f \right\|^2 \left( x_{i,r} \right)+O \left( \epsilon^2+\epsilon\delta+\delta^2 \right)\bigg\}\\
   &\leq \epsilon^{\frac{3d}{2}}\delta^{2 \left( d-1 \right)}\left( C'\epsilon+O \left( \epsilon^2+\epsilon\delta+\delta^2 \right)\right)
\end{align*}
where
\begin{equation*}
  C'=\frac{m_0^2m_{21}'p_M^4\left\|\nabla_{UT\!M}^Hf\right\|_{\infty}^2}{\omega_{d-1}^4p_m}>0.
\end{equation*}
Note that $O \left( \delta \right)$ terms do not show up in this bound, intuitively because $X_j=\mathbb{E}_2Y_{j,s}$ is already the expectation along the fibre direction, which ``freezes'' the variability controlled by the fibrewise bandwidth $\delta$.

It remains to bound
\begin{equation*}
  \mathbb{E}_2^{\xi_j}Y_j^2-\left(\mathbb{E}_2^{\xi_j}Y_j\right)^2
\end{equation*}
for each $1\leq j\leq N_B$. Since we picked $\beta= O \left( \epsilon^2+\epsilon\delta+\delta^2 \right)$,
\begin{align*}
    \left| Y_{j,s} \right| = \left| F_{j,s}\mathbb{E}_1\mathbb{E}_2G-G_{j,s}\mathbb{E}_1\mathbb{E}_2F+\beta \left( \mathbb{E}_1\mathbb{E}_2G-G_{j,s} \right)\mathbb{E}_1\mathbb{E}_2G \right|\leq C\epsilon^{\frac{d}{2}}\delta^{\frac{d-1}{2}}
\end{align*}
where
\begin{equation*}
  C = C \left( \left\|K\right\|_{\infty},\left\|f\right\|_{\infty},p_m,p_M \right)
\end{equation*}
is a positive constant. Again taking advantage of $\beta=O \left( \epsilon^2+\epsilon\delta+\delta^2 \right)$, we have
\begin{equation*}
  \begin{aligned}
    \mathbb{E}_2Y_j^2-\left(\mathbb{E}_2Y_j\right)^2&=\left[ \mathbb{E}_2F_{j,s}^2- \left( \mathbb{E}_2F_j \right)^2 \right]\left(\mathbb{E}_1\mathbb{E}_2G\right)^2+\left[ \mathbb{E}_2G_{j,s}^2- \left( \mathbb{E}_2G_j \right)^2 \right]\left(\mathbb{E}_1\mathbb{E}_2F\right)^2\\
    &\quad+2 \left[ \left( \mathbb{E}_2F_j \right)\left( \mathbb{E}_2G_j \right)-\mathbb{E}_2 \left( F_{j,s}G_{j,s} \right) \right])\left(\mathbb{E}_1\mathbb{E}_2F\right)\left(\mathbb{E}_1\mathbb{E}_2G\right)\\
    &\quad+O \left[\epsilon^d\delta^{2 \left( d-1 \right)} \left( \epsilon^2+\epsilon\delta+\delta^2 \right)\right].
  \end{aligned}
\end{equation*}
Note that
\begin{equation*}
  \mathbb{E}_2F_{j,s}^2=O \left( \delta^{\frac{d-1}{2}} \right),\quad \mathbb{E}_2G_{j,s}^2=O \left( \delta^{\frac{d-1}{2}} \right),\quad \mathbb{E}_2 \left[ F_{j,s}G_{j,s} \right]=O \left( \delta^{\frac{d-1}{2}} \right),
\end{equation*}
but
\begin{equation*}
  \left(\mathbb{E}_2F_{j,s}\right)^2=O \left( \delta^{d-1} \right),\quad\left(\mathbb{E}_2G_{j,s}\right)^2=O \left( \delta^{d-1} \right),\quad\left( \mathbb{E}_2F_j \right)\left( \mathbb{E}_2G_j \right)=O \left( \delta^{d-1} \right),
\end{equation*}
the leading order error term in $\mathbb{E}_2Y_j^2-\left(\mathbb{E}_2Y_j\right)^2$ is
\begin{equation*}
  \left(\mathbb{E}_2F_{j,s}^2\right)\left(\mathbb{E}_1\mathbb{E}_2G\right)^2+\left(\mathbb{E}_2G_{j,s}^2\right)\left(\mathbb{E}_1\mathbb{E}_2F\right)^2-2\mathbb{E}_2 \left( F_{j,s}G_{j,s} \right)\left(\mathbb{E}_1\mathbb{E}_2F\right)\left(\mathbb{E}_1\mathbb{E}_2G\right).
\end{equation*}
By Lemma~\ref{lem:euclidean_kernel_integral},
\begin{align*}
    &\mathbb{E}_2F_{j,s}^2=\delta^{\frac{d-1}{2}}\bigg\{\widetilde{M}_0\left[\frac{f^2p}{\langle p\rangle\circ\pi}\right]\left( P_{\xi_j,\xi_i}x_{i,r} \right)\\
    &+\frac{\delta}{2}\widetilde{M}_2\bigg( \Delta_{UT\!M}^V \left[ \frac{f^2p}{\langle p\rangle\circ\pi} \right]\left(P_{\xi_j,\xi_i}x_{i,r}\right)+E_2\cdot\left[ \frac{f^2p}{\langle p\rangle\circ\pi} \right] \bigg)\left( P_{\xi_j,\xi_i}x_{i,r} \right)+O \left( \delta^2 \right)  \bigg\}.
\end{align*}
Similarly,
\begin{align*}
   \mathbb{E}_2G_{j,s}^2=&\delta^{\frac{d-1}{2}}\bigg\{\widetilde{M}_0\left[\frac{p}{\langle p \rangle\circ\pi}\right]\left( P_{\xi_j,\xi_i}x_{i,r} \right)\\
   &+\frac{\delta}{2}\widetilde{M}_2\bigg( \Delta_{UT\!M}^V \left[ \frac{p}{\langle p\rangle\circ\pi} \right]+E_2\cdot\left( \frac{p}{\langle p \rangle\circ\pi} \right) \bigg)\left( P_{\xi_j,\xi_i}x_{i,r} \right)+O \left( \delta^2 \right)  \bigg\},\\
   \mathbb{E}_2\left(F_{j,s}G_{j,s}\right)&=\delta^{\frac{d-1}{2}}\bigg\{\widetilde{M}_0\left[\frac{fp}{\langle p \rangle\circ\pi}\right]\left( P_{\xi_j,\xi_i}x_{i,r} \right)\\
   &+\frac{\delta}{2}\widetilde{M}_2\bigg( \Delta_{UT\!M}^V \left[ \frac{fp}{\langle p \rangle\circ\pi} \right]+E_2\cdot\left[ \frac{fp}{\langle p \rangle\circ\pi} \right] \bigg)\left( P_{\xi_j,\xi_i}x_{i,r} \right)+O \left( \delta^2 \right)  \bigg\}.\\
\end{align*}
Since the kernel $K$ is compactly supported and $f$ is Lipschitz ($UT\!M$ compact), the difference $ f \left( P_{\xi_j,\xi_i}x_{i,r} \right)-f \left( x_{i,r} \right)$ is of order $O \left(d_M \left( \xi_j,\xi_i \right)  \right)=O \left( \epsilon^{\frac{1}{2}} \right)$. Thus
\begin{equation*}
  \begin{aligned}
    &\left| \left(\mathbb{E}_2F_{j,s}^2\right)\left(\mathbb{E}_1\mathbb{E}_2G\right)^2+\left(\mathbb{E}_2G_{j,s}^2\right)\left(\mathbb{E}_1\mathbb{E}_2F\right)^2-2\mathbb{E}_2 \left( F_{j,s}G_{j,s} \right)\left(\mathbb{E}_1\mathbb{E}_2F\right)\left(\mathbb{E}_1\mathbb{E}_2G\right) \right|\\
    &\leq \epsilon^d\delta^{\frac{3 \left( d-1 \right)}{2}} \left( C'\epsilon+C''\delta \right),\quad C'>0,C''>0
  \end{aligned}
\end{equation*}
and
\begin{equation*}
  \mathbb{E}_2^{\xi_j}Y_j^2-\left(\mathbb{E}_2^{\xi_j}Y_j\right)^2=O \left( \epsilon^d\delta^{\frac{3 \left( d-1 \right)}{2}}\left( \epsilon+\delta \right) \right).
\end{equation*}
If we let $C''_1,C''_2$ be constants such that
\begin{equation*}
  C_1''\epsilon^{\frac{d}{2}}\delta^{\frac{d-1}{2}}\leq \left|\mathbb{E}_1\mathbb{E}_2G\right| \leq C_2''\epsilon^{\frac{d}{2}}\delta^{\frac{d-1}{2}},
\end{equation*}
then for any $\theta\in\left( 0,1 \right)$, by $\beta=O \left( \epsilon^2+\epsilon\delta+\delta^2 \right)$ and Lemma~\ref{lem:large_deviation_bound},
\begin{equation}
\label{eq:special_case_bound}
\begin{aligned}
  &p \left( N_B,N_F,\beta \right)\leq\\
   & N_B\exp \left\{ -\frac{\left(1-\theta\right)^2N_F\epsilon^d\delta^{\frac{d-1}{2}}\beta^2}{C_1 \left( \epsilon+\delta \right)+O\left(\epsilon^{\frac{d}{2}}\left( \epsilon^2+\delta^2 \right)\right)} \right\}+\exp \left\{ -\frac{\theta^2N_B\epsilon^{\frac{d}{2}}\beta^2}{C_2\epsilon+O \left( \epsilon^2+\delta^2 \right)} \right\}.
\end{aligned}
\end{equation}
As pointed out in Remark~\ref{rem:insight_large_deviation_bound}, the second term in this bound is the sampling error on the base manifold; the noise error resulted from this term is of the order
\begin{equation*}
  O\left[\left( N_B\epsilon^{\frac{d}{2}-1} \right)^{-\frac{1}{2}}\right]=O \left(N^{-\frac{1}{2}}_B\epsilon^{\frac{1}{2}-\frac{d}{4}} \right),
\end{equation*}
which is in accordance with the convergence rate obtained in \cite{Singer2006ConvergenceRate}. The first term in the bound reflects the accumulated fibrewise sampling error and grows linearly with respect to the number of fibres sampled, but can be reduced as one increases $N_F$ accordingly (which has an effect of reducing fibrewise sampling errors). The choice of $\theta$ is important: as $\theta$ increases from $0$ to $1$, the first term in the bound decreases but the second term increases. One may wish to pick an ``optimal'' $\theta\in \left( 0,1 \right)$, but this does not make sense unless one chooses $\epsilon,\delta,N_F$ appropriately so as to make the sum of the two terms smaller than $1$. Let us consider $\theta_{*}\in \left( 0,1 \right)$ satisfying
\begin{equation}
\label{eq:seraching_for_theta_star}
  \left(1-\theta_{*}\right)^2N_F\epsilon^d\delta^{\frac{d-1}{2}}=\theta_{*}^2N_B\epsilon^{\frac{d}{2}},
\end{equation}
or equivalently
\begin{equation}
\label{eq:theta_star}
  \epsilon^{\frac{d}{4}}\delta^{\frac{d-1}{4}}\sqrt{\frac{N_F}{N_B}}=\frac{\theta_{*}}{1-\theta_{*}}\Leftrightarrow\theta_{*}=\frac{\displaystyle \epsilon^{\frac{d}{4}}\delta^{\frac{d-1}{4}}\sqrt{\frac{N_F}{N_B}}}{\displaystyle 1+\epsilon^{\frac{d}{4}}\delta^{\frac{d-1}{4}}\sqrt{\frac{N_F}{N_B}}}.
\end{equation}
Setting $\theta=\theta_{*}$ in \eqref{eq:special_case_bound}, we have for some $C>0$
\begin{equation}
\label{eq:bound_with_theta_star}
  \begin{aligned}
    p \left( N_B,N_F,\beta \right)& \leq \left( N_B+1 \right)\exp \left\{ -\frac{\theta_{*}^2N_B\epsilon^{\frac{d}{2}}\beta^2}{C \left( \epsilon+\delta \right)} \right\}=\exp \left(-\frac{\theta_{*}^2N_B\epsilon^{\frac{d}{2}}\beta^2}{C \left( \epsilon+\delta \right)}+\log \left( N_B+1 \right) \right),
  \end{aligned}
\end{equation}
Since
\begin{equation*}
  \lim_{N_B\rightarrow\infty}\frac{N_B}{\log N_B}=\infty,
\end{equation*}
for any fixed $\epsilon,\delta$ we have $p \left( N_B,N_F,\beta \right)\rightarrow0$ as $N_B\rightarrow\infty$, as long as one increases $N_F$ accordingly so as to prevent $\theta_{*}$ from approaching $0$ or $1$; for instance, this is the case if the assumption \eqref{item:42} in Theorem~{\rm \ref{thm:utm_finite_sampling_noiseless}} is satisfied:
\begin{equation}
\label{eq:balance_sampling_number}
  \lim_{N_B\rightarrow\infty\atop N_F\rightarrow\infty}\frac{N_F}{N_B}=\beta\in \left( 0,\infty \right).
\end{equation}
This completes the proof for the pointwise convergence of $\hat{H}_{\epsilon,\delta}^{\left( 0 \right)}f$ in probability.

We now turn to the general case $\alpha\neq 0$. Recall that
\begin{equation*}
  \begin{aligned}
    \hat{H}_{\epsilon,\delta}^{\left( \alpha \right)} f \left( x_{i,r} \right)&=\frac{\displaystyle\sum_{j=1}^{N_B}\sum_{s=1}^{N_F}\hat{K}^{\left(\alpha\right)}_{\epsilon,\delta}\left( x_{i,r},x_{j,s} \right)f \left( x_{j,s} \right)}{\displaystyle\sum_{j=1}^{N_B}\sum_{s=1}^{N_F}\hat{K}^{\left(\alpha\right)}_{\epsilon,\delta}\left( x_{i,r},x_{j,s} \right)}=\frac{\displaystyle\sum_{j=1}^{N_B}\sum_{s=1}^{N_F}\frac{\hat{K}_{\epsilon,\delta}\left( x_{i,r},x_{j,s} \right)f \left( x_{j,s} \right)}{\hat{p}^{\alpha}_{\epsilon,\delta}\left( x_{i,r} \right)\hat{p}^{\alpha}_{\epsilon,\delta}\left( x_{j,s} \right)}}{\displaystyle\sum_{j=1}^{N_B}\sum_{s=1}^{N_F}\frac{\hat{K}_{\epsilon,\delta}\left( x_{i,r},x_{j,s} \right)}{\hat{p}^{\alpha}_{\epsilon,\delta}\left( x_{i,r} \right)\hat{p}^{\alpha}_{\epsilon,\delta}\left( x_{j,s} \right)}}
  \end{aligned}
\end{equation*}
where
\begin{equation*}
  \begin{aligned}
    \hat{p}\left( x_{j,s} \right)=\sum_{k=1}^{N_B}\sum_{t=1}^{N_F}\hat{K}_{\epsilon,\delta}\left( x_{j,s},x_{k,t} \right).
  \end{aligned}
\end{equation*}
By the law of large numbers,
\begin{equation*}
  \begin{aligned}
    \lim_{N_B\rightarrow\infty}\frac{1}{N_B}\lim_{N_F\rightarrow\infty}\frac{1}{N_F}\hat{p}\left( x_{j,s} \right)&=\int_{UT\!M}\tilde{K}_{\epsilon,\delta}\left( x_{i,r},\eta \right)p \left( \eta \right)d\Theta \left( \eta \right)\\
     &=\tilde{p} \left( x_{i,r} \right)=\mathbb{E}_1\mathbb{E}_2\left[\tilde{K}_{\epsilon,\delta}\left( x_{j,s},\cdot \right)\right].
  \end{aligned}
\end{equation*}
Therefore, as $N_B\rightarrow\infty,N_F\rightarrow\infty$, we expect $\hat{H}_{\epsilon,\delta}^{\left( \alpha \right)} f \left( x_{i,r} \right)$ to converge to
\begin{equation*}
  \begin{aligned}
    &\frac{\displaystyle\int_{UT\!M}\tilde{K}^{\left(\alpha\right)}_{\epsilon,\delta} \left( x_{i,r},\eta \right)f \left( \eta \right)p \left( \eta \right)d\Theta \left( y,w \right)}{\displaystyle\int_{UT\!M}\tilde{K}^{\left(\alpha\right)}_{\epsilon,\delta} \left( x_{i,r},\eta \right)p \left( \eta \right)d\Theta \left( y,w \right)}=\tilde{H}_{\epsilon,\delta}^{\left(\alpha\right)}f \left( x_{i,r} \right)\\
    =&f \left( x_{i,r} \right)+\epsilon \frac{m_{21}}{2m_0}\left(\frac{\Delta_{UT\!M}^H\left[fp^{1-\alpha}\right]\left( x_{i,r} \right)}{p^{1-\alpha}\left( x_{i,r} \right)}-f \left( x_{i,r} \right) \frac{\Delta_{UT\!M}^Hp^{1-\alpha}\left( x_{i,r} \right)}{p^{1-\alpha}\left( x_{i,r} \right)}  \right)\\
    &+\delta \frac{m_{22}}{2m_0}\left(\frac{\Delta_{UT\!M}^V\left[fp^{1-\alpha}\right]\left( x_{i,r} \right)}{p^{1-\alpha}\left( x_{i,r} \right)}-f \left( x_{i,r} \right) \frac{\Delta_{UT\!M}^Vp^{1-\alpha}\left( x_{i,r} \right)}{p^{1-\alpha}\left( x_{i,r} \right)}  \right)+O \left( \epsilon^2+\epsilon\delta+\delta^2 \right),
  \end{aligned}
\end{equation*}
which gives the same bias error $O \left( \epsilon^2+\epsilon\delta+\delta^2 \right)$ as in the $\alpha=0$ case.

It remains to estimate the variance error. By
\begin{equation*}
  \begin{aligned}
    \hat{H}_{\epsilon,\delta}^\alpha f \left( x_{i,r} \right)=\frac{\displaystyle\sum_{j=1}^{N_B}\sum_{s=1}^{N_F}\hat{K}_{\epsilon,\delta}\left( x_{i,r},x_{j,s} \right)\hat{p}^{-\alpha}_{\epsilon,\delta}\left( x_{j,s} \right)f \left( x_{j,s} \right)}{\displaystyle\sum_{j=1}^{N_B}\sum_{s=1}^{N_F}\hat{K}_{\epsilon,\delta}\left( x_{i,r},x_{j,s} \right)\hat{p}^{-\alpha}_{\epsilon,\delta}\left( x_{j,s} \right)},
  \end{aligned}
\end{equation*}
and
\begin{align*}
  &\frac{\displaystyle\sum_{j=1}^{N_B}\sum_{s=1}^{N_F}\hat{K}_{\epsilon,\delta}\left( x_{i,r},x_{j,s} \right)\hat{p}^{-\alpha}_{\epsilon,\delta}\left( x_{j,s} \right)f \left( x_{j,s} \right)}{\displaystyle\sum_{j=1}^{N_B}\sum_{s=1}^{N_F}\hat{K}_{\epsilon,\delta}\left( x_{i,r},x_{j,s} \right)\hat{p}^{-\alpha}_{\epsilon,\delta}\left( x_{j,s} \right)}-\frac{\displaystyle\sum_{j=1}^{N_B}\sum_{s=1}^{N_F}\hat{K}_{\epsilon,\delta}\left( x_{i,r},x_{j,s} \right)\tilde{p}^{-\alpha}_{\epsilon,\delta}\left( x_{j,s} \right)f \left( x_{j,s} \right)}{\displaystyle\sum_{j=1}^{N_B}\sum_{s=1}^{N_F}\hat{K}_{\epsilon,\delta}\left( x_{i,r},x_{j,s} \right)\tilde{p}^{-\alpha}_{\epsilon,\delta}\left( x_{j,s} \right)}\\
  =&\frac{\displaystyle\sum_{j=1}^{N_B}\sum_{s=1}^{N_F}\hat{K}_{\epsilon,\delta}\left( x_{i,r},x_{j,s} \right)\left[N_B^{\left(\alpha\right)}N^{\left(\alpha\right)}_F\hat{p}^{-\alpha}_{\epsilon,\delta}\left( x_{j,s} \right)-\tilde{p}^{-\alpha}_{\epsilon,\delta}\left( x_{j,s} \right)\right]f \left( x_{j,s} \right)}{\displaystyle\sum_{j=1}^{N_B}\sum_{s=1}^{N_F}\hat{K}_{\epsilon,\delta}\left( x_{i,r},x_{j,s} \right)N_B^{\alpha}N_F^{\alpha}\hat{p}^{-\alpha}_{\epsilon,\delta}\left( x_{j,s} \right)}\\
  &+\sum_{j=1}^{N_B}\sum_{s=1}^{N_F}\hat{K}_{\epsilon,\delta}\left( x_{i,r},x_{j,s} \right)\tilde{p}^{-\alpha}_{\epsilon,\delta}\left( x_{j,s} \right)f \left( x_{j,s} \right)\times\\
  &\quad\left[\frac{-\displaystyle \sum_{j=1}^{N_B}\sum_{s=1}^{N_F}\hat{K}_{\epsilon,\delta}\left( x_{i,r},x_{j,s} \right)\left[N_B^{\alpha}N_F^{\alpha}\hat{p}^{-\alpha}_{\epsilon,\delta}\left( x_{j,s} \right)-\tilde{p}^{-\alpha}_{\epsilon,\delta}\left( x_{j,s} \right)\right]}{\displaystyle\left(\sum_{j=1}^{N_B}\sum_{s=1}^{N_F}\hat{K}_{\epsilon,\delta}\left( x_{i,r},x_{j,s} \right)N_B^{\alpha}N_F^{\alpha}\hat{p}^{-\alpha}_{\epsilon,\delta}\left( x_{j,s} \right)\right)\left(\sum_{j=1}^{N_B}\sum_{s=1}^{N_F}\hat{K}_{\epsilon,\delta}\left( x_{i,r},x_{j,s} \right)\tilde{p}^{-\alpha}_{\epsilon,\delta}\left( x_{j,s} \right) \right)} \right]\\
  &=: \left( A \right)+ \left( B \right),
\end{align*}
thus if we estimate $\left( A \right)$, $\left( B \right)$ by controlling the error
\begin{equation*}
\left[N_B^{\alpha}N^{\alpha}_F\hat{p}^{-\alpha}_{\epsilon,\delta}\left( x_{j,s} \right)-\tilde{p}^{-\alpha}_{\epsilon,\delta}\left( x_{j,s} \right)\right]
\end{equation*}
then it suffices to estimate the variance error caused by
\begin{equation}
\label{eq:better_variance_error}
  \left[\sum_{j=1}^{N_B}\sum_{s=1}^{N_F}\frac{\hat{K}_{\epsilon,\delta}\left( x_{i,r},x_{j,s} \right)f \left( x_{j,s} \right)}{\tilde{p}^{\left(\alpha\right)}_{\epsilon,\delta}\left( x_{i,r} \right)\tilde{p}^{\left(\alpha\right)}_{\epsilon,\delta}\left( x_{j,s} \right)}\right]\left[\sum_{j=1}^{N_B}\sum_{s=1}^{N_F}\frac{\hat{K}_{\epsilon,\delta}\left( x_{i,r},x_{j,s} \right)}{\tilde{p}^{\left(\alpha\right)}_{\epsilon,\delta}\left( x_{i,r} \right)\tilde{p}^{\left(\alpha\right)}_{\epsilon,\delta}\left( x_{j,s} \right)}\right]^{-1}.
\end{equation}
Our previous proof for the special case $\alpha=0$ can then be applied to \eqref{eq:better_variance_error}: the only adjustment is to replace the kernel $\hat{K}_{\epsilon,\delta}\left( x,y \right)$ in that proof with the $\alpha$-normalized kernel
\begin{equation*}
  \frac{\tilde{K}_{\epsilon,\delta}\left( x,y \right)}{\tilde{p}_{\epsilon,\delta}^{\alpha}\left( x \right)\tilde{p}_{\epsilon,\delta}^{\alpha}\left( y \right)}.
\end{equation*}
We would like to estimate the tail probability
\begin{equation*}
  \mathbb{P}\left\{ \frac{1}{N_BN_F}\hat{p}_{\epsilon,\delta}\left( x_{j,s} \right)-\tilde{p}_{\epsilon,\delta}\left( x_{j,s} \right)>\beta \right\},
\end{equation*}
but since $\tilde{p}_{\epsilon,\delta}\left( x_{j,s} \right)=O \left( \epsilon^{\frac{d}{2}}\delta^{\frac{d-1}{2}} \right)$, it is not lower bounded away from $0$ as $\epsilon,\delta\rightarrow0$. We thus estimate the following tail probability instead:
\begin{equation*}
  \begin{aligned}
    q \left( N_B, N_F, \beta \right)&:=\mathbb{P}\left\{ \frac{1}{N_BN_F}\epsilon^{-\frac{d}{2}}\delta^{-\frac{d-1}{2}}\tilde{p}_{\epsilon,\delta}\left( x_{j,s} \right)-\epsilon^{-\frac{d}{2}}\delta^{-\frac{d-1}{2}}\tilde{p}_{\epsilon,\delta}\left( x_{j,s} \right)>\beta \right\}\\
    &=\mathbb{P}\left\{ \frac{1}{N_BN_F}\tilde{p}_{\epsilon,\delta}\left( x_{j,s} \right)-\tilde{p}_{\epsilon,\delta}\left( x_{j,s} \right)>\epsilon^{\frac{d}{2}}\delta^{\frac{d-1}{2}}\beta \right\},
  \end{aligned}
\end{equation*}
where
\begin{equation*}
  \begin{aligned}
    \hat{p}_{\epsilon,\delta}\left( x_{j,s} \right)=\sum_{k=1}^{N_B}\sum_{t=1}^{N_F}\hat{K}_{\epsilon,\delta}\left( x_{j,s},x_{k,t} \right),\quad\tilde{p}_{\epsilon,\delta}\left( x_{j,s} \right)=\mathbb{E}_1\mathbb{E}_2\left[\tilde{K}_{\epsilon,\delta}\left( x_{j,s},\cdot \right)\right].
  \end{aligned}
\end{equation*}
Noting that for some positive constant $C = C \left( \left\|K\right\|_{\infty},p_M,p_m,d \right)$
\begin{equation*}
  \begin{aligned}
    \left|\tilde{K}_{\epsilon,\delta}\left( x_{i,r},x_{j,s} \right)\right| &\leq \left\|K\right\|_{\infty},\quad \left|\mathbb{E}_2\left[\tilde{K}_{\epsilon,\delta}\left( x_{i,r},\cdot \right)\right]\right|\leq C\delta^{\frac{d-1}{2}},\\
    &\left|\mathbb{E}_1\mathbb{E}_2 \left[ \tilde{K}_{\epsilon,\delta}\left( x_{i,r},\cdot \right)\right]\right|\leq C\epsilon^{\frac{d}{2}}\delta^{\frac{d-1}{2}},
  \end{aligned}
\end{equation*}
and by direct computation
\begin{equation*}
  \mathbb{E}_2^{\xi_j}\left[\tilde{K}_{\epsilon,\delta}\left( x_{i,r},\cdot \right)\right]^2=O \left( \delta^{\frac{d-1}{2}} \right),\quad \mathbb{E}_1\left[\mathbb{E}_2\tilde{K}_{\epsilon,\delta}\left( x_{i,r},\cdot \right)\right]^2=O \left( \epsilon^{\frac{d}{2}}\delta^{d-1} \right),
\end{equation*}
Lemma~\ref{lem:large_deviation_bound} and $\beta=O \left( \epsilon^2+\epsilon\delta+\delta^2 \right)$ gives
\begin{align*}
  \begin{aligned}
    q \left( N_B,N_F,\beta \right)&\leq N_B\exp \left\{ -\frac{\left(1-\theta\right)^2N_F\epsilon^d\delta^{d-1}\beta^2}{2C_1\delta^{\frac{d-1}{2}}} \right\}+\exp \left\{ -\frac{\theta^2N_B\epsilon^d\delta^{d-1}\beta^2}{2C_1\epsilon^{\frac{d}{2}}\delta^{d-1}} \right\}\\
    &=N_B\exp \left\{ -\frac{\left(1-\theta\right)^2N_F\epsilon^d\delta^{\frac{d-1}{2}}\beta^2}{2C_1} \right\}+\exp \left\{ -\frac{\theta^2N_B\epsilon^{\frac{d}{2}}\beta^2}{2C_1} \right\}
  \end{aligned}
\end{align*}
for $C_1>0$ some constant. A simple union bound gives
\begin{equation}
\label{eq:density_uniform_bound}
  \begin{aligned}
    &\mathbb{P}\left(\bigcup_{j,s}\left\{\left|\frac{1}{N_BN_F}\hat{p}_{\epsilon,\delta}\left( x_{j,s} \right)-\tilde{p}_{\epsilon,\delta}\left( x_{j,s} \right) \right|>\epsilon^{\frac{d}{2}}\delta^{\frac{d-1}{2}}\beta\right\}\right)\\
    &\leq N_BN_F \left[ N_B\exp \left\{ -\frac{\left(1-\theta\right)^2N_F\epsilon^d\delta^{\frac{d-1}{2}}\beta^2}{2C_1} \right\}+\exp \left\{ -\frac{\theta^2N_B\epsilon^{\frac{d}{2}}\beta^2}{2C_1} \right\} \right]\\
    &=N_B \left( N_B+1 \right)N_F\exp \left\{ -\frac{\theta_{*}^2N_B\epsilon^{\frac{d}{2}}\beta^2}{2C_1} \right\}\quad\textrm{setting $\theta=\theta_{*}$ as in \eqref{eq:seraching_for_theta_star}.}
  \end{aligned}
\end{equation}
We are interested in seeing how this bound compares with the bound in \eqref{eq:bound_with_theta_star}. As $N_B,N_F\rightarrow\infty$, as long as \eqref{eq:balance_sampling_number} holds,
\begin{align*}
  &\frac{\displaystyle N_B \left( N_B+1 \right)N_F \exp \left\{ -\frac{\theta_{*}^2N_B\epsilon^{\frac{d}{2}}\beta^2}{2C_1} \right\}}{\displaystyle \left( N_B+1 \right)\exp \left\{ -\frac{\theta_{*}^2N_B\epsilon^{\frac{d}{2}}\beta^2}{C \left( \epsilon+\delta \right)} \right\}}\\
  =&N_B N_F\exp \left\{ -\theta_{*}^2N_B\epsilon^{\frac{d}{2}}\beta^2\left[\frac{1}{2C_1}-\frac{1}{C \left( \epsilon+\delta \right)}\right] \right\}\longrightarrow\infty\quad\textrm{for small $\epsilon,\delta$},
\end{align*}
thus the bound in \eqref{eq:bound_with_theta_star} is asymptotically negligible compared to the bound in \eqref{eq:density_uniform_bound}. This means that when $\alpha\neq 0$ the density estimation in general slows down the convergence rate by a factor $\left( \epsilon+\delta \right)^{\frac{1}{2}}$, which is consistent with the conclusion for standard diffusion maps on manifolds~\cite{HeinAudibertVonLuxburg2007,SingerWu2012VDM}. Therefore, for probability at least
\begin{equation*}
  1-N_B \left( N_B+1 \right)N_F\exp \left\{ -\frac{\theta_{*}^2N_B\epsilon^{\frac{d}{2}}\beta^2}{2C_1} \right\}
\end{equation*}
we have
\begin{equation*}
  \left|\frac{\displaystyle\sum_{j=1}^{N_B}\sum_{s=1}^{N_F}\frac{\displaystyle\hat{K}_{\epsilon,\delta}\left( x_{i,r},x_{j,s} \right)f \left( x_{j,s} \right)}{\tilde{p}^{\alpha}_{\epsilon,\delta}\left( x_{i,r} \right)\tilde{p}^{\alpha}_{\epsilon,\delta}\left( x_{j,s} \right)}}{\displaystyle \sum_{j=1}^{N_B}\sum_{s=1}^{N_F}\frac{\hat{K}_{\epsilon,\delta}\left( x_{i,r},x_{j,s} \right)\tilde{p}^{-\alpha}_{\epsilon,\delta}\left( x_{j,s} \right)}{\tilde{p}^{\alpha}_{\epsilon,\delta}\left( x_{i,r} \right)\tilde{p}^{\alpha}_{\epsilon,\delta}\left( x_{j,s} \right)}}-\tilde{H}^{\left(\alpha\right)}_{\epsilon,\delta}f \left( x_{i,r} \right)\right|\leq \beta
\end{equation*}
as well as
\begin{equation*}
  \left|\frac{1}{N_BN_F}\hat{p}_{\epsilon,\delta}\left( x_{j,s} \right)-\tilde{p}_{\epsilon,\delta}\left( x_{j,s} \right) \right|\leq \epsilon^{\frac{d}{2}}\delta^{\frac{d-1}{2}}\beta\quad\textrm{for all }1\leq j\leq N_B, 1\leq s\leq N_F.
\end{equation*}
Note that by our assumption
\begin{equation*}
  0<p_m\leq p \left( x,v \right)\leq p_M<\infty\quad\textrm{for all }\left( x,v \right)\in UT\!M
\end{equation*}
there exist constants $C_1,C_2$ such that
\begin{equation*}
  0<C_1 < \epsilon^{-\frac{d}{2}}\delta^{-\frac{d-1}{2}}\tilde{p}_{\epsilon,\delta}\left( x_{j,s} \right) < C_2<\infty.
\end{equation*}
For sufficiently small $\beta$, we also have
\begin{equation*}
  0<C_1 < \frac{1}{N_BN_F}\epsilon^{-\frac{d}{2}}\delta^{-\frac{d-1}{2}}\hat{p}_{\epsilon,\delta}\left( x_{j,s} \right) < C_2<\infty.
\end{equation*}
Thus
\begin{equation*}
  \left| N_BN_F\hat{p}_{\epsilon,\delta}^{-1}\left( x_{j,s} \right)-\tilde{p}_{\epsilon,\delta}^{-1}\left( x_{j,s} \right) \right|\leq \epsilon^{\frac{d}{2}}\delta^{\frac{d-1}{2}}\beta\cdot \frac{1}{C_1^2\epsilon^d\delta^{d-1}}=\frac{\beta}{C_1^2\epsilon^{\frac{d}{2}}\delta^{\frac{d-1}{2}}}.
\end{equation*}
and $\left( A \right)$, $\left( B \right)$ can be bounded as
\begin{align*}
  &\left|\left( A \right)\right|\leq C_2^{\left(\alpha\right)}\epsilon^{\frac{\alpha d}{2}}\delta^{\frac{\alpha \left( d-1 \right)}{2}}\left\|f\right\|_{\infty}\cdot \alpha \left(\frac{2}{C_2\epsilon^{\frac{d}{2}}\delta^{\frac{d-1}{2}}}\right)^{\alpha-1}\frac{\beta}{C_1^2\epsilon^{\frac{d}{2}}\delta^{\frac{d-1}{2}}}=\frac{2^{\alpha-1}\alpha C_2 \left\|f\right\|_{\infty}}{C_1^2}\beta,\\
  &\left|\left( B \right)\right|\leq \frac{C_2^{2\alpha}\epsilon^{\alpha d}\delta^{\alpha \left( d-1 \right)}}{C_1^{\left(\alpha\right)}\epsilon^{\frac{\alpha d}{2}}\delta^{\frac{\alpha \left( d-1 \right)}{2}}}\left\|f\right\|_{\infty}\cdot \alpha \left(\frac{2}{C_2\epsilon^{\frac{d}{2}}\delta^{\frac{d-1}{2}}}\right)^{\alpha-1}\frac{\beta}{C_1^2\epsilon^{\frac{d}{2}}\delta^{\frac{d-1}{2}}}=\frac{2^{\alpha-1}\alpha C_2^{\alpha+1} \left\|f\right\|_{\infty}}{C_1^{\alpha+2}}\beta.
\end{align*}
Since $C_1,C_2$ only depend on the kernel function $K$, the dimension $d$, and $p_m,p_M$, these bounds ensures that
\begin{equation*}
  \left| \hat{H}_{\epsilon,\delta}^{\left(\alpha\right)}f \left( x_{i,r} \right)-\tilde{H}_{\epsilon,\delta}^{\left(\alpha\right)}f \left( x_{i,r} \right) \right| < C\beta
\end{equation*}
with probability at least
\begin{equation*}
  1-N_B \left( N_B+1 \right)N_F\exp \left\{ -\frac{\theta_{*}^2N_B\epsilon^{\frac{d}{2}}\beta^2}{2C_1} \right\},
\end{equation*}
where constants $C,C_1$ only depend on the kernel function $K$, the dimension $d$, and $p_m,p_M$. This establishes the conclusion for all $\alpha\in \left[ 0,1 \right]$.
\end{proof}

\subsubsection{Sampling from Empirical Tangent Spaces}
\label{sec:proor-theor-noise}

The following two lemmas from \cite{SingerWu2012VDM} provide estimates for the error of approximating parallel-transports from local PCA. We adapted these lemmas to our notation; note that the statements are more compact than their original form since we assume $M$ is closed.

\begin{lemma}
\label{lem:local_PCA}
  Suppose $K_{\mathrm{PCA}}\in C^2 \left( \left[ 0,1 \right] \right)$. If $\epsilon_{\mathrm{PCA}}=O \left( N_B^{-\frac{2}{d+2}} \right)$, then, with high probability, the columns of the $D\times d$ matrix $O_i$ determined by local PCA form an orthonormal basis to a $d$-dimensional subspace of $\mathbb{R}^D$ that deviates from $\iota_{*}T_{x_i}M$ by $O \left( \epsilon_{\mathrm{PCA}}^{\frac{3}{2}} \right)$, in the following sense:
  \begin{equation}
    \label{eq:localPCAerror}
    \min_{O\in O \left( d \right)}\|O_i^{\top}\Theta_i-O\|_{\mathrm{HS}}=O \left( \epsilon_{\mathrm{PCA}}^{\frac{3}{2}} \right)=O \left( N_B^{-\frac{3}{d+2}} \right),
  \end{equation}
where $\Theta_i$ is a $D\times d$ matrix whose columns form an orthonormal basis to $\iota_{*}T_{x_i}M$. Let the minimizer if \eqref{eq:localPCAerror} be
\begin{equation}
  \label{eq:localPCAminimizer}
  \hat{O}_i=\argmin_{O\in O \left( d \right)}\|O_i^{\top}\Theta_i-O\|_{\mathrm{F}},
\end{equation}
and denote by $Q_i$ the $D\times d$ matrix
\begin{equation}
  \label{eq:localPCAbasis}
  Q_i:=\Theta_i\hat{O}_i^{\top},
\end{equation}
The columns of $Q_i$ form an orthonormal basis to $\iota_{*}T_{x_i}M$, and
\begin{equation}
  \label{eq:localPCAbasiserror}
  \|O_i-Q_i\|_{\mathrm{F}}=O \left( \epsilon_{\mathrm{PCA}} \right),
\end{equation}
where $\left\|\cdot\right\|_{\mathrm{F}}$ is the matrix Frobenius norm.
\end{lemma}
\begin{proof}
  See \cite[Lemma B.1]{SingerWu2012VDM}.
\end{proof}

\begin{lemma}
\label{lem:approximate_parallel_transport}
  Consider points $x_i,x_j\in M$ such that the geodesic distance between them is $O \left( \epsilon^{\frac{1}{2}} \right)$. For $\epsilon_{\mathrm{PCA}}=O \left( N_B^{-\frac{2}{d+2}} \right)$, with high probability, $O_{ij}$ approximates $P_{x_i,x_j}$ in the following sense:
  \begin{equation}
    \label{eq:approximate_parallel_transport}
    O_{ij}\bar{X}_j=\left( \langle\iota_{*}P_{x_i,x_j}X \left( x_j \right),u_l \left( x_i \right)\rangle \right)_{l=1}^d+O \left( \epsilon_{\mathrm{PCA}}^{\frac{1}{2}}+\epsilon^{\frac{3}{2}} \right),\quad\textrm{for all }X\in \Gamma \left( M,TM \right),
  \end{equation}
where $\left\{ u_l \left( x_i \right) \right\}_{l=1}^d$ is an orthonormal set determined by local PCA, and
\begin{equation*}
  \bar{X}_i\equiv \left( \langle\iota_{*}X \left( x_i \right),u_l \left( x_i \right)\rangle \right)_{l=1}^d\in\mathbb{R}^d.
\end{equation*}
\end{lemma}
\begin{proof}
  See \cite[Theorem B.2]{SingerWu2012VDM}.
\end{proof}

\begin{proof}[Proof of Theorem{\rm ~\ref{thm:utm_finite_sampling_noise}}]
By Definition~\ref{defn:utm_finite_sampling_noise}~\eqref{item:40},
\begin{equation*}
  O_{ji}c_{i,r}=\frac{O_{ji}B_i^{\top}\overline{\tau}_{i,r}}{\left\|B_i^{\top}\overline{\tau}_{i,r}\right\|}.
\end{equation*}
By Lemma~\ref{lem:approximate_parallel_transport},
\begin{equation*}
  O_{ji}B_i^{\top}\overline{\tau}_{i,r}=B_j^{\top}\left( P_{\xi_j,\xi_i}\overline{\tau}_{i,r} \right)+O \left( \epsilon_{\mathrm{PCA}}^{\frac{1}{2}}+\epsilon^{\frac{3}{2}} \right),
\end{equation*}
thus
\begin{align*}
  \frac{O_{ji}B_i^{\top}\overline{\tau}_{i,r}}{\left\|B_i^{\top}\overline{\tau}_{i,r}\right\|}=\frac{B_j^{\top}\left( P_{\xi_j,\xi_i}\overline{\tau}_{i,r} \right)}{\left\|B_i^{\top}\overline{\tau}_{i,r}\right\|}+O \left( \epsilon_{\mathrm{PCA}}^{\frac{1}{2}}+\epsilon^{\frac{3}{2}} \right),
\end{align*}
where we used $\left\|B_j^{\top}\left( P_{\xi_j,\xi_i}\overline{\tau}_{i,r} \right)\right\|\leq \left\|P_{\xi_j,\xi_i}\overline{\tau}_{i,r}\right\|=1$ and
\begin{align*}
  \left| \left\|B_i^{\top}\overline{\tau}_{i,r}\right\|_{\mathrm{F}}-1 \right|&=\left| \left\|B_i^{\top}\overline{\tau}_{i,r}\right\|_{\mathrm{F}}-\left\|Q_i^{\top}\overline{\tau}_{i,r}\right\|_{\mathrm{F}}\right|\leq \left\|B_i^{\top}\overline{\tau}_{i,r}-Q_i^{\top}\overline{\tau}_{i,r}\right\|_{\mathrm{F}}\\
  &\leq \left\|B_i^{\top}-Q_i^{\top}\right\|_{\mathrm{F}}=O \left( \epsilon_{\mathrm{PCA}} \right).
\end{align*}
Thus
\begin{align*}
  O_{ji}c_{i,r}-c_{j,s}&=\frac{B_j^{\top}\left( P_{\xi_j,\xi_i}\overline{\tau}_{i,r} \right)}{\left\|B_i^{\top}\overline{\tau}_{i,r}\right\|}+O \left( \epsilon_{\mathrm{PCA}}^{\frac{1}{2}}+\epsilon^{\frac{3}{2}} \right)-\frac{B_j^{\top}\overline{\tau}_{j,s}}{\left\|B_j^{\top}\overline{\tau}_{j,s}\right\|}\\
  &=P_{\xi_j,\xi_i}\overline{\tau}_{i,r}-\overline{\tau}_{j,s}+O \left( \epsilon_{\mathrm{PCA}}^{\frac{1}{2}}+\epsilon^{\frac{3}{2}} \right),
\end{align*}
\begin{align*}
  \left|\left\|O_{ji}c_{i,r}-c_{j,s}\right\|^2-\left\|P_{\xi_j,\xi_i}\overline{\tau}_{i,r}-\overline{\tau}_{j,s}\right\|^2\right|=O \left( \epsilon_{\mathrm{PCA}}^{\frac{1}{2}}+\epsilon^{\frac{3}{2}} \right),
\end{align*}
and
\begin{align*}
  &K\left( \frac{\left\|\xi_i-\xi_j\right\|^2}{\epsilon},\frac{\left\|O_{ji}c_{i,r}-c_{j,s}\right\|^2}{\delta} \right)=K \left(\frac{\left\|\xi_i-\xi_j\right\|^2}{\epsilon},\frac{\left\|P_{\xi_j,\xi_i}\overline{\tau}_{i,r}-\overline{\tau}_{j,s}\right\|^2}{\delta} \right)\\
  &\qquad\qquad\qquad\qquad+\partial_2K\left(\frac{\left\|\xi_i-\xi_j\right\|^2}{\epsilon},\frac{\left\|P_{\xi_j,\xi_i}\overline{\tau}_{i,r}-\overline{\tau}_{j,s}\right\|^2}{\delta} \right)\cdot \frac{O \left( \epsilon_{\mathrm{PCA}}^{\frac{1}{2}}+\epsilon^{\frac{3}{2}} \right)}{\delta}.
\end{align*}
Thus for any function $g\in C^{\infty}\left( UT\!M \right)$ we have
\begin{align*}
  &\int_{UT\!M}\mathscr{K}_{\epsilon,\delta}\left( \overline{\tau}_{i,r},\eta \right)g \left( \eta \right)d\Theta \left( \eta \right)\\
  &=\int_{UT\!M}\tilde{K}_{\epsilon,\delta}\left( \overline{\tau}_{i,r},\eta \right)g \left( \eta \right)d\Theta \left( \eta \right)+\epsilon^{\frac{d}{2}}\delta^{\frac{d-1}{2}-1}O \left( \epsilon_{\mathrm{PCA}}^{\frac{1}{2}}+\epsilon^{\frac{3}{2}} \right).
\end{align*}
Following the notation used in the proof of Theorem~\ref{thm:utm_finite_sampling_noiseless}, by the law of large numbers
\begin{align*}
  &\lim_{N_B\rightarrow\infty}\lim_{N_F\rightarrow\infty}\frac{1}{N_BN_F}\hat{q}_{\epsilon,\delta}\left( \overline{\tau}_{i,r} \right)=\mathbb{E}_1\mathbb{E}_2\left[\mathscr{K}_{\epsilon,\delta}\left(\overline{\tau}_{i,r},\cdot \right)\right]\\
  &=\mathbb{E}_1\mathbb{E}_2\left[\tilde{K}_{\epsilon,\delta}\left(\overline{\tau}_{i,r},\cdot \right)\right]+\epsilon^{\frac{d}{2}}\delta^{\frac{d-1}{2}-1}O \left( \epsilon_{\mathrm{PCA}}^{\frac{1}{2}}+\epsilon^{\frac{3}{2}} \right),
\end{align*}
hence we expect $\mathscr{H}_{\epsilon,\delta}^{\left(\alpha\right)}f \left( \overline{\tau}_{i,r} \right)$ to converge to
\begin{align*}
    &\tilde{H}_{\epsilon,\delta}^{\left(\alpha\right)}f \left( \overline{\tau}_{i,r} \right)+O \left( \delta^{-1}\left( \epsilon_{\mathrm{PCA}}^{\frac{1}{2}}+\epsilon^{\frac{3}{2}} \right) \right)\\
    &= f \left( \overline{\tau}_{i,r} \right)+\epsilon \frac{m_{21}}{2m_0}\left[\frac{\Delta_{UT\!M}^H\left[fp^{1-\alpha}\right]\left( \overline{\tau}_{i,r} \right)}{p^{1-\alpha}\left( \overline{\tau}_{i,r} \right)}-f \left( \overline{\tau}_{i,r} \right) \frac{\Delta_{UT\!M}^Hp^{1-\alpha}\left( \overline{\tau}_{i,r} \right)}{p^{1-\alpha}\left( \overline{\tau}_{i,r} \right)}  \right]\\
    &\quad+\delta \frac{m_{22}}{2m_0}\left[\frac{\Delta_{UT\!M}^V\left[fp^{1-\alpha}\right]\left( \overline{\tau}_{i,r} \right)}{p^{1-\alpha}\left( \overline{\tau}_{i,r} \right)}-f \left( \overline{\tau}_{i,r} \right) \frac{\Delta_{UT\!M}^Vp^{1-\alpha}\left( \overline{\tau}_{i,r} \right)}{p^{1-\alpha}\left( \overline{\tau}_{i,r} \right)}  \right]\\
    &\quad+O \left( \epsilon^2+\epsilon\delta+\delta^2\right)+O \left( \delta^{-1}\left( \epsilon_{\mathrm{PCA}}^{\frac{1}{2}}+\epsilon^{\frac{3}{2}} \right) \right).
\end{align*}
In fact, noting that
\begin{align*}
  \frac{1}{\epsilon^{\frac{d}{2}}\delta^{\frac{d-1}{2}}N_BN_F}\hat{q}_{\epsilon,\delta}\left( \overline{\tau}_{i,r} \right)=\frac{1}{\epsilon^{\frac{d}{2}}\delta^{\frac{d-1}{2}}N_BN_F}\hat{p}_{\epsilon,\delta}\left( \overline{\tau}_{i,r},\overline{\tau}_{j,s} \right)+\frac{O \left( \delta^{-1}\left( \epsilon_{\mathrm{PCA}}^{\frac{1}{2}}+\epsilon^{\frac{3}{2}} \right) \right)}{\epsilon^{\frac{d}{2}}\delta^{\frac{d-1}{2}}},
\end{align*}
we have
\begin{align*}
  &\epsilon^{\alpha d}\delta^{\alpha\left(d-1\right)}N_B^{2\alpha}N_F^{2\alpha}\mathscr{K}_{\epsilon,\delta}^{\left(\alpha\right)} \left( \overline{\tau}_{i,r}, \overline{\tau}_{j,s} \right)\\
  &=\epsilon^{\alpha d}\delta^{\alpha\left(d-1\right)}N_B^{2\alpha}N_F^{2\alpha}K_{\epsilon,\delta}^{\left(\alpha\right)}\left( \overline{\tau}_{i,r},\overline{\tau}_{j,s} \right)+O \left( \delta^{-1}\left( \epsilon_{\mathrm{PCA}}^{\frac{1}{2}}+\epsilon^{\frac{3}{2}} \right) \right).
\end{align*}
Consequently,
\begin{align*}
  &\mathscr{H}_{\epsilon,\delta}^{\left(\alpha\right)}f \left( \overline{\tau}_{i,r} \right)=\frac{\displaystyle \sum_{j=1}^{N_B}\sum_{s=1}^{N_F}\tilde{K}_{\epsilon,\delta}^{\left(\alpha\right)} \left( \overline{\tau}_{i,r}, \overline{\tau}_{j,s}\right)f \left( \overline{\tau}_{j,s} \right)}{\displaystyle \sum_{j=1}^{N_B}\sum_{s=1}^{N_F}\tilde{K}_{\epsilon,\delta}^{\left(\alpha\right)} \left( \overline{\tau}_{i,r}, \overline{\tau}_{j,s}\right)}+O \left( \delta^{-1}\left( \epsilon_{\mathrm{PCA}}^{\frac{1}{2}}+\epsilon^{\frac{3}{2}} \right) \right)\\
  &=\hat{H}_{\epsilon,\delta}^{\left(\alpha\right)}f \left( \overline{\tau}_{i,r} \right)+O \left( \delta^{-1}\left( \epsilon_{\mathrm{PCA}}^{\frac{1}{2}}+\epsilon^{\frac{3}{2}} \right) \right).
\end{align*}
Under the assumption that
\begin{align*}
  \delta^{-1}\left( \epsilon_{\mathrm{PCA}}^{\frac{1}{2}}+\epsilon^{\frac{3}{2}} \right)\longrightarrow 0\quad \textrm{as $\epsilon\rightarrow0$,}
\end{align*}
we can apply Theorem~\ref{thm:utm_finite_sampling_noiseless}. This completes the proof of Theorem~\ref{thm:utm_finite_sampling_noise}.
\end{proof}

\bibliographystyle{alpha}
\bibliography{ref}

\end{document}